\documentclass[reqno,11pt]{amsart}

\usepackage[left=1in,right=1in,top=1in,bottom=1in]{geometry}
\usepackage[T1]{fontenc}
\usepackage{lmodern}
\usepackage{microtype}
\microtypesetup{expansion=false}
\usepackage{enumitem}
\usepackage{amsmath}
\usepackage{amssymb}
\usepackage{mathtools}
\usepackage{aliascnt}
\usepackage{graphicx}
\usepackage{needspace}
\usepackage[colorlinks=true,citecolor=blue,linkcolor=blue,urlcolor=blue]{hyperref}
\usepackage[nameinlink,noabbrev]{cleveref}

\numberwithin{equation}{section}

\theoremstyle{plain}
\newtheorem{thm}{Theorem}[section]
\newaliascnt{cor}{thm}

\aliascntresetthe{cor}
\newaliascnt{con}{thm}

\aliascntresetthe{con}
\newaliascnt{lem}{thm}
\newtheorem{lem}[lem]{Lemma}
\aliascntresetthe{lem}
\newaliascnt{prop}{thm}
\newtheorem{prop}[prop]{Proposition}
\aliascntresetthe{prop}
\newaliascnt{assum}{thm}
\newtheorem{assum}[assum]{Assumption}
\aliascntresetthe{assum}

\theoremstyle{definition}
\newaliascnt{defn}{thm}
\newtheorem{defn}[defn]{Definition}
\aliascntresetthe{defn}
\newaliascnt{example}{thm}

\aliascntresetthe{example}
\newaliascnt{q}{thm}

\aliascntresetthe{q}

\theoremstyle{remark}
\newaliascnt{rem}{thm}
\newtheorem{rem}[rem]{Remark}
\aliascntresetthe{rem}

\crefname{thm}{theorem}{theorems}
\Crefname{thm}{Theorem}{Theorems}
\crefname{cor}{corollary}{corollaries}
\Crefname{cor}{Corollary}{Corollaries}
\crefname{con}{conjecture}{conjectures}
\Crefname{con}{Conjecture}{Conjectures}
\crefname{lem}{lemma}{lemmas}
\Crefname{lem}{Lemma}{Lemmas}
\crefname{prop}{proposition}{propositions}
\Crefname{prop}{Proposition}{Propositions}
\crefname{assum}{assumption}{assumptions}
\Crefname{assum}{Assumption}{Assumptions}
\crefname{defn}{definition}{definitions}
\Crefname{defn}{Definition}{Definitions}
\crefname{example}{example}{examples}
\Crefname{example}{Example}{Examples}
\crefname{q}{question}{questions}
\Crefname{q}{Question}{Questions}
\crefname{rem}{remark}{remarks}
\Crefname{rem}{Remark}{Remarks}

\renewcommand{\theenumi}{\roman{enumi}}
\newlist{maincomments}{enumerate}{1}
\setlist[maincomments]{
  label=\textnormal{(\roman*)},
  ref=\roman*,
  leftmargin=*
}
\setlist[enumerate,1]{
  label=\textnormal{(\theenumi)},
  ref=\theenumi,
  leftmargin=*
}

\DeclarePairedDelimiter{\abs}{\lvert}{\rvert}

\DeclarePairedDelimiterX{\inner}[2]{\langle}{\rangle}{#1,#2}
\newcommand{\dd}{\,\mathrm{d}}
\newcommand{\R}{\mathbb R}
\newcommand{\T}{\mathbb T}
\newcommand{\eps}{\varepsilon}
\newcommand{\cH}{\mathcal H}

\newcommand{\dist}{\operatorname{dist}}
\newcommand{\diver}{\operatorname{div}}
\newcommand{\tr}{\operatorname{tr}}
\newcommand{\Id}{\mathrm I}

\hypersetup{
  pdftitle={Parabolic Bethuel compactness and a two-mobility curvature balance},
  pdfauthor={Anonymous},
  pdfkeywords={vectorial Allen--Cahn flow; multi-well potential; rectifiable varifold; two-mobility curvature balance; singular perturbation}
}

\title[Asymptotics for two-dimensional parabolic vectorial Allen-Cahn systems]{Asymptotics for two-dimensional parabolic vectorial Allen-Cahn systems}
\author{Zhiyuan Dai}
\address{School of Mathematical Sciences, Zhejiang University, Hangzhou 310058, China}
\email{12635038@zju.edu.cn}

\author{Haotong Fu}
\address{School of Mathematical Sciences, Peking University, Beijing 100871, China}
\email{2301110012@pku.edu.cn}

\author{Huaijie Wang}
\address{School of Mathematical Sciences, Peking University, Beijing 100871, China}
\email{huaijie\_wang@163.com}

\author{Wei Wang}
\address{School of Mathematical Sciences, Peking University, Beijing 100871, China}
\email{wwmath166@outlook.com}
\email{2201110024@stu.pku.edu.cn}
\date{September 23, 2026}

\begin{document}

\begin{abstract}
We study the parabolic vectorial Allen--Cahn equation in two space
dimensions for a fixed smooth potential with finitely many non-degenerate
wells.  Our result provides a parabolic extension of Bethuel's elliptic
compactness theorem: at almost every positive time, the diffuse energy,
potential energy, and gradient tensor converge to measures concentrated on a
countably \(1\)-rectifiable interface and satisfy all the
Bethuel-type
relations.  In the scalar case, the tangential
defect vanishes and the usual mean-curvature-flow structure is recovered.
For general vectorial potentials, the limiting stress and dissipation yield
a two-mobility system: the normal energy flux balances weighted curvature,
while a tangential internal mobility transports residual tangential energy.
Thus the result extends the Allen--Cahn-to-Brakke framework of Ilmanen from
scalar to vector-valued systems, with an additional non-negative
dissipation defect in the localized energy inequality.
\end{abstract}


\maketitle

\section{Introduction}

Let \(V:\R^k\to[0,\infty)\) be a smooth potential whose zero set consists of
finitely many non-degenerate wells,
\[
 \Sigma:=V^{-1}(0)=\{\sigma_1,\ldots,\sigma_q\}.
\]
For \(\eps>0\), the associated Allen--Cahn energy on a planar domain
\(\Omega\) is
\[
 E_\eps(u;\Omega)
 :=
 \int_\Omega
 \left(
 \frac{\eps}{2}|\nabla u|^2+\frac1\eps V(u)
 \right)\dd x.
\]
Its critical points solve
\[
 -\eps\Delta u+\frac1\eps DV(u)=0.
\]
In the scalar two-well case, the interfacce is asymptotically
one-dimensional and its gradient and potential energies are asymptotically
equipartitioned.  Equivalently, the discrepancy measure
\[
 \xi_\eps
 :=
 \left(
 \frac{\eps}{2}|\nabla u|^2-\frac1\eps V(u)
 \right)\dd x
\]
vanishes in the sharp-interface limit.  This equipartition is one of the
main structural reasons why the scalar diffuse stress converges to the
standard tangential projection on the limiting interface.

The variational origins of this picture go back to the phase-field model of
Allen and Cahn \cite{AllenCahn1979} and the \(\Gamma\)-convergence result of
Modica and Mortola \cite{ModicaMortola1977}; see also Modica's
minimal-interface criterion and gradient estimate
\cite{Modica1985,Modica1987}.  The extension from two wells to finitely many
wells, and the corresponding compactness and minimal-partition theory, was
developed by Baldo, Fonseca--Tartar, and Kohn--Sternberg
\cite{Baldo1990,FonsecaTartar1989,KohnSternberg1989}.

Bethuel proved a substantially more general elliptic compactness theorem for
bounded-energy critical points of vectorial Allen--Cahn energies in two
space dimensions \cite{Bethuel2025}.  More precisely, after passing to a
subsequence, the diffuse full-energy, potential-energy, and gradient-tensor
measures
\[
 \nu_\eps
 :=
 \left(
 \frac{\eps}{2}|\nabla u_\eps|^2+\frac1\eps V(u_\eps)
 \right)\dd x,
 \qquad
 \zeta_\eps:=\frac1\eps V(u_\eps)\dd x,
\]
and
\[
 \mu_{\eps,ij}
 :=
 \eps\,\partial_i u_\eps\cdot\partial_j u_\eps\,\dd x
\]
concentrate on a countably \(1\)-rectifiable set \(S\).  Writing
\[
 \nu=e\,\mathcal H^1\llcorner S,
 \qquad
 \zeta=\Theta\,\mathcal H^1\llcorner S,
 \qquad
 \mu_{ij}=m_{ij}\,\mathcal H^1\llcorner S,
\]
and denoting by \((\tau,n)\) an approximate tangent--normal frame of \(S\),
Bethuel's relations take the form
\[
 e=m_{nn},
 \qquad
 m_{n\tau}=0,
 \qquad
 2\Theta=m_{nn}-m_{\tau\tau}.
\]
Thus the limiting stress tensor is rank one:
\[
 T:=\nu I-\mu
 =2\Theta\,\tau\otimes\tau\,
 \mathcal H^1\llcorner S.
\]
In contrast with the scalar theory of Hutchinson and Tonegawa
\cite{HutchinsonTonegawa2000}, these identities do not imply equipartition.
Indeed,
\[
 \rho:=e-2\Theta=m_{\tau\tau}\geq0
\]
may be non-zero.  Bethuel's periodic pseudo-profiles show that this
possibility is genuine: oscillations internal to a vector-valued transition
layer may carry a non-trivial amount of tangential gradient energy even
after the layer itself collapses onto a one-dimensional set.  We refer to
\(\rho\) as the \emph{tangential microstructure density}.

The purpose of the present paper is to establish a parabolic counterpart of
Bethuel's compactness theorem.  We consider smooth solutions
\[
 u_\eps:\T^2\times(0,T)\longrightarrow\R^k
\]
of the \(L^2\)-gradient flow of \(E_\eps\),
\begin{equation}\label{eq:intro-flow}
 \partial_tu_\eps
 =\Delta u_\eps-\frac1{\eps^2}DV(u_\eps).
\end{equation}
The corresponding \emph{diffuse energy density} and \emph {dissipation density} are
\[
 e_\eps
 :=
 \frac{\eps}{2}|\nabla u_\eps|^2+\frac1\eps V(u_\eps),
 \qquad
 \alpha_\eps:=\eps|\partial_tu_\eps|^2,
\]
and they satisfy the exact global energy identity
\[
 E_\eps(u_\eps(\cdot,t_2))
 +\int_{t_1}^{t_2}\int_{\T^2}\alpha_\eps\,\dd x\dd t
 =E_\eps(u_\eps(\cdot,t_1)).
\]
The diffuse momentum and stress are defined by
\[
 q_{\eps,i}
 :=
 \eps\,\partial_tu_\eps\cdot\partial_i u_\eps,
 \qquad
 T_{\eps,ij}
 :=
 e_\eps\delta_{ij}
 -\eps\,\partial_i u_\eps\cdot\partial_j u_\eps.
\]
They obey the exact distributional identities
\[
 \operatorname{div}_xT_\eps=-q_\eps,
 \qquad
 \partial_te_\eps=\operatorname{div}_xq_\eps-\alpha_\eps,
\]
where\;$T_{\varepsilon}=(T_{\varepsilon,ij})_{i,j},q_{\varepsilon}=(q_{\varepsilon,i})_{i}.$

These identities provide the basic link between the spatial geometry of the
diffuse interfaces, their energy flux, and the dissipation of the flow.

For scalar Allen--Cahn equations, the formal and smooth-interface limits
were established by Bronsard--Kohn and Chen
\cite{BronsardKohn1991,Chen1992}, while the level-set formulation beyond
singularities was developed by Evans--Soner--Souganidis
\cite{EvansSonerSouganidis1992}.  In the varifold framework, Brakke's
monograph \cite{Brakke1978} and Ilmanen's elliptic-regularization work
\cite{Ilmanen1994} provide the geometric background.  Ilmanen's convergence
theorem \cite{Ilmanen1993} gives the fundamental passage from diffuse
interface dynamics to Brakke's weak motion by mean curvature.  Starting
from a uniform energy bound, Ilmanen associates a varifold to the diffuse
energy and uses the scalar discrepancy estimate, together with a localized
monotonicity formula and the energy-dissipation identity, to obtain
rectifiability of the limiting spatial measures and the Brakke inequality.
The discrepancy control asymptotically identifies the gradient and
potential parts of the energy.  Consequently, the limiting stress is the
standard tangential varifold stress, the diffuse momentum produces a
generalized normal velocity, and the limiting law is the classical
mean-curvature balance
\[
 \mathcal V=H_S,
\]
in the weak Brakke sense.  This argument is intrinsically tied to the
scalar order structure and the resulting discrepancy control.  For a
general fixed vectorial potential, Bethuel's examples show that the
analogous discrepancy need not vanish, so Ilmanen's scalar identification
cannot be applied directly.

A number of important results address vector-valued or multiphase
Allen--Cahn evolutions under additional hypotheses.  For stationary
vectorial systems, layered and multi-junction profiles were constructed in
representative works by Bronsard--Gui--Schatzman and Alikakos--Betel\'u--Chen
\cite{BronsardGuiSchatzman1996,AlikakosBeteluChen2006}; Hamiltonian
identities provide a complementary tool for junction laws
\cite{Gui2008}.  For the evolution problem, Laux and Simon, and more
recently Steinke, obtain weak multiphase mean-curvature evolutions under
assumptions that identify the limiting diffuse energy with the corresponding
interfacial energy
\cite{LauxSimon2018,Steinke2026}.  Fischer and Marveggio prove quantitative
convergence by a relative-entropy method as long as a sufficiently regular
calibrated multiphase flow exists \cite{FischerMarveggio2024}.  Laux and
Takasao obtain convergence to a multiphase Brakke flow for a specially
constructed, \(\eps\)-dependent weakly coupled potential
\cite{LauxTakasao2026}.  These theories capture important regimes, but they
do not address the unconditional sharp-interface compactness of
\eqref{eq:intro-flow} for an arbitrary fixed multi-well potential in
Bethuel's class, where a non-zero tangential density may persist.

Our limiting description therefore retains two distinct interfacial
densities.  At almost every time \(t\), the limiting full and potential
energy measures have the form
\[
 \nu_t=e_t\,\mathcal H^1\llcorner S_t,
 \qquad
 \zeta_t=\Theta_t\,\mathcal H^1\llcorner S_t,
\]
where \(S_t\) is countably \(1\)-rectifiable.  The full-energy density
\(e_t\) governs the transport of total diffuse energy, whereas the
potential-energy density \(\Theta_t\) determines the limiting spatial
stress and hence the weighted first variation.  Their difference
\begin{equation}\label{eq:intro-rho}
 \rho_t:=e_t-2\Theta_t
\end{equation}
is precisely the residual tangential gradient density furnished by
Bethuel's relations.

To formulate the limiting evolution, write the spacetime full-energy
measure as
\[
 \nu:=\nu_t\,\dd t
 =e_t\,\mathcal H^1\llcorner S_t\,\dd t,
\]
and let \(q\) be the weak limit of the diffuse momentum measures
\(q_\eps\,\dd x\dd t\).  Its time disintegration has the form
\[
 q=q_t\,\dd t,
 \qquad
 q_t=
 \bigl(q_{\tau,t}\tau_t+q_{n,t}n_t\bigr)
 \mathcal H^1\llcorner S_t.
\]
Here \(q_{n,t}\) and \(q_{\tau,t}\) denote the normal and tangential
components of the limiting energy momentum.  We define the normal
energy-flux field and the tangential internal-mobility field by
\[
 \mathcal V_t:=-\frac{q_{n,t}}{e_t}\,n_t,
 \qquad
 \mathcal W_t:=-\frac{q_{\tau,t}}{\rho_t}\,\tau_t,
\]
with the convention that \(\mathcal W_t=0\) on \(\{\rho_t=0\}\).  The sharp
action bound proved below forces \(q_{\tau,t}=0\) on this set, so the
definition is consistent.

Let \(H_{S_t}\) denote the normal component of the generalized curvature of
the weighted varifold associated with
\(\Theta_t\,\mathcal H^1\llcorner S_t\), and let
\(\nabla_{S_t}\Theta_t\) denote the corresponding tangential derivative of
the weight.  The two components of the limiting stress balance yield
\begin{equation}\label{eq:intro-two-mobility}
 e_t\mathcal V_t=2\Theta_tH_{S_t},
 \qquad
 \rho_t\mathcal W_t=2\nabla_{S_t}\Theta_t.
\end{equation}
The first identity is the normal curvature balance, while the second
describes tangential transport of the interfacial density carried by the
vectorial microstructure.  The total energy-flux density with respect to
\(\nu\) is consequently
\[
 -\frac{\dd q}{\dd\nu}
 =
 \mathcal V_t+\frac{\rho_t}{e_t}\mathcal W_t.
\]

When \(\rho_t=0\), Bethuel's relations reduce to the equipartition identity
\(e_t=2\Theta_t\); the sharp action bound eliminates the tangential flux,
and \eqref{eq:intro-two-mobility} becomes
\[
 \mathcal V_t=H_{S_t}.
\]
Thus the usual mean-curvature balance and the corresponding Brakke
inequality are recovered on the equipartition region.  When \(\rho_t>0\),
however, the limiting energy possesses an additional tangential transport
channel.  Without a further no-hidden-energy or rank-one hypothesis, neither
\(\mathcal V_t\) nor the total energy-flux field
\(-\dd q/\dd\nu\) is asserted to coincide with the velocity of the reduced
boundaries of the limiting phase partition.  The theorem instead gives an
unconditional evolution law for the full limiting energy, including the
collapsed or oscillatory structures that may be invisible in the reduced
phase map.

\subsection{Main results}
\begin{assum}\label{assum:potential}
The potential $V\in C^\infty(\R^k;[0,\infty))$ satisfies the following
conditions.
\begin{enumerate}
\item The vacuum set is finite:
\[
 \Sigma:=V^{-1}(0)=\{\sigma_1,\ldots,\sigma_q\},
 \quad q\geq2.
\]
\item The Hessian $D^2V(\sigma_a)$ is positive definite for every $a$.
\item There are $R_\infty,\alpha_\infty>0$ such that
\[
 y\cdot DV(y)\geq\alpha_\infty\abs{y}^2
 \quad\text{when }\abs{y}\geq R_\infty,
 \quad
 V(y)\to\infty
 \quad\text{as }\abs{y}\to\infty.
\]
\end{enumerate}
\end{assum}

For a solution of \eqref{eq:intro-flow}, set
\begin{equation}\label{eq:intro-measures}
\begin{aligned}
 \nu_\eps^t
 &:={}
 \left[
 \frac{\eps}{2}\abs{\nabla u_\eps(\cdot,t)}^2
 +\frac{1}{\eps}V(u_\eps(\cdot,t))
 \right]\dd x,\\
 \zeta_\eps^t
 &:={}
 \frac{1}{\eps}V(u_\eps(\cdot,t))\dd x,\\
 \mu_{\eps,ij}^t
 &:={}
 \eps\,\partial_i u_\eps(\cdot,t)\cdot
 \partial_j u_\eps(\cdot,t)\dd x.
\end{aligned}
\end{equation}
The spacetime dissipation and momentum measures are
\begin{equation}\label{eq:intro-action}
 \alpha_\eps
 :=\eps\abs{\partial_tu_\eps}^2\dd x\dd t,
 \quad
 q_{\eps,i}
 :=\eps\,\partial_tu_\eps\cdot\partial_i u_\eps\dd x\dd t.
\end{equation}
The diffuse stress is
\begin{equation}\label{eq:intro-stress}
 T_\eps^t
 :=\nu_\eps^t\Id-(\mu_{\eps,ij}^t)_{i,j=1}^2.
\end{equation}

Our main results are as follows.

\begin{thm}[Parabolic Bethuel-type compactness]\label{thm:main}
Let $V$ satisfy Assumption~\ref{assum:potential}, let $T>0$, and let
$\eps_j\downarrow0$.  Suppose that $u_j:=u_{\eps_j}$ are smooth solutions
of \eqref{eq:intro-flow} satisfying
\begin{equation}\label{eq:initial-energy-bound}
 E_{\eps_j}(u_j(\cdot,0);\T^2)\leq M_0.
\end{equation}
After passing to a subsequence, the following conclusions hold.

\begin{enumerate}
\item There is a spacetime Caccioppoli partition $(E_a)_{a=1}^q$ such that
\begin{equation}\label{eq:phase-convergence}
 u_j\to
 u_0:=\sum_{a=1}^q\sigma_a\mathbf1_{E_a}
 \quad\text{in }L^1(\T^2\times(0,T)).
\end{equation}

\item There are a weakly measurable family of Radon measures $(\nu_t)$ and a
common full-measure set $\mathcal C\subset(0,T)$ such that, after choosing
the right-continuous representative,
\begin{equation}\label{eq:energy-slice-convergence}
 \nu_j^t\stackrel{*}{\rightharpoonup}\nu_t
\end{equation}
for every $t\in\mathcal C$.  Moreover,
\begin{equation}\label{eq:spacetime-measure-convergence}
\begin{aligned}
 \nu_j^t\dd t&\stackrel{*}{\rightharpoonup}\nu_t\dd t,
 &\zeta_j^t\dd t&\stackrel{*}{\rightharpoonup}\zeta_t\dd t,\\
 \mu_{j,ab}^t\dd t&\stackrel{*}{\rightharpoonup}\mu_{ab,t}\dd t,
 &q_j&\stackrel{*}{\rightharpoonup}q,
 &\alpha_j&\stackrel{*}{\rightharpoonup}\alpha.
\end{aligned}
\end{equation}
The notation $\zeta_t$ and $\mu_{ab,t}$ refers to the time disintegration
of the spacetime limits; no separate time-slice convergence is asserted for
these measures.
Set $\nu:=\nu_t\dd t$ for the limiting spacetime full-energy measure.

\item For almost every $t\in(0,T)$ there are a closed countably
$1$-rectifiable set $S_t\subset\T^2$ and densities
$e_t,\Theta_t,m_{ab,t}$ such that
\begin{equation}\label{eq:slice-representations}
 \nu_t=e_t\,\cH^1\llcorner S_t,
 \quad
 \zeta_t=\Theta_t\,\cH^1\llcorner S_t,
 \quad
 \mu_{ab,t}=m_{ab,t}\,\cH^1\llcorner S_t.
\end{equation}
For every $\tau>0$, there is $C_\tau<\infty$ such that
\begin{equation}\label{eq:energy-potential-comparison}
 \zeta_t\leq\nu_t\leq C_\tau\zeta_t
 \quad\text{for almost every }t\in(\tau,T).
\end{equation}
At $\cH^1$-almost every point of $S_t$, let $(\tau_t,n_t)$ be the
approximate tangent--normal frame.  Then
\begin{equation}\label{eq:bethuel-slice-relations}
 e_t=m_{n_tn_t,t},
 \quad
 m_{n_t\tau_t,t}=0,
 \quad
 2\Theta_t=m_{n_tn_t,t}-m_{\tau_t\tau_t,t}.
\end{equation}
In particular,
\begin{equation}\label{eq:rho-definition}
 \rho_t:=e_t-2\Theta_t=m_{\tau_t\tau_t,t}\geq0.
\end{equation}

\item The limiting stress satisfies
\begin{equation}\label{eq:limit-stress}
 T_t
 =2\Theta_t\,\tau_t\otimes\tau_t\,
 \cH^1\llcorner S_t,
 \quad
 \diver_xT_t=-q_t
\end{equation}
for almost every $t$, where $q=q_t\dd t$.  The measure
$q$ is absolutely continuous with respect to $\nu_t\dd t$ and has an
$L^2$ density.

\item Let $\lambda_t:=\cH^1\llcorner S_t$.  On
$\T^2\times(0,T)$ write
\begin{equation}\label{eq:action-decomposition}
 \alpha=a\,\lambda_t\dd t+\alpha^{\rm s},
 \quad
 \alpha^{\rm s}\perp\lambda_t\dd t.
\end{equation}
Write
\begin{equation}\label{eq:momentum-density}
 q_t=\mathbf q_t\,\lambda_t,
 \quad
 \mathbf q_t=q_{\tau,t}\tau_t+q_{n,t}n_t.
\end{equation}
Then
\begin{equation}\label{eq:sharp-action-bound}
 a\geq
 \frac{\abs{q_{n,t}}^2}{e_t}
 +\frac{\abs{q_{\tau,t}}^2}{\rho_t}
 \quad\text{for }\lambda_t\dd t\text{-almost every }(x,t),
\end{equation}
where the second quotient is interpreted as $+\infty$ unless
$q_{\tau,t}=0$ on $\{\rho_t=0\}$.

\item The weighted varifold
$\mathbf V_t:=\mathbf V(S_t,\Theta_t)$ has generalized curvature
\[
 H_{\Theta,t}\in
 L^2_{\rm loc}(\T^2\times(0,T),\zeta_t\dd t;\R^2),
\]
and
\begin{equation}\label{eq:weighted-first-variation}
 \delta\mathbf V_t(X)
 =-\int_{S_t}H_{\Theta,t}\cdot X\dd\zeta_t,
 \quad
 q_t=-2\Theta_tH_{\Theta,t}\lambda_t.
\end{equation}
On every regular branch in the sense of
Definition~\ref{defn:regular-branch},
\begin{equation}\label{eq:weighted-curvature-splitting}
 H_{\Theta,t}
 =H_{S_t}+\nabla_{S_t}\log\Theta_t.
\end{equation}
Globally, the notation in the following formulas means
\begin{equation}\label{eq:curvature-projections}
 H_{S_t}:=(n_t\otimes n_t)H_{\Theta,t},
 \quad
 \nabla_{S_t}\Theta_t
 :=\Theta_t(\tau_t\otimes\tau_t)H_{\Theta,t}.
\end{equation}
Define the normal energy-flux field $\mathcal{V}_t$ and the tangential
internal-mobility field $\mathcal{W}_t$ by
\begin{equation}\label{eq:velocity-definitions}
 \mathcal{V}_t:=-\frac{q_{n,t}}{e_t}n_t,
 \quad
 \mathcal{W}_t:=-\frac{q_{\tau,t}}{\rho_t}\tau_t.
\end{equation}
On $\{\rho_t=0\}$, set $\mathcal W_t=0$; the action bound forces
$q_{\tau,t}=0$ there.
Thus the total energy-flux density decomposes as
\begin{equation}\label{eq:total-energy-flux}
 -\frac{\dd q}{\dd\nu}
 =\mathcal{V}_t+\frac{\rho_t}{e_t}\mathcal{W}_t.
\end{equation}
Then
\begin{equation}\label{eq:two-mobility-law}
e_t\mathcal{V}_t=2\Theta_tH_{S_t},
 \quad
 \rho_t\mathcal{W}_t=2\nabla_{S_t}\Theta_t.
\end{equation}

\item For every non-negative
$\phi\in C^1(\T^2\times[0,T])$ and almost every
$0<t_1<t_2<T$,
\begin{equation}\label{eq:localized-flow-inequality}
\begin{aligned}
 &\int_{S_{t_2}}\phi(\cdot,t_2)e_{t_2}\dd\lambda_{t_2}
 -\int_{S_{t_1}}\phi(\cdot,t_1)e_{t_1}\dd\lambda_{t_1}
 \\
 &\leq
 \int_{t_1}^{t_2}\int_{S_t}
 \Bigg[
 e_t\partial_t\phi
 +2\Theta_tH_{S_t}\cdot\nabla\phi
 +2\nabla_{S_t}\Theta_t\cdot\nabla_{S_t}\phi
 \\
 &\hspace{7em}
 -\phi\left(
 \frac{4\Theta_t^2}{e_t}\abs{H_{S_t}}^2
 +\frac{4}{\rho_t}\abs{\nabla_{S_t}\Theta_t}^2
 \right)
 \Bigg]\dd\lambda_t\dd t.
\end{aligned}
\end{equation}
The last quotient follows the same zero-denominator convention as
\eqref{eq:sharp-action-bound}.  The sharper identity contains the additional
non-positive term $-\int\phi\dd\gamma$, where
\begin{equation}\label{eq:dissipation-defect}
 \gamma
 :=\alpha-
 \left(e_t\abs{\mathcal{V}_t}^2+\rho_t\abs{\mathcal{W}_t}^2\right)
 \lambda_t\dd t
 \geq0.
\end{equation}
\end{enumerate}
\end{thm}

\begin{rem}[Relation to the elliptic and scalar theories]
\label{rem:main-comments}
The theorem should be viewed as a parabolic extension of Bethuel's
elliptic compactness theorem.  At almost every positive time, the spatial
energy, potential energy, and gradient tensor satisfy the same
Bethuel-type relations, while the parabolic dissipation determines how the
corresponding rectifiable interface evolves.  In the vectorial setting, the
possible residual density \(\rho_t\) produces an additional tangential
internal-mobility channel.  Consequently,
\[
 -\frac{\dd q}{\dd\nu}
 =
 \mathcal V_t+\frac{\rho_t}{e_t}\mathcal W_t
\]
is the total energy-flux field and, in general, need not coincide with the
velocity of the reduced boundary of the limiting phase partition.

In the scalar equipartition regime, one has
\[
 \rho_t=0,\qquad q_{\tau,t}=0,
\qquad e_t=2\Theta_t.
\]
Hence the tangential channel disappears,
\[
 \mathcal V_t=H_{S_t},
\]
and the localized energy inequality reduces, up to the surface-energy
factor \(2\Theta_t\), to the usual Brakke inequality.  Thus the present
result recovers the scalar Allen--Cahn convergence framework of
Ilmanen~\cite{Ilmanen1993} while extending it to general vector-valued
multi-well potentials, where tangential microstructure and a non-negative
dissipation defect may persist.
\end{rem}

\subsection{Proof strategy and organization}

Bethuel’s elliptic theorem suggests a natural approach to the parabolic problem, namely to recover the spatial structure of the interface at almost every time. The difficulty is that a time slice now satisfies an elliptic equation with forcing $F_\varepsilon=-\varepsilon\partial_tu_\varepsilon$. It would be tempting to regard this as a small perturbation of the stationary problem, since a bound on $\int |F_\varepsilon|^2/\varepsilon$ implies that $F_\varepsilon\to0$ in $L^2$. This smallness, however, does not carry over to the stress balance. The gradients concentrate in the transition layers, and $F_\varepsilon\cdot\nabla u_\varepsilon$ may have a nonzero limit. A forcing term that disappears in the equation can therefore still contribute to the limiting curvature balance. We extend Bethuel’s clearing-out estimates while retaining these contributions at the scale $\int |F_\varepsilon|^2/\varepsilon$, which agrees exactly with the parabolic dissipation. Unlike the stationary theory, our result allows the limiting stress to have nonzero divergence. Fortunately, its singular part still has rank one in two dimensions, and the rectifiable structure survives. There is a further issue in transferring this elliptic analysis to the evolution. The dissipation estimate provides only the spacetime bound
\[
\int_0^T\!\int \frac{|F_\varepsilon|^2}{\varepsilon}\,dx\,dt
=\int_0^T\!\int \varepsilon|\partial_tu_\varepsilon|^2\,dx\,dt\le C,
\]
rather than a uniform bound at each individual time. Consequently, the subsequence extracted to analyze a fixed time may depend on that time, and such time-dependent choices cannot yield a coherent evolution law. To avoid this problem, we first select the full-measure set of times at which the energy slices converge, and then extract a single subsequence from the spacetime measures
\[
\frac{\varepsilon}{2}|\nabla u_\varepsilon|^2\,dx\,dt
\quad\text{and}\quad
\frac1\varepsilon V(u_\varepsilon)\,dx\,dt.
\]
Their disintegration in time simultaneously identifies the limiting potential and gradient-tensor measures on almost every slice. This common spacetime limit is what allows the elliptic structure theorem to be applied consistently along the evolution and the resulting geometric balance law to be formulated in time.

The difference from the scalar theory is clearest in the evolution law. In the scalar case, equipartition gives $\nu_t=2\zeta_t$, so the full energy density and the weight in the limiting stress are the same up to the fixed factor $2$. In the vectorial setting, we only know the weaker comparison $\zeta_t\le \nu_t\le C\zeta_t$. Thus, $\nu_t$ and $\zeta_t$ are concentrated on the same rectifiable interface, but their densities need not coincide. This distinction reflects the fact that the energy stored in tangential oscillations inside the transition layer may survive in the limit. As Bethuel’s examples show, the collapse of the layer does not necessarily eliminate this internal tangential energy. We therefore have to retain the residual energy $\rho_t=e_t-2\Theta_t$ in the limiting description. Consequently, the energy flux need not be purely normal. For this reason, the dissipation estimate must involve the full spatial gradient tensor, rather than only the total flux. A bound on the total flux would not separate the normal component from the tangential component. The positivity of the spacetime Gram matrix provides exactly this separation. Combining it with the stress balance gives the normal evolution law $e_t\mathcal V_t=2\Theta_t H_{S_t}$, as well as a tangential internal-mobility law. The residual energy also changes the normal balance itself---the transported energy has density $e_t$, while the stress is weighted by $2\Theta_t$. These evolution laws therefore describe the full limiting energy, including possible concentrations that are not visible in the limiting phase partition. Finally, one obtains a localized energy inequality with a non-negative dissipation defect. If equipartition does hold, then the same dissipation estimate forces the tangential flux to vanish. In that case, the normal law reduces to the usual mean-curvature balance, and the standard Brakke inequality follows.

\medskip
\noindent
\textbf{Step 1: Forced energy decay and clearing-out.}
We first study the forced elliptic equation
\[
 -\varepsilon\Delta v+\frac{1}{\varepsilon}DV(v)=F
\]
at the natural action scale
\[
 \mathcal D_\varepsilon(F):=\int\frac{|F|^2}{\varepsilon}.
\]
A good-circle argument shows that small boundary energy forces the trace of
\(v\) to remain close to a single well.  Target-space localization near the
wells, a planar level-set argument, and a forced Pohozaev identity then give
the decay estimate
\[
 E_\varepsilon(v;B_{1/2})
 \leq C\left(
 E_\varepsilon(v;B_1)^{3/2}
 +\varepsilon E_\varepsilon(v;B_1)
 +\sqrt{\mathcal D_\varepsilon(F;B_1)E_\varepsilon(v;B_1)}
 +\varepsilon^2\mathcal D_\varepsilon(F;B_1)
 \right).
\]
A dyadic iteration converts this estimate into a clearing-out theorem:
sufficiently small normalized energy and forcing confine \(v\) to a single
well and yield strong decay on smaller balls.

\medskip
\noindent
\textbf{Step 2: Forced planar Bethuel compactness.}
The clearing-out theorem implies a positive lower one-dimensional density
for every non-trivial limiting energy measure.  A \(5r\)-covering argument
therefore gives locally finite \(\mathcal H^1\)-measure of its support and
shows that the limiting energy is singular with respect to two-dimensional Lebesgue
measure.  An additional local estimate gives the mutual absolute
continuity of the limiting full and potential energies,
\[
 \zeta\leq\nu\leq C\zeta.
\]
For the diffuse stress
\[
 T_\varepsilon
 =e_\varepsilon(v)I-\varepsilon\,\nabla v^{\mathsf T}\nabla v,
\]
the equation gives the exact identities
\[
 \operatorname{div}T_\varepsilon=F\cdot\nabla v,
 \qquad
 \operatorname{tr}T_\varepsilon=\frac{2}{\varepsilon}V(v).
\]
The structure theorem for divergence-constrained measures forces the polar
of the singular limiting stress to have rank one.  Allard's theorem then
gives a countably \(1\)-rectifiable carrier \(S\), and
\[
 T=2\Theta\,\tau\otimes\tau\,
 \mathcal H^1\llcorner S.
\]
Comparison with \(T=\nu I-\mu\) in the tangent--normal frame
\((\tau,n)\) yields the Bethuel relations
\[
 e=m_{nn},
 \qquad
 m_{n\tau}=0,
 \qquad
 2\Theta=m_{nn}-m_{\tau\tau}.
\]
In particular, the residual tangential density
\[
 \rho:=e-2\Theta=m_{\tau\tau}
\]
is non-negative but need not vanish.

\medskip
\noindent
\textbf{Step 3: Spacetime compactness and almost-every-time structure.}
For the parabolic equation, the exact energy identities are
\[
 \operatorname{div}_xT_\varepsilon=-q_\varepsilon,
 \qquad
 \partial_te_\varepsilon
 =\operatorname{div}_xq_\varepsilon-\alpha_\varepsilon,
\]
where
\[
 q_{\varepsilon,i}
 =\varepsilon\partial_tu_\varepsilon\cdot\partial_i u_\varepsilon,
 \qquad
 \alpha_\varepsilon
 =\varepsilon|\partial_tu_\varepsilon|^2.
\]
The global energy law provides uniform bounds for both the diffuse energy
and the spacetime dissipation.  Weighted distances to the wells are bounded
in spacetime \(BV\), which gives convergence to a Caccioppoli partition.
Helly compactness applied to the time-dependent energy measures produces
limiting slices \(\nu_t\).

For almost every \(t>0\), the dissipation satisfies
\[
 \liminf_{\varepsilon\downarrow0}
 \int\varepsilon|\partial_tu_\varepsilon(\cdot,t)|^2<\infty.
\]
At such a time, \(u_\varepsilon(\cdot,t)\) solves the forced elliptic
equation with
\[
 F_\varepsilon=-\varepsilon\partial_tu_\varepsilon,
 \qquad
 \int\frac{|F_\varepsilon|^2}{\varepsilon}
 =\int\varepsilon|\partial_tu_\varepsilon|^2.
\]
The forced Bethuel theorem therefore applies and yields a rectifiable
interface \(S_t\) and the Bethuel structure on almost every time slice.
A spacetime energy--potential comparison then identifies the disintegrated
limits \(\zeta_t\) and \(\mu_t\) with measures supported on the same set
\(S_t\).

\medskip
\noindent
\textbf{Step 4: Sharp dissipation and the two-mobility law.}
Consider the positive semidefinite spacetime Gram matrix
\[
 \mathcal G_\varepsilon
 =
 \varepsilon
 \begin{pmatrix}
  \partial_tu_\varepsilon\\
  \partial_1u_\varepsilon\\
  \partial_2u_\varepsilon
 \end{pmatrix}
 \begin{pmatrix}
  \partial_tu_\varepsilon\\
  \partial_1u_\varepsilon\\
  \partial_2u_\varepsilon
 \end{pmatrix}^{\!\mathsf T}.
\]
Positivity passes to the limit.  Since the spatial block is
\[
 \mu_t
 =\rho_t\,\tau_t\otimes\tau_t
  +e_t\,n_t\otimes n_t,
\]
the Schur-complement inequality gives the sharp action bound
\[
 a\geq
 \frac{|q_{n,t}|^2}{e_t}
 +\frac{|q_{\tau,t}|^2}{\rho_t}.
\]
On the other hand, the limiting stress identity identifies the momentum
with the first variation of the weighted varifold
\(\mathbf V(S_t,\Theta_t)\):
\[
 q_t=-2\Theta_tH_{\Theta,t}\,
 \mathcal H^1\llcorner S_t.
\]
On a regular branch,
\[
 H_{\Theta,t}
 =H_{S_t}+\nabla_{S_t}\log\Theta_t.
\]
Separating normal and tangential components and defining
\[
 \mathcal V_t=-\frac{q_{n,t}}{e_t}n_t,
 \qquad
 \mathcal W_t=-\frac{q_{\tau,t}}{\rho_t}\tau_t,
\]
we obtain
\[
 e_t\mathcal V_t=2\Theta_tH_{S_t},
 \qquad
 \rho_t\mathcal W_t=2\nabla_{S_t}\Theta_t.
\]
The first identity describes normal curvature motion, while the second
records tangential transport of the residual vectorial microstructure.

\medskip
\noindent
\textbf{Step 5: Localized energy inequality.}
Finally, we pass to the limit in the exact localized diffuse-energy
identity.  The momentum decomposition and the two mobility identities
identify the limiting flux terms, while the sharp action bound gives the
non-negative defect measure
\[
 \gamma
 :=
 \alpha-
 \left(
 e_t|\mathcal V_t|^2
 +\rho_t|\mathcal W_t|^2
 \right)
 \mathcal H^1\llcorner S_t\,\dd t
 \geq0.
\]
Substitution yields the localized energy inequality, with the two
dissipation channels
\[
 \frac{4\Theta_t^2}{e_t}|H_{S_t}|^2
 \qquad\text{and}\qquad
 \frac{4}{\rho_t}|\nabla_{S_t}\Theta_t|^2.
\]
When \(\rho_t=0\), finite action forces the tangential channel to vanish,
\(e_t=2\Theta_t\), and the normal law reduces to
\(\mathcal V_t=H_{S_t}\), recovering the usual Brakke inequality.

\medskip
\noindent
\textbf{Organization of the paper.}
The remainder of the paper follows the preceding strategy.  Section~2
develops the forced clearing-out theory by Bethuel's arguments.  Section~3 proves the forced planar Bethuel
compactness theorem by combining the clearing-out mechanism with the
energy--potential comparison, the rank-one structure of the limiting
stress, and varifold rectifiability.  Section~4 returns to the parabolic
problem: it establishes the diffuse balance laws, positive-time estimates,
phase and energy compactness, and the almost-every-time rectifiable
Bethuel structure.  Section~5 studies the evolution of the limiting
measures.  It proves the sharp action lower bound, identifies the weighted
curvature and the two mobility fields, derives the localized energy
inequality and its dissipation defect, and concludes with the balance law
at finite junctions.  The main theorem is obtained by assembling these
elliptic, compactness, and measure-evolution results.

\section{Forced energy decay and clearing-out}
\label{sec:forced-clearing}

The purpose of this chapter is to show that, when both the total energy and
the forcing are sufficiently small at the natural scale
\[
D_\eps(F;B_1):=\int_{B_1}\frac{|F|^2}{\eps}\,\dd x,
\]
the solution \(v\) remains close to a single well of the potential in the
interior, and its energy enjoys an additional \(\eps^2\)-scale decay.

The proof is divided into three stages. 

In the first stage, we establish a
unit-scale \(3/2\)-decay estimate.  We first use the coarea formula to select
a circle with small energy.  Combining the near-well energy estimate with a
planar level-set projection argument, we then construct an inner circle
which is entirely contained in a single well region.  The forced Pohozaev
identity is subsequently used to control the interior potential energy.  This
yields
\[
E_\eps(v;B_{1/2})
\leq
C\left(
E_\eps(v;B_1)^{3/2}
+\eps E_\eps(v;B_1)
+D_\eps(F;B_1)^{1/2}E_\eps(v;B_1)^{1/2}
+\eps^2D_\eps(F;B_1)
\right).
\]

In the second stage, this estimate is rescaled to dyadic balls and iterated.
When the radius reaches the order of \(\sqrt{\eps}\), the energy satisfies the
weak clearing-out criterion.  The latter follows from the gradient bound and
the positivity of the potential away from the wells, and yields clearing-out at
the center.

In the final stage, the center result is applied after rescaling
around every point of \(B_{3/4}\).  Continuity of \(v\), together with the
connectedness of \(B_{3/4}\), then implies that the entire image of the
interior disk lies in one and the same well.  Finally, testing the equation
with \(\chi^2(v-\sigma)\), where \(\chi\) is a cutoff supported in
\(B_{3/4}\), gives the strong inner energy estimate on \(B_{5/8}\).  Apart
from the explicit error terms generated by the forcing, the argument follows
the energy-decay and clearing-out framework developed by Bethuel.

\subsection{Setting and relation with the parabolic equation}

For an open set \(U\subset B_1\), define
\[
 e_\eps(v)
 :=
 \frac{\eps}{2}|\nabla v|^2+\frac1\eps V(v),
 \qquad
 E_\eps(v;U):=\int_U e_\eps(v)\,\dd x .
\]
Throughout this section we assume that
\[
 v\in C^2(B_1;\R^m)\cap C^1(\overline{B_1};\R^m),
 \qquad
 F\in L^2(B_1;\R^m),
\]
and that
\begin{equation}\label{eq:forced-elliptic}
-\eps\Delta v+\frac1\eps DV(v)=F
\qquad\text{in }B_1
\end{equation}
in the distributional sense.

\begin{rem}[Relation with the parabolic problem]
\label{rem:forced-parabolic}

The parameter \(\eps\) in the elliptic equation is the same parameter
as in the parabolic Allen--Cahn equation \eqref{eq:intro-flow}.  At a
fixed time \(t\), the latter can be rewritten as
\[
-\eps\Delta u_\eps(\cdot,t)
+\frac1\eps DV(u_\eps(\cdot,t))
=
-\eps\partial_tu_\eps(\cdot,t).
\]
Thus, in the parabolic application,
\[
v=u_\eps(\cdot,t),
\qquad
F=-\eps\partial_tu_\eps(\cdot,t),
\]
and therefore
\[
\int_{B_1}\frac{|F|^2}{\eps}\,\dd x
=
\int_{B_1}\eps|\partial_tu_\eps(\cdot,t)|^2\,\dd x .
\]
This is precisely the time-slice dissipation density.

When \(F=0\), the arguments below reduce to the elliptic estimates
developed by Bethuel in Sections~2--6 of \cite{Bethuel2025}.  Each
subsequent result contains a remark identifying the corresponding
statement in that paper.
\end{rem}

Fix \(\mu>0\), depending only on \(V\), so small that the closed balls
\(\overline{B_\mu(\sigma_a)}\), \(a=1,\ldots,q\), are pairwise disjoint.
We also assume that, for suitable constants \(c_V,C_V>0\),
\begin{equation}
\label{eq:forced-well-coercivity}
\begin{aligned}
c_V|y-\sigma_a|^2
&\leq V(y)\leq C_V|y-\sigma_a|^2,\\
DV(y)\cdot(y-\sigma_a)
&\geq c_V|y-\sigma_a|^2
\end{aligned}
\end{equation}
whenever \(|y-\sigma_a|\leq\mu\).

For every fixed \(L<\infty\), after decreasing \(\mu\) if necessary,
\[
c_0(V,L,\mu)
:=
\inf\left\{
V(y):
|y|\leq L,\ 
\dist(y,\Sigma)\geq\frac{\mu}{4}
\right\}>0 .
\]
Equivalently,
\begin{equation}
\label{eq:forced-potential-away}
V(y)\geq c_0(V,L,\mu)
\quad\text{if } |y|\leq L,\ 
\dist(y,\Sigma)\geq\frac{\mu}{4}.
\end{equation}

\begin{thm}[Forced energy decay and clearing-out]
\label{thm:forced-clearing}

Let \(L<\infty\).  There exist constants
\[
\eta_{\rm f}=\eta_{\rm f}(V,L)>0,\qquad
\gamma_{\rm f}=\gamma_{\rm f}(V,L)>0,\qquad
C_{\rm f}=C_{\rm f}(V,L)<\infty
\]
such that the following holds.

Assume \(0<\eps\leq1\), \(F\in L^2(B_1;\R^m)\), \(v\) solves
\eqref{eq:forced-elliptic}, and
\begin{equation}
\label{eq:forced-uniform-bound}
\|v\|_{L^\infty(B_1)}
+\eps\|\nabla v\|_{L^\infty(B_1)}
\leq L .
\end{equation}
If
\begin{equation}
\label{eq:forced-smallness}
E_\eps(v;B_1)\leq2\eta_{\rm f},
\qquad
\int_{B_1}\frac{|F|^2}{\eps}\,\dd x
\leq\gamma_{\rm f},
\end{equation}
then there exists a single \(\sigma\in\Sigma\) such that
\begin{equation}
\label{eq:forced-confinement}
|v-\sigma|\leq\frac{\mu}{2}
\qquad\text{on }B_{3/4}.
\end{equation}

Moreover,
\begin{equation}
\label{eq:forced-strong-inner}
E_\eps\left(v;B_{5/8}\right)
\leq
C_{\rm f}\eps^2
\left\{
E_\eps(v;B_1)
+\int_{B_1}\frac{|F|^2}{\eps}\,\dd x
\right\}.
\end{equation}

Before confinement is established, one has
\begin{equation}
\label{eq:forced-unit-decay}
\begin{aligned}
E_\eps\left(v;B_{1/2}\right)
\leq C_{\rm f}\bigg\{&
E_\eps(v;B_1)^{3/2}
+\eps E_\eps(v;B_1)\\
&+
\left(
\int_{B_1}\frac{|F|^2}{\eps}\,\dd x
\right)^{1/2}
E_\eps(v;B_1)^{1/2}\\
&+
\eps^2\int_{B_1}\frac{|F|^2}{\eps}\,\dd x
\bigg\}.
\end{aligned}
\end{equation}
\end{thm}

\begin{rem}[Correspondence with Bethuel]
\label{rem:forced-clearing-bethuel}

When \(F=0\), the confinement estimate
\eqref{eq:forced-confinement} is the analogue of
\cite[Theorem~1.11, (1.58)]{Bethuel2025}, while
\eqref{eq:forced-unit-decay} corresponds to
\cite[Proposition~1.12, (1.60)]{Bethuel2025}.

The estimate \eqref{eq:forced-strong-inner} is the forced analogue of
the near-well estimate in \cite[Proposition~6.8 and Theorem~1.11,
(1.59)]{Bethuel2025}.  The additional terms are generated only by the
forcing and are controlled by Cauchy--Schwarz and Young inequalities.
\end{rem}

\subsection{Good circles and a planar level-set lemma}

Choose a smooth non-decreasing function
\[
\psi:[0,\infty)\longrightarrow[0,3\mu/4]
\]
such that
\[
\psi(s)=s\quad\text{for }s\leq\frac{\mu}{2},
\qquad
0\leq\psi'\leq1,
\qquad
\psi(s)=\frac{3\mu}{4}\quad\text{for }s\geq\mu.
\]
For \(a=1,\ldots,q\), define
\[
\chi_a(y):=\psi(|y-\sigma_a|).
\]
Whenever \(\nabla(\chi_a\circ v)\neq0\), the point \(v\) lies in
\(B_\mu(\sigma_a)\), and therefore
\[
\chi_a(v)^2\leq C V(v),
\qquad
|\nabla(\chi_a\circ v)|\leq|\nabla v|.
\]
Consequently,
\begin{equation}
\label{eq:forced-modica-mortola}
\left|\nabla\bigl(\chi_a(v)^2\bigr)\right|
\leq C\sqrt{V(v)}\,|\nabla v|
\leq C e_\eps(v).
\end{equation}

\begin{rem}[Correspondence with Bethuel]

The construction of \(\chi_a\) and estimate
\eqref{eq:forced-modica-mortola} correspond to
\cite[Lemmas~2.4 and 2.5, especially (2.10)--(2.13)]{Bethuel2025}.
They are algebraic consequences of the structure of \(V\) near its wells
and are unaffected by the forcing.
\end{rem}

\begin{lem}[A good circle]
\label{lem:forced-circle}

Let \(L<\infty\).  There exist \(h_0=h_0(V,L)>0\) and
\(C=C(V,L)<\infty\) with the following property.

Assume \(0<\eps\leq4\), \(r\in[1/2,1]\), and
\[
\|v\|_{L^\infty(B_1)}\leq L .
\]
If
\[
\int_{\partial B_r}e_\eps(v)\,\dd\cH^1\leq h_0,
\]
then there exists \(\sigma\in\Sigma\) such that
\begin{equation}
\label{eq:forced-circle-bound}
\sup_{\partial B_r}|v-\sigma|
\leq
C\left(
\int_{\partial B_r}e_\eps(v)\,\dd\cH^1
\right)^{1/2}.
\end{equation}

Furthermore, for every \(1/2\leq r_0<r_1\leq1\), there exists
\(r\in[r_0,r_1]\) such that
\begin{equation}
\label{eq:forced-circle-average}
\int_{\partial B_r}e_\eps(v)\,\dd\cH^1
\leq
\frac1{r_1-r_0}
\int_{B_{r_1}\setminus B_{r_0}}e_\eps(v)\,\dd x .
\end{equation}
\end{lem}

\begin{proof}

The second assertion follows directly from the coarea formula:
\[
\int_{r_0}^{r_1}
\left(
\int_{\partial B_s}e_\eps(v)\,\dd\cH^1
\right)\dd s
=
\int_{B_{r_1}\setminus B_{r_0}}e_\eps(v)\,\dd x .
\]

Set
\[
h:=\int_{\partial B_r}e_\eps(v)\,\dd\cH^1 .
\]
Since \(r\geq1/2\) and \(\eps\leq4\),
\[
\int_{\partial B_r}V(v)\,\dd\cH^1
\leq \eps h
\leq4h.
\]
Because \(\cH^1(\partial B_r)=2\pi r\geq\pi\), there exists
\(x_0\in\partial B_r\) such that
\[
V(v(x_0))\leq\frac{4}{\pi}h.
\]
For \(h_0\) sufficiently small, \eqref{eq:forced-potential-away}
implies that \(v(x_0)\in B_{\mu/2}(\sigma)\) for some \(\sigma\in\Sigma\).
Using \eqref{eq:forced-well-coercivity},
\[
|v(x_0)-\sigma|^2
\leq C V(v(x_0))
\leq Ch .
\]
After decreasing \(h_0\), we may assume
\(|v(x_0)-\sigma|<\mu/2\).  Hence
\[
\chi_\sigma(v(x_0))=|v(x_0)-\sigma|.
\]

For any \(x\in\partial B_r\), choose one of the two arcs of
\(\partial B_r\) joining \(x_0\) to \(x\).  By \eqref{eq:forced-modica-mortola},
\[
\begin{aligned}
\chi_\sigma(v(x))^2
&\leq
\chi_\sigma(v(x_0))^2
+
\int_{\partial B_r}
\left|\nabla(\chi_\sigma(v)^2)\right|\,\dd\cH^1\\
&\leq Ch+Ch
\leq Ch .
\end{aligned}
\]
Thus, after decreasing \(h_0\) once more,
\[
\chi_\sigma(v(x))<\frac{\mu}{2}
\qquad\text{for all }x\in\partial B_r .
\]
By the definition of \(\psi\),
\[
\chi_\sigma(v(x))=|v(x)-\sigma|
\qquad\text{on }\partial B_r,
\]
and \eqref{eq:forced-circle-bound} follows.
\end{proof}

\begin{rem}[Correspondence with Bethuel]

Lemma~\ref{lem:forced-circle} combines the one-dimensional estimate
in \cite[Lemma~2.6]{Bethuel2025} with the good-radius selection in
\cite[Lemma~2.7]{Bethuel2025}.  Neither part uses the elliptic equation,
so the argument is unchanged by the forcing.
\end{rem}

\begin{lem}[Planar shadow estimate]
\label{lem:planar-shadow}

Let \(w\in C(\overline{B_\rho})\) be \(C^1\) in a neighborhood of
\(\{w\geq s\}\).  Assume
\[
s<h,\qquad
w<s\quad\text{on }\partial B_\rho,
\]
and assume that \(s\) is a regular value of \(w\).

For \(0<a<\rho\), define
\[
\mathcal R
:=
\left\{
r\in[a,\rho]:
\partial B_r\cap\{w>h\}\neq\varnothing
\right\}.
\]
Then
\begin{equation}
\label{eq:planar-shadow}
\mathcal L^1(\mathcal R)
\leq
\frac12
\cH^1\bigl(\{w=s\}\cap B_\rho\bigr).
\end{equation}
\end{lem}

\begin{proof}

Let \((G_\ell)_\ell\) be the connected components of
\(\{w>s\}\) which intersect \(\{w>h\}\).  Since \(w<s\) on
\(\partial B_\rho\), every \(G_\ell\) is relatively compact in
\(B_\rho\).  Its radial projection
\[
I_\ell:=\{|x|:x\in G_\ell\}
\]
is an interval, because \(G_\ell\) is connected.  Consequently,
\[
\mathcal L^1(I_\ell)
\leq
diam(G_\ell).
\]
Every \(r\in\mathcal R\) belongs to at least one \(I_\ell\), and hence
\[
\mathcal L^1(\mathcal R)
\leq
\sum_\ell diam(G_\ell).
\]

Because \(s\) is a regular value, \(\{w=s\}\) is a compact
one-dimensional \(C^1\) submanifold of \(B_\rho\).  In particular,
\(\partial G_\ell\) is a finite union of \(C^1\) closed curves and
\[
\partial G_\ell\subset\{w=s\}.
\]

We use the following elementary planar fact:
\begin{equation}
\label{eq:planar-boundary-diameter}
2 diam(G_\ell)
\leq
\cH^1(\partial G_\ell).
\end{equation}
Indeed, let \(D_\ell=diam(G_\ell)\).  Choose
\(p_\ell,q_\ell\in\overline{G_\ell}\) such that
\[
|p_\ell-q_\ell|=D_\ell,
\]
and put
\[
e_\ell:=\frac{p_\ell-q_\ell}{|p_\ell-q_\ell|}.
\]
Let
\[
\alpha_\ell:=\min_{\overline{G_\ell}}x\cdot e_\ell,
\qquad
\beta_\ell:=\max_{\overline{G_\ell}}x\cdot e_\ell .
\]
Then
\[
\beta_\ell-\alpha_\ell\geq D_\ell.
\]
For almost every \(t\in(\alpha_\ell,\beta_\ell)\), the line
\[
\{x:x\cdot e_\ell=t\}
\]
intersects \(G_\ell\) in a bounded open interval containing at least
one point.  Its two endpoints belong to \(\partial G_\ell\).  Therefore
\[
\cH^0\bigl(
\partial G_\ell\cap\{x:x\cdot e_\ell=t\}
\bigr)\geq2
\]
for almost every \(t\in(\alpha_\ell,\beta_\ell)\).  Applying the
one-dimensional coarea formula to the \(C^1\) curves
\(\partial G_\ell\), we obtain
\[
\begin{aligned}
\cH^1(\partial G_\ell)
&\geq
\int_{\partial G_\ell}
\left|\nabla_{\partial G_\ell}(x\cdot e_\ell)\right|
\,\dd\cH^1\\
&=
\int_\R
\cH^0\bigl(
\partial G_\ell\cap\{x:x\cdot e_\ell=t\}
\bigr)\,\dd t\\
&\geq
2(\beta_\ell-\alpha_\ell)
\geq2D_\ell .
\end{aligned}
\]
This proves \eqref{eq:planar-boundary-diameter}.

Distinct components \(G_\ell\) have disjoint boundary curves up to
\(\cH^1\)-null sets.  Hence
\[
\sum_\ell\cH^1(\partial G_\ell)
\leq
\cH^1\bigl(\{w=s\}\cap B_\rho\bigr).
\]
Combining the preceding inequalities gives
\[
2\mathcal L^1(\mathcal R)
\leq
2\sum_\ell diam(G_\ell)
\leq
\sum_\ell\cH^1(\partial G_\ell)
\leq
\cH^1\bigl(\{w=s\}\cap B_\rho\bigr),
\]
which is \eqref{eq:planar-shadow}.
\end{proof}

\begin{rem}[Correspondence with Bethuel]

Lemma~\ref{lem:planar-shadow} is the geometric ingredient behind
\cite[Proposition~2.9]{Bethuel2025}.  Bethuel formulates the conclusion
as a lower bound for the set of good radii, whereas
\eqref{eq:planar-shadow} estimates the complementary set of bad radii.
The two formulations are equivalent.  This lemma is purely planar and
does not use the equation.
\end{rem}

\subsection{Energy estimates near the wells}

For \(0<\rho\leq4/5\), define
\[
\Xi_\rho
:=
\left\{
x\in B_\rho:
\dist(v(x),\Sigma)\geq\frac{\mu}{4}
\right\}.
\]
By \eqref{eq:forced-uniform-bound},
\[
\frac{\eps}{2}|\nabla v|^2
\leq
\frac{L^2}{2\eps}.
\]
Together with \eqref{eq:forced-potential-away}, this implies
\begin{equation}
\label{eq:forced-far-energy}
e_\eps(v)
\leq
\frac{C(V,L)}{\eps}V(v)
\qquad\text{on }\Xi_\rho .
\end{equation}

For \(0<s\leq\mu/2\), set
\[
\Omega_a(s,\rho)
:=
\left\{
x\in B_\rho:
|v(x)-\sigma_a|<s
\right\},
\]
and
\[
\Upsilon(s,\rho)
:=
\bigcup_{a=1}^q\Omega_a(s,\rho).
\]
The sets \(\Omega_a(s,\rho)\) are pairwise disjoint.

\begin{lem}[Energy near the wells]
\label{lem:forced-target}

Let
\[
\rho\in\left[\frac{9}{16},\frac45\right].
\]
Then
\begin{equation}
\label{eq:forced-fixed-well}
\begin{aligned}
E_\eps\left(
v;\Upsilon\left(\frac{\mu}{4},\rho\right)
\right)
\leq C\bigg\{&
\int_{B_\rho}\frac{V(v)}{\eps}\,\dd x
+\eps\int_{\partial B_\rho}e_\eps(v)\,\dd\cH^1\\
&+
\eps^2\int_{B_\rho}\frac{|F|^2}{\eps}\,\dd x
\bigg\}.
\end{aligned}
\end{equation}
Consequently,
\begin{equation}
\label{eq:forced-full-by-potential}
\begin{aligned}
E_\eps(v;B_\rho)
\leq C\bigg\{&
\int_{B_\rho}\frac{V(v)}{\eps}\,\dd x
+\eps\int_{\partial B_\rho}e_\eps(v)\,\dd\cH^1\\
&+
\eps^2\int_{B_\rho}\frac{|F|^2}{\eps}\,\dd x
\bigg\}.
\end{aligned}
\end{equation}

Assume in addition that
\[
0<\kappa<\frac{\mu}{4}
\]
and that, for some \(\sigma_1\in\Sigma\),
\begin{equation}
\label{eq:forced-boundary-well}
|v-\sigma_1|\leq\frac{\kappa}{4}
\qquad\text{on }\partial B_\rho .
\end{equation}
Then
\begin{equation}
\label{eq:forced-refined-well}
\begin{aligned}
E_\eps(v;\Upsilon(\kappa,\rho))
\leq C\bigg\{&
\kappa\int_{B_\rho}\frac{V(v)}{\eps}\,\dd x
+\eps\int_{\partial B_\rho}e_\eps(v)\,\dd\cH^1\\
&+
\eps
\left(
\int_{B_\rho}\frac{|F|^2}{\eps}\,\dd x
\right)^{1/2}
\left(
\int_{B_\rho}\frac{V(v)}{\eps}\,\dd x
\right)^{1/2}\\
&+
\eps^2\int_{B_\rho}\frac{|F|^2}{\eps}\,\dd x
\bigg\}.
\end{aligned}
\end{equation}
\end{lem}

\begin{proof}

Write
\[
r_a(x):=|v(x)-\sigma_a|.
\]
We first establish the integration-by-parts identity for regular levels.

Let \(s\in(0,\mu/2)\) be a regular value of \(r_a\), and define
\[
\Gamma_a(s,\rho)
:=
\partial\Omega_a(s,\rho)\cap B_\rho,
\qquad
\Pi_a(s,\rho)
:=
\partial\Omega_a(s,\rho)\cap\partial B_\rho .
\]
On \(\Gamma_a(s,\rho)\), the outward unit normal to
\(\Omega_a(s,\rho)\) satisfies
\[
\partial_\nu r_a=|\nabla r_a|.
\]
Testing the equation by \(v-\sigma_a\) on \(\Omega_a(s,\rho)\) gives
\begin{align}
&\int_{\Omega_a(s,\rho)}
\left\{
\eps|\nabla v|^2
+\frac1\eps DV(v)\cdot(v-\sigma_a)
\right\}\dd x
\label{eq:forced-target-identity}\\
&\quad=
\eps s
\int_{\Gamma_a(s,\rho)}
|\nabla r_a|\,\dd\cH^1
+
\eps
\int_{\Pi_a(s,\rho)}
\partial_\nu v\cdot(v-\sigma_a)\,\dd\cH^1
\nonumber\\
&\qquad
+
\int_{\Omega_a(s,\rho)}
F\cdot(v-\sigma_a)\,\dd x .
\nonumber
\end{align}
Indeed, integration by parts gives
\[
-\eps\int_{\Omega_a}\Delta v\cdot(v-\sigma_a)
=
\eps\int_{\Omega_a}|\nabla v|^2
-
\eps\int_{\partial\Omega_a}
\partial_\nu v\cdot(v-\sigma_a),
\]
and on \(\Gamma_a(s,\rho)\),
\[
\partial_\nu v\cdot(v-\sigma_a)
=
\partial_\nu r_a\,r_a
=
s|\nabla r_a|.
\]

By \eqref{eq:forced-well-coercivity},
\[
DV(v)\cdot(v-\sigma_a)\geq c_V|v-\sigma_a|^2,
\qquad
V(v)\leq C_V|v-\sigma_a|^2
\]
on \(\Omega_a(s,\rho)\). Therefore there is \(c_*>0\) such that
\begin{equation}
\label{eq:forced-target-coercive}
E_\eps(v;\Omega_a(s,\rho))
\leq
C
\int_{\Omega_a(s,\rho)}
\left\{
\eps|\nabla v|^2
+\frac1\eps DV(v)\cdot(v-\sigma_a)
\right\}\dd x .
\end{equation}

The forcing term satisfies
\begin{align}
\left|
\int_{\Omega_a(s,\rho)}
F\cdot(v-\sigma_a)\,\dd x
\right|
&\leq
\left(
\int_{\Omega_a(s,\rho)}
\frac{|F|^2}{\eps}\,\dd x
\right)^{1/2}
\left(
\int_{\Omega_a(s,\rho)}
\eps|v-\sigma_a|^2\,\dd x
\right)^{1/2}
\nonumber\\
&\leq
C\eps
\left(
\int_{\Omega_a(s,\rho)}
\frac{|F|^2}{\eps}\,\dd x
\right)^{1/2}
\left(
\int_{\Omega_a(s,\rho)}
\frac{V(v)}{\eps}\,\dd x
\right)^{1/2}.
\label{eq:forced-target-force}
\end{align}
Consequently, for every \(\delta>0\),
\begin{equation}
\label{eq:forced-target-force-young}
\left|
\int_{\Omega_a(s,\rho)}
F\cdot(v-\sigma_a)\,\dd x
\right|
\leq
\delta
\int_{\Omega_a(s,\rho)}
\frac{V(v)}{\eps}\,\dd x
+
C_\delta\eps^2
\int_{\Omega_a(s,\rho)}
\frac{|F|^2}{\eps}\,\dd x .
\end{equation}

Define the transition shell
\[
\mathcal S_\rho
:=
\bigcup_{a=1}^q
\left\{
\frac{\mu}{4}<r_a<\frac{\mu}{2}
\right\}\cap B_\rho .
\]
By the separation of the wells,
\[
\mathcal S_\rho\subset\Xi_\rho.
\]
Thus, by \eqref{eq:forced-far-energy},
\[
\int_{\mathcal S_\rho}e_\eps(v)\,\dd x
\leq
C\int_{B_\rho}\frac{V(v)}{\eps}\,\dd x .
\]

Since
\[
|\nabla r_a|\leq|\nabla v|,
\qquad
\eps|\nabla v|^2\leq2e_\eps(v),
\]
the coarea formula implies
\begin{align*}
&\int_{\mu/4}^{\mu/2}
\left[
\eps\sum_{a=1}^q
\int_{\{r_a=s\}\cap B_\rho}
|\nabla r_a|\,\dd\cH^1
\right]\dd s\\
&\qquad=
\eps\sum_{a=1}^q
\int_{\{\mu/4<r_a<\mu/2\}\cap B_\rho}
|\nabla r_a|^2\,\dd x\\
&\qquad\leq
C\int_{\mathcal S_\rho}e_\eps(v)\,\dd x\\
&\qquad\leq
C\int_{B_\rho}\frac{V(v)}{\eps}\,\dd x .
\end{align*}
The set of common regular values of the finitely many functions
\(r_a\) has full measure.  Hence we can choose
\[
s_*\in\left(\frac{\mu}{4},\frac{\mu}{2}\right)
\]
which is regular for every \(r_a\) and satisfies
\begin{equation}
\label{eq:forced-outer-level-flux}
\eps\sum_{a=1}^q
\int_{\{r_a=s_*\}\cap B_\rho}
|\nabla r_a|\,\dd\cH^1
\leq
C\int_{B_\rho}\frac{V(v)}{\eps}\,\dd x .
\end{equation}

On \(\Pi_a(s_*,\rho)\), one has \(r_a=s_*\leq\mu/2\), and therefore
\[
|v-\sigma_a|\leq C\sqrt{V(v)}.
\]
Thus
\begin{align}
\eps
|\partial_\nu v\cdot(v-\sigma_a)|
&\leq
C\eps|\nabla v|\sqrt{V(v)}
\nonumber\\
&\leq
C\bigl(\eps^2|\nabla v|^2+V(v)\bigr)
=
C\eps e_\eps(v).
\label{eq:forced-boundary-term}
\end{align}

Choose \(\delta>0\) in \eqref{eq:forced-target-force-young}
sufficiently small so that the term
\[
\delta\int\frac{V(v)}{\eps}
\]
can be absorbed into the coercive left-hand side of
\eqref{eq:forced-target-identity}.  Inserting
\eqref{eq:forced-outer-level-flux},
\eqref{eq:forced-boundary-term}, and
\eqref{eq:forced-target-force-young} into
\eqref{eq:forced-target-identity}, and summing over \(a\), gives
\[
\begin{aligned}
E_\eps(v;\Upsilon(s_*,\rho))
\leq C\bigg\{&
\int_{B_\rho}\frac{V(v)}{\eps}\,\dd x
+\eps\int_{\partial B_\rho}e_\eps(v)\,\dd\cH^1\\
&+
\eps^2\int_{B_\rho}\frac{|F|^2}{\eps}\,\dd x
\bigg\}.
\end{aligned}
\]
Since
\[
\Upsilon\left(\frac{\mu}{4},\rho\right)
\subset
\Upsilon(s_*,\rho),
\]
this proves \eqref{eq:forced-fixed-well}.

Moreover,
\[
B_\rho\setminus\Upsilon\left(\frac{\mu}{4},\rho\right)
\subset\Xi_\rho.
\]
Using \eqref{eq:forced-far-energy} on this complement proves
\eqref{eq:forced-full-by-potential}.

It remains to prove the refined estimate.  We first assume that
\(\kappa\) is a common regular value of all functions \(r_a\).  Define
\[
\mathcal A_a
:=
\{\kappa<r_a<s_*\}\cap B_\rho .
\]
The boundary condition \eqref{eq:forced-boundary-well} and the
separation of the wells imply
\[
\mathcal A_a\cap\partial B_\rho=\varnothing .
\]
Indeed, for \(a=1\) one has \(r_1\leq\kappa/4<\kappa\) on
\(\partial B_\rho\), while for \(a\neq1\), the separation of the wells
implies \(r_a>s_*\) on \(\partial B_\rho\).

On \(\mathcal A_a\), \(r_a\geq\kappa>0\), and direct differentiation
gives
\begin{equation}
\label{eq:forced-distance-equation}
\Delta r_a
=
\frac{|\nabla v|^2-|\nabla r_a|^2}{r_a}
+
\frac1{\eps^2}
DV(v)\cdot\frac{v-\sigma_a}{r_a}
-
\frac1\eps
F\cdot\frac{v-\sigma_a}{r_a}.
\end{equation}
The first two terms on the right-hand side are non-negative.  Hence
\[
\Delta r_a\geq-\frac1\eps|F|.
\]
Integrating over \(\mathcal A_a\), using the divergence theorem, yields
\begin{equation}
\label{eq:forced-level-flux-comparison}
\begin{aligned}
\eps
\int_{\{r_a=\kappa\}\cap B_\rho}
|\nabla r_a|\,\dd\cH^1
\leq{}&
\eps
\int_{\{r_a=s_*\}\cap B_\rho}
|\nabla r_a|\,\dd\cH^1\\
&+
\int_{\mathcal A_a}|F|\,\dd x .
\end{aligned}
\end{equation}

Since \(V(v)\geq c\kappa^2\) on \(\mathcal A_a\),
\begin{equation}
\label{eq:forced-transition-measure}
|\mathcal A_a|
\leq
\frac{C\eps}{\kappa^2}
\int_{\mathcal A_a}\frac{V(v)}{\eps}\,\dd x .
\end{equation}
The sets \(\mathcal A_a\) are pairwise disjoint.  Therefore,
\begin{align}
\kappa\sum_{a=1}^q\int_{\mathcal A_a}|F|\,\dd x
&\leq
\kappa
\left(
\int_{B_\rho}\frac{|F|^2}{\eps}\,\dd x
\right)^{1/2}
\left(
\eps\sum_{a=1}^q|\mathcal A_a|
\right)^{1/2}
\nonumber\\
&\leq
C\eps
\left(
\int_{B_\rho}\frac{|F|^2}{\eps}\,\dd x
\right)^{1/2}
\left(
\int_{B_\rho}\frac{V(v)}{\eps}\,\dd x
\right)^{1/2}.
\label{eq:forced-transition-force}
\end{align}

We now use \eqref{eq:forced-target-identity} with \(s=\kappa\).
Multiplying \eqref{eq:forced-level-flux-comparison} by \(\kappa\),
summing over \(a\), and using
\eqref{eq:forced-outer-level-flux} and
\eqref{eq:forced-transition-force}, we obtain
\[
\begin{aligned}
&\eps\kappa
\sum_{a=1}^q
\int_{\{r_a=\kappa\}\cap B_\rho}
|\nabla r_a|\,\dd\cH^1\\
&\qquad\leq
C\kappa\int_{B_\rho}\frac{V(v)}{\eps}\,\dd x
+
C\eps
\left(
\int_{B_\rho}\frac{|F|^2}{\eps}\,\dd x
\right)^{1/2}
\left(
\int_{B_\rho}\frac{V(v)}{\eps}\,\dd x
\right)^{1/2}.
\end{aligned}
\]
The boundary terms on \(\partial B_\rho\) are controlled by
\eqref{eq:forced-boundary-term}, and the bulk forcing terms are
controlled by \eqref{eq:forced-target-force-young}.  Choosing the
constant \(\delta\) in that estimate sufficiently small and absorbing
the resulting potential term, we conclude that
\eqref{eq:forced-refined-well} holds for every common regular value
\(\kappa\).

For a general \(0<\kappa<\mu/4\), choose common regular values
\(\kappa_j\downarrow\kappa\) with \(\kappa_j<\mu/4\).  The boundary
condition remains valid because
\[
|v-\sigma_1|\leq\frac{\kappa}{4}
\leq\frac{\kappa_j}{4}
\qquad\text{on }\partial B_\rho .
\]
Moreover,
\[
\Upsilon(\kappa,\rho)
\subset
\Upsilon(\kappa_j,\rho).
\]
Applying the estimate at \(\kappa_j\) and passing to the limit gives
\eqref{eq:forced-refined-well}.
\end{proof}

\begin{rem}[Correspondence with Bethuel]

For \(F=0\), the integration-by-parts identity
\eqref{eq:forced-target-identity} corresponds to
\cite[Lemmas~4.2--4.5]{Bethuel2025}.  Estimate
\eqref{eq:forced-fixed-well}, together with the estimate on the
complementary region, combines the roles of
\cite[Propositions~4.1 and 4.11]{Bethuel2025}.

The refined estimate \eqref{eq:forced-refined-well} is the forced
counterpart of \cite[Proposition~4.8]{Bethuel2025}.  The monotonicity
of the level-set flux used in Bethuel's proof is replaced here by
\eqref{eq:forced-level-flux-comparison}.  The forcing contributes the
last term in \eqref{eq:forced-distance-equation} and the direct bulk
term in \eqref{eq:forced-target-force}.
\end{rem}

\subsection{The forced Pohozaev identity and \(3/2\)-decay}

\begin{lem}[Forced Pohozaev identity]
\label{lem:forced-pohozaev}

For almost every \(0<r<1\),
\begin{equation}
\label{eq:forced-pohozaev-identity}
\begin{aligned}
\frac2\eps\int_{B_r}V(v)\,\dd x
={}&
r\int_{\partial B_r}
\left\{
\frac{\eps}{2}
\left(
|\nabla_\tau v|^2-|\partial_rv|^2
\right)
+\frac1\eps V(v)
\right\}\dd\cH^1\\
&-
\int_{B_r}F\cdot(x\cdot\nabla v)\,\dd x .
\end{aligned}
\end{equation}
Consequently,
\begin{equation}
\label{eq:forced-pohozaev-bound}
\begin{aligned}
\frac2\eps\int_{B_r}V(v)\,\dd x
\leq{}&
r\int_{\partial B_r}e_\eps(v)\,\dd\cH^1\\
&+
r
\left(
\int_{B_r}\frac{|F|^2}{\eps}\,\dd x
\right)^{1/2}
\left(
2E_\eps(v;B_r)
\right)^{1/2}.
\end{aligned}
\end{equation}
\end{lem}

\begin{proof}

Define
\[
A_{ij}
:=
e_\eps(v)\delta_{ij}
-\eps\,\partial_i v\cdot\partial_jv .
\]
A direct computation gives
\[
\begin{aligned}
\partial_jA_{ij}
&=
\partial_i e_\eps(v)
-\eps\,\partial_{ji}v\cdot\partial_jv
-\eps\,\partial_i v\cdot\Delta v\\
&=
\left(
-\eps\Delta v+\frac1\eps DV(v)
\right)\cdot\partial_i v\\
&=
F\cdot\partial_i v .
\end{aligned}
\]
Since the spatial dimension is two,
\[
\tr A
=
2e_\eps(v)-\eps|\nabla v|^2
=
\frac2\eps V(v).
\]
Therefore,
\[
\partial_j(A_{ij}x_i)
=
\tr A+F\cdot(x\cdot\nabla v).
\]
Integrating over \(B_r\) gives
\begin{equation}
\label{eq:forced-pohozaev-pre}
\frac2\eps\int_{B_r}V(v)\,\dd x
=
\int_{\partial B_r}(Ax)\cdot\nu\,\dd\cH^1
-
\int_{B_r}F\cdot(x\cdot\nabla v)\,\dd x .
\end{equation}

On \(\partial B_r\), \(x=r\nu\), so
\[
\begin{aligned}
(Ax)\cdot\nu
&=
r\left(e_\eps(v)-\eps|\partial_rv|^2\right)\\
&=
r\left\{
\frac{\eps}{2}
\left(
|\nabla_\tau v|^2-|\partial_rv|^2
\right)
+\frac1\eps V(v)
\right\}.
\end{aligned}
\]
Substitution into \eqref{eq:forced-pohozaev-pre} proves
\eqref{eq:forced-pohozaev-identity}.

Finally,
\[
\begin{aligned}
\left|
\int_{B_r}F\cdot(x\cdot\nabla v)\,\dd x
\right|
&\leq
r
\left(
\int_{B_r}\frac{|F|^2}{\eps}\,\dd x
\right)^{1/2}
\left(
\int_{B_r}\eps|\nabla v|^2\,\dd x
\right)^{1/2}\\
&\leq
r
\left(
\int_{B_r}\frac{|F|^2}{\eps}\,\dd x
\right)^{1/2}
\left(
2E_\eps(v;B_r)
\right)^{1/2}.
\end{aligned}
\]
Moreover,
\[
\frac{\eps}{2}
\left(
|\nabla_\tau v|^2-|\partial_rv|^2
\right)
+\frac1\eps V(v)
\leq e_\eps(v).
\]
This proves \eqref{eq:forced-pohozaev-bound}.
\end{proof}

\begin{rem}[Correspondence with Bethuel]

When \(F=0\), identity \eqref{eq:forced-pohozaev-identity} is
\cite[Lemma~3.8, (3.22)]{Bethuel2025}, and
\eqref{eq:forced-pohozaev-bound} corresponds to
\cite[Proposition~3.9, (3.24)]{Bethuel2025}, after adjusting the
normalization of the energy.

The only additional term is
\[
-\int_{B_r}F\cdot(x\cdot\nabla v)\,\dd x,
\]
which is controlled by Cauchy--Schwarz at the action scale
\(\int |F|^2/\eps\).
\end{rem}

\begin{prop}[Forced \(3/2\)-decay]
\label{prop:forced-decay}

There exist \(\eta_{\rm d}=\eta_{\rm d}(V,L)>0\) and
\(C=C(V,L)<\infty\) such that the following holds.

Assume \(0<\eps\leq1\), \(v\) satisfies
\eqref{eq:forced-uniform-bound}, and
\[
E_\eps(v;B_1)\leq\eta_{\rm d}.
\]
Then
\begin{equation}
\label{eq:forced-decay}
\begin{aligned}
E_\eps\left(v;B_{1/2}\right)
\leq C\bigg\{&
E_\eps(v;B_1)^{3/2}
+\eps E_\eps(v;B_1)\\
&+
\left(
\int_{B_1}\frac{|F|^2}{\eps}\,\dd x
\right)^{1/2}
E_\eps(v;B_1)^{1/2}\\
&+
\eps^2\int_{B_1}\frac{|F|^2}{\eps}\,\dd x
\bigg\}.
\end{aligned}
\end{equation}
\end{prop}

\begin{proof}

Set
\[
E:=E_\eps(v;B_1),
\qquad
D:=\int_{B_1}\frac{|F|^2}{\eps}\,\dd x .
\]
If \(E=0\), then \(\nabla v=0\) and \(V(v)=0\) almost everywhere.
By continuity, \(v\equiv\sigma\) for some \(\sigma\in\Sigma\), and the
claim is immediate.  We therefore assume \(E>0\).

By \eqref{eq:forced-circle-average}, there exists
\(\rho\in[3/4,4/5]\) such that
\begin{equation}
\label{eq:forced-outer-good-circle}
\int_{\partial B_\rho}e_\eps(v)\,\dd\cH^1
\leq20E.
\end{equation}
If \(\eta_{\rm d}\) is sufficiently small, Lemma~\ref{lem:forced-circle}
gives \(\sigma\in\Sigma\) such that
\[
\sup_{\partial B_\rho}|v-\sigma|
\leq C\sqrt E .
\]
Fix \(K\geq4C\) and set
\[
\kappa:=K\sqrt E .
\]
After decreasing \(\eta_{\rm d}\), we may assume
\[
0<\kappa<\frac{\mu}{4},
\qquad
|v-\sigma|\leq\frac{\kappa}{4}
\quad\text{on }\partial B_\rho .
\]

Applying \eqref{eq:forced-refined-well}, we obtain
\begin{equation}
\label{eq:forced-good-region-energy}
\begin{aligned}
E_\eps(v;\Upsilon(\kappa,\rho))
\leq C\bigg\{&
E^{3/2}+\eps E\\
&+
\eps D^{1/2}E^{1/2}
+\eps^2D
\bigg\}.
\end{aligned}
\end{equation}

Define
\[
w:=\chi_\sigma\circ v.
\]
Since \(w\leq\kappa/4\) on \(\partial B_\rho\), all level sets
\[
\{w=s\},
\qquad
s\in[\kappa/2,3\kappa/4],
\]
are compactly contained in \(B_\rho\).  On the set
\(\{\kappa/2<w<3\kappa/4\}\),
\[
|\nabla w|
\leq
\frac1\kappa|\nabla(w^2)|.
\]
By \eqref{eq:forced-modica-mortola} and the coarea formula,
\[
\begin{aligned}
\int_{\kappa/2}^{3\kappa/4}
\cH^1(\{w=s\}\cap B_\rho)\,\dd s
&=
\int_{\{\kappa/2<w<3\kappa/4\}}|\nabla w|\,\dd x\\
&\leq
\frac{C}{\kappa}
\int_{B_\rho}e_\eps(v)\,\dd x\\
&\leq
\frac{CE}{\kappa}.
\end{aligned}
\]
Since the interval \([\kappa/2,3\kappa/4]\) has length \(\kappa/4\),
there exists a regular value
\[
s_0\in(\kappa/2,3\kappa/4)
\]
such that
\begin{equation}
\label{eq:forced-short-level}
\cH^1(\{w=s_0\}\cap B_\rho)
\leq
\frac{CE}{\kappa^2}
=
\frac{C}{K^2}.
\end{equation}
Choose \(K\) sufficiently large that \(C/K^2<1/64\).

For every \(h\in(3\kappa/4,\kappa)\), Lemma~\ref{lem:planar-shadow}
applied with \(s=s_0\) gives
\[
\mathcal L^1
\left(
\left\{
r\in[5/8,\rho]:
\sup_{\partial B_r}w>h
\right\}
\right)
\leq
\frac12
\cH^1(\{w=s_0\}\cap B_\rho)
<\frac1{128}.
\]
Choose a sequence \(h_j\uparrow\kappa\).  Since the corresponding
sets decrease,
\[
\mathcal L^1
\left(
\left\{
r\in[5/8,\rho]:
\sup_{\partial B_r}w\geq\kappa
\right\}
\right)
\leq\frac1{128}.
\]
Therefore
\[
\mathcal I
:=
\left\{
r\in[5/8,\rho]:
\partial B_r\subset\{w<\kappa\}
\right\}
\]
satisfies
\[
|\mathcal I|
\geq
(\rho-5/8)-\frac1{128}
\geq
\frac1{8}-\frac1{128}
=
\frac{15}{128}.
\]

For every \(r\in\mathcal I\),
\[
\partial B_r\subset B_\kappa(\sigma)
\subset\Upsilon(\kappa,\rho).
\]
By Fubini's theorem,
\[
\int_{\mathcal I}
\left(
\int_{\partial B_r}e_\eps(v)\,\dd\cH^1
\right)\dd r
\leq
E_\eps(v;\Upsilon(\kappa,\rho)).
\]
Hence there exists \(\tau\in\mathcal I\) such that
\begin{equation}
\label{eq:forced-good-inner-circle}
\begin{aligned}
\int_{\partial B_\tau}e_\eps(v)\,\dd\cH^1
\leq C\bigg\{&
E^{3/2}+\eps E\\
&+
\eps D^{1/2}E^{1/2}
+\eps^2D
\bigg\}.
\end{aligned}
\end{equation}

Applying \eqref{eq:forced-pohozaev-bound} on \(B_\tau\), using
\(\tau\leq1\), \(\tau\geq5/8\), and
\(E_\eps(v;B_\tau)\leq E\), gives
\[
\begin{aligned}
\int_{B_\tau}\frac{V(v)}{\eps}\,\dd x
\leq C\bigg\{&
\int_{\partial B_\tau}e_\eps(v)\,\dd\cH^1
+D^{1/2}E^{1/2}
\bigg\}.
\end{aligned}
\]
Since \(\eps\leq1\), the mixed term in
\eqref{eq:forced-good-inner-circle} is bounded by
\(D^{1/2}E^{1/2}\).  Hence
\begin{equation}
\label{eq:forced-inner-potential}
\begin{aligned}
\int_{B_{5/8}}\frac{V(v)}{\eps}\,\dd x
\leq C\bigg\{&
E^{3/2}+\eps E\\
&+
D^{1/2}E^{1/2}
+\eps^2D
\bigg\}.
\end{aligned}
\end{equation}

Finally, by \eqref{eq:forced-circle-average}, choose
\(r_*\in[9/16,5/8]\) such that
\[
\int_{\partial B_{r_*}}e_\eps(v)\,\dd\cH^1
\leq
16E_\eps\left(
v;B_{5/8}\setminus B_{9/16}
\right)
\leq16E .
\]
Since \(B_{1/2}\subset B_{r_*}\), applying
\eqref{eq:forced-full-by-potential} on \(B_{r_*}\), and using
\eqref{eq:forced-inner-potential}, proves \eqref{eq:forced-decay}.
\end{proof}

\begin{rem}[Correspondence with Bethuel]

Proposition~\ref{prop:forced-decay} is the forced analogue of
\cite[Proposition~1.12]{Bethuel2025}.  Its proof follows the same
three ingredients used by Bethuel:

\begin{enumerate}
\item the refined near-well estimate from
\cite[Proposition~4.8]{Bethuel2025};

\item the construction of a good inner circle through
\cite[Proposition~2.9 and Lemma~2.10]{Bethuel2025};

\item the Pohozaev estimate from
\cite[Proposition~3.9]{Bethuel2025}.
\end{enumerate}

The forcing errors are tracked explicitly in
\eqref{eq:forced-good-region-energy},
\eqref{eq:forced-good-inner-circle}, and
\eqref{eq:forced-inner-potential}.
\end{rem}

\subsection{Scaling and dyadic clearing}

Let \(x_0\in B_1\), let \(R>0\), and assume
\(B_R(x_0)\subset B_1\).  Define
\[
v_R(y):=v(x_0+Ry),
\qquad
F_R(y):=RF(x_0+Ry),
\qquad
\eps_R:=\frac{\eps}{R}.
\]
Then
\[
-\eps_R\Delta v_R+\frac1{\eps_R}DV(v_R)=F_R
\qquad\text{in }B_1.
\]
Moreover,
\begin{equation}
\label{eq:forced-scaling}
E_{\eps_R}(v_R;B_1)
=
\frac1R E_\eps(v;B_R(x_0)),
\end{equation}
and
\begin{equation}
\label{eq:forced-action-scaling}
\int_{B_1}\frac{|F_R|^2}{\eps_R}\,\dd y
=
R\int_{B_R(x_0)}\frac{|F|^2}{\eps}\,\dd x .
\end{equation}

Whenever \(\eps_R\leq1\) and the energy smallness condition in
Proposition~\ref{prop:forced-decay} is satisfied, applying that
proposition to \(v_R\) gives
\begin{equation}
\label{eq:forced-scaled-decay}
\begin{aligned}
E_\eps\left(v;B_{R/2}(x_0)\right)
\leq C\bigg\{&
\frac{E_\eps(v;B_R(x_0))^{3/2}}{\sqrt R}
+\frac{\eps}{R}E_\eps(v;B_R(x_0))\\
&+
R
\left(
\int_{B_R(x_0)}\frac{|F|^2}{\eps}\,\dd x
\right)^{1/2}
E_\eps(v;B_R(x_0))^{1/2}\\
&+
\eps^2
\int_{B_R(x_0)}\frac{|F|^2}{\eps}\,\dd x
\bigg\}.
\end{aligned}
\end{equation}

\begin{rem}[Correspondence with Bethuel]

The scaling relations \eqref{eq:forced-scaling} and
\eqref{eq:forced-scaled-decay} correspond to
\cite[(1.52)--(1.55) and (1.61)]{Bethuel2025}, up to the choice of the
inner radius.  Formula \eqref{eq:forced-action-scaling} is the additional
scaling identity required for the parabolic forcing.
\end{rem}

\begin{lem}[Weak clearing-out]
\label{lem:forced-weak-clearing}

There exists \(\eta_{\rm w}=\eta_{\rm w}(V,L)>0\) such that, whenever
\[
0<\eps\leq4,
\qquad
\|v\|_{L^\infty(B_1)}
+\eps\|\nabla v\|_{L^\infty(B_1)}
\leq L,
\]
and
\[
E_\eps(v;B_1)\leq\eta_{\rm w}\eps,
\]
there exists \(\sigma\in\Sigma\) such that
\begin{equation}
\label{eq:forced-weak-confinement}
v(B_{3/4})\subset B_{\mu/2}(\sigma).
\end{equation}
\end{lem}

\begin{proof}

Suppose, by contradiction, that there exists \(x_0\in B_{3/4}\) such that
\[
\dist(v(x_0),\Sigma)\geq\frac{\mu}{2}.
\]
Since
\[
\|\nabla v\|_{L^\infty(B_1)}
\leq\frac{L}{\eps},
\]
we have
\[
|v(x)-v(x_0)|
\leq
\frac{L}{\eps}|x-x_0|.
\]
Choose
\[
r_0:=c_*(V,L,\mu)\min\{\eps,1\}
\]
with \(c_*>0\) sufficiently small that
\(r_0\leq1/8\) and
\[
\frac{L}{\eps}r_0\leq\frac{\mu}{4}.
\]
Then
\[
B_{r_0}(x_0)\subset B_1
\]
and
\[
\dist(v(x),\Sigma)\geq\frac{\mu}{4}
\qquad\text{for }x\in B_{r_0}(x_0).
\]
By \eqref{eq:forced-potential-away},
\[
V(v(x))\geq c_0(V,L,\mu)
\qquad\text{on }B_{r_0}(x_0).
\]
Consequently,
\[
\begin{aligned}
E_\eps(v;B_1)
&\geq
\frac1\eps
\int_{B_{r_0}(x_0)}V(v)\,\dd x\\
&\geq
\frac{c_0\pi r_0^2}{\eps}\\
&\geq
\begin{cases}
c_1\eps,&0<\eps\leq1,\\[1mm]
c_1\eps^{-1},&1\leq\eps\leq4.
\end{cases}
\end{aligned}
\]
For \(1\leq\eps\leq4\),
\[
\eps^{-1}\geq\frac{\eps}{16}.
\]
Thus in both cases
\[
E_\eps(v;B_1)\geq c_2\eps.
\]
Choosing \(\eta_{\rm w}<c_2\) gives a contradiction.

Therefore
\[
v(B_{3/4})
\subset
\bigcup_{a=1}^q B_{\mu/2}(\sigma_a).
\]
The balls are pairwise disjoint, \(B_{3/4}\) is connected, and \(v\) is
continuous.  Hence \(v(B_{3/4})\) is contained in one of the balls,
which proves \eqref{eq:forced-weak-confinement}.
\end{proof}

\begin{rem}[Correspondence with Bethuel]

Lemma~\ref{lem:forced-weak-clearing} corresponds to the weak
clearing-out result \cite[Proposition~6.1]{Bethuel2025}.  Under the
uniform scaled gradient bound, the proof is an elementary ball
argument.  The forcing does not enter this step.
\end{rem}

\begin{lem}[Clearing-out at the centre]
\label{lem:forced-dyadic-centre}

There exist constants
\[
\eta_{\rm c}=\eta_{\rm c}(V,L)>0,
\qquad
\gamma_{\rm c}=\gamma_{\rm c}(V,L)>0
\]
such that, whenever
\[
0<\eps\leq4,
\qquad
\|v\|_{L^\infty(B_1)}
+\eps\|\nabla v\|_{L^\infty(B_1)}
\leq L,
\]
and
\[
E_\eps(v;B_1)\leq\eta_{\rm c},
\qquad
\int_{B_1}\frac{|F|^2}{\eps}\,\dd x\leq\gamma_{\rm c},
\]
one has
\[
\dist(v(0),\Sigma)\leq\frac{\mu}{2}.
\]
\end{lem}

\begin{proof}

We first consider \(0<\eps\leq\eps_0\), where
\(\eps_0\in(0,1)\) will be chosen below.

Set
\[
r_n:=2^{-n},
\qquad
E_n:=E_\eps(v;B_{r_n}),
\qquad
a_n:=\frac{E_n}{r_n},
\]
and
\[
g_n
:=
r_n\int_{B_{r_n}}\frac{|F|^2}{\eps}\,\dd x,
\qquad
\eps_n:=\frac{\eps}{r_n}.
\]
As long as \(r_n\geq\sqrt\eps\),
\[
\eps_n\leq\sqrt\eps\leq\sqrt{\eps_0},
\qquad
g_n\leq\gamma_{\rm c}r_n.
\]

The rescaled uniform estimate is unchanged:
\[
\|v_{r_n}\|_{L^\infty(B_1)}
+
\eps_n\|\nabla v_{r_n}\|_{L^\infty(B_1)}
=
\|v\|_{L^\infty(B_{r_n})}
+
\eps\|\nabla v\|_{L^\infty(B_{r_n})}
\leq L.
\]
Furthermore,
\[
E_{\eps_n}(v_{r_n};B_1)=a_n.
\]
Choose
\[
\bar\eta\leq\eta_{\rm d}.
\]
Whenever \(a_n\leq\bar\eta\), Proposition~\ref{prop:forced-decay}
applies to the rescaled map and gives
\begin{equation}
\label{eq:forced-normalized-recurrence}
a_{n+1}
\leq
C\left\{
a_n^{3/2}
+\eps_n a_n
+\sqrt{g_na_n}
+\eps_n^2g_n
\right\}.
\end{equation}

Choose \(\bar\eta>0\) and \(\eps_0>0\) so small that
\[
C\sqrt{\bar\eta}\leq\frac1{16},
\qquad
C\sqrt{\eps_0}\leq\frac1{16}.
\]
If \(a_n\leq\bar\eta\), then
\[
Ca_n^{3/2}\leq\frac1{16}a_n,
\qquad
C\eps_n a_n\leq\frac1{16}a_n.
\]
By Young's inequality,
\[
C\sqrt{g_na_n}
\leq
\frac18a_n+C_1g_n.
\]
Since \(r_n\geq\sqrt\eps\),
\[
\eps_n^2=\frac{\eps^2}{r_n^2}\leq1,
\]
and therefore
\[
\eps_n^2g_n\leq g_n\leq\gamma_{\rm c}r_n.
\]
Thus
\begin{equation}
\label{eq:forced-linear-recurrence}
a_{n+1}
\leq
\frac14a_n+C_2\gamma_{\rm c}r_n
\end{equation}
whenever \(r_n\geq\sqrt\eps\) and \(a_n\leq\bar\eta\).

Choose
\[
\eta_{\rm c}\leq\bar\eta,
\qquad
C_2\gamma_{\rm c}\leq\frac{\bar\eta}{4}.
\]
Then \eqref{eq:forced-linear-recurrence} implies inductively that
\(a_n\leq\bar\eta\) for every index for which \(r_n\geq\sqrt\eps\).

Iterating \eqref{eq:forced-linear-recurrence}, we obtain
\[
a_n
\leq
4^{-n}a_0
+
C_2\gamma_{\rm c}
\sum_{j=0}^{n-1}4^{-(n-1-j)}2^{-j}.
\]
Since \(r_n=2^{-n}\),
\[
4^{-n}=r_n^2,
\qquad
\sum_{j=0}^{n-1}4^{-(n-1-j)}2^{-j}
\leq C r_n.
\]
Therefore
\begin{equation}
\label{eq:forced-dyadic-bound}
a_n\leq r_n^2a_0+C_3\gamma_{\rm c}r_n,
\qquad
E_n\leq a_0r_n^3+C_3\gamma_{\rm c}r_n^2.
\end{equation}

Choose \(N\) so that
\[
\sqrt\eps\leq r_N<2\sqrt\eps .
\]
Then \eqref{eq:forced-dyadic-bound} gives
\[
E_N
\leq
8a_0\eps^{3/2}
+
4C_3\gamma_{\rm c}\eps .
\]
Since \(a_0=E_\eps(v;B_1)\leq\eta_{\rm c}\), we can choose
\(\eps_0,\eta_{\rm c},\gamma_{\rm c}\) so small that
\[
8\eta_{\rm c}\sqrt{\eps_0}
\leq\frac{\eta_{\rm w}}2,
\qquad
4C_3\gamma_{\rm c}
\leq\frac{\eta_{\rm w}}2.
\]
Consequently,
\[
E_N\leq\eta_{\rm w}\eps.
\]

Rescale \(B_{r_N}\) to \(B_1\).  The rescaled parameter is
\[
\widetilde\eps:=\frac{\eps}{r_N}.
\]
Moreover,
\[
E_{\widetilde\eps}(v_{r_N};B_1)
=
\frac{E_N}{r_N}
\leq
\eta_{\rm w}\frac{\eps}{r_N}
=
\eta_{\rm w}\widetilde\eps.
\]
Since \(r_N\geq\sqrt\eps\),
\[
0<\widetilde\eps\leq\sqrt\eps\leq1\leq4.
\]
Lemma~\ref{lem:forced-weak-clearing} therefore applies and yields
\[
\dist(v(0),\Sigma)
=
\dist(v_{r_N}(0),\Sigma)
\leq\frac{\mu}{2}.
\]

It remains to consider \(\eps\geq\eps_0\).  Choose
\[
\eta_{\rm c}\leq\eta_{\rm w}\eps_0.
\]
Then
\[
E_\eps(v;B_1)
\leq\eta_{\rm c}
\leq\eta_{\rm w}\eps,
\]
so Lemma~\ref{lem:forced-weak-clearing} applies directly.
\end{proof}

\begin{rem}[Correspondence with Bethuel]

Lemma~\ref{lem:forced-dyadic-centre} is the forced counterpart of
\cite[Proposition~6.5]{Bethuel2025}.  Bethuel iterates the scaled
\(3/2\)-decay estimate until the weak clearing criterion of
\cite[Proposition~6.1]{Bethuel2025} becomes applicable.

The normalized recurrence \eqref{eq:forced-normalized-recurrence}
contains the two additional forcing contributions
\[
\sqrt{g_na_n}
\qquad\text{and}\qquad
\eps_n^2g_n.
\]
The definition
\[
g_n=r_n\int_{B_{r_n}}\frac{|F|^2}{\eps}\,\dd x
\]
makes both terms compatible with the summable error in
\eqref{eq:forced-linear-recurrence}.
\end{rem}

\subsection{Completion of the clearing-out theorem}

\begin{proof}[Proof of Theorem~\ref{thm:forced-clearing}]

Let \(\eta_{\rm c}\) and \(\gamma_{\rm c}\) be the constants from
Lemma~\ref{lem:forced-dyadic-centre}.  Decrease
\(\eta_{\rm f}\) and \(\gamma_{\rm f}\), if necessary, so that
\begin{equation}
\label{eq:forced-final-smallness-choice}
8\eta_{\rm f}\leq\eta_{\rm c},
\qquad
\frac{\gamma_{\rm f}}4\leq\gamma_{\rm c}.
\end{equation}

The estimate \eqref{eq:forced-unit-decay} is exactly the conclusion of
Proposition~\ref{prop:forced-decay}.  It remains to prove confinement and
the strong inner estimate.

\medskip
\noindent
\textbf{Step 1: Pointwise confinement.}

Fix \(x_0\in B_{3/4}\).  Then
\[
\dist(x_0,\partial B_1)\geq\frac14,
\]
hence
\[
B_{1/4}(x_0)\subset B_1.
\]

Define
\[
\widetilde v(y)
:=
v\left(x_0+\frac y4\right),
\qquad
\widetilde\eps:=4\eps,
\qquad
\widetilde F(y)
:=
\frac14F\left(x_0+\frac y4\right).
\]
Since
\[
\nabla_y\widetilde v(y)
=
\frac14\nabla_xv\left(x_0+\frac y4\right),
\qquad
\Delta_y\widetilde v(y)
=
\frac1{16}\Delta_xv\left(x_0+\frac y4\right),
\]
we have
\[
-\widetilde\eps\Delta_y\widetilde v
+\frac1{\widetilde\eps}DV(\widetilde v)
=
\widetilde F
\qquad\text{in }B_1.
\]

The energy scaling is
\[
\begin{aligned}
E_{\widetilde\eps}(\widetilde v;B_1)
&=
\int_{B_1}
\left(
\frac{\widetilde\eps}{2}|\nabla_y\widetilde v|^2
+\frac1{\widetilde\eps}V(\widetilde v)
\right)\dd y\\
&=
4\int_{B_{1/4}(x_0)}
\left(
\frac{\eps}{2}|\nabla_xv|^2
+\frac1\eps V(v)
\right)\dd x\\
&=
4E_\eps(v;B_{1/4}(x_0))\\
&\leq
4E_\eps(v;B_1)
\leq8\eta_{\rm f}
\leq\eta_{\rm c}.
\end{aligned}
\]

For the forcing term, using \(\dd y=16\,\dd x\),
\[
\begin{aligned}
\int_{B_1}
\frac{|\widetilde F(y)|^2}{\widetilde\eps}\,\dd y
&=
\int_{B_1}
\frac{\frac1{16}|F(x_0+y/4)|^2}{4\eps}\,\dd y\\
&=
\frac14
\int_{B_{1/4}(x_0)}
\frac{|F(x)|^2}{\eps}\,\dd x\\
&\leq
\frac{\gamma_{\rm f}}4
\leq\gamma_{\rm c}.
\end{aligned}
\]

The uniform bound is invariant under this scaling:
\[
\begin{aligned}
\|\widetilde v\|_{L^\infty(B_1)}
+
\widetilde\eps
\|\nabla_y\widetilde v\|_{L^\infty(B_1)}
&=
\|v\|_{L^\infty(B_{1/4}(x_0))}
+
\eps\|\nabla_xv\|_{L^\infty(B_{1/4}(x_0))}\\
&\leq L.
\end{aligned}
\]
Also,
\[
0<\widetilde\eps=4\eps\leq4.
\]
Thus all assumptions of Lemma~\ref{lem:forced-dyadic-centre} hold for
\((\widetilde v,\widetilde\eps,\widetilde F)\).  We obtain
\[
\dist(\widetilde v(0),\Sigma)\leq\frac{\mu}{2}.
\]
Since \(\widetilde v(0)=v(x_0)\),
\[
\dist(v(x_0),\Sigma)\leq\frac{\mu}{2}.
\]

As \(x_0\in B_{3/4}\) was arbitrary,
\[
v(B_{3/4})
\subset
\bigcup_{a=1}^q\overline{B_{\mu/2}(\sigma_a)}.
\]
The closed balls are pairwise disjoint, while \(v(B_{3/4})\) is
connected because \(B_{3/4}\) is connected and \(v\) is continuous.
Therefore one \(\sigma\in\Sigma\) satisfies
\[
|v-\sigma|\leq\frac{\mu}{2}
\qquad\text{on }B_{3/4}.
\]
This proves \eqref{eq:forced-confinement}.

\medskip
\noindent
\textbf{Step 2: Caccioppoli estimate after confinement.}

Set
\[
w:=v-\sigma.
\]
By confinement and \eqref{eq:forced-well-coercivity},
\[
c_V|w|^2
\leq V(v)\leq C_V|w|^2,
\qquad
DV(v)\cdot w\geq c_V|w|^2
\qquad\text{on }B_{3/4}.
\]

Choose
\[
\chi\in C_c^\infty(B_{3/4})
\]
such that
\[
0\leq\chi\leq1,
\qquad
\chi\equiv1\quad\text{on }B_{5/8},
\qquad
|\nabla\chi|\leq C.
\]

Multiply the equation by \(\chi^2w\) and integrate over \(B_{3/4}\).
Since \(\nabla w=\nabla v\), integration by parts gives
\[
\begin{aligned}
&\int_{B_{3/4}}
\chi^2
\left(
\eps|\nabla v|^2
+\frac1\eps DV(v)\cdot w
\right)\dd x\\
&\qquad=
-2\eps
\int_{B_{3/4}}
\chi\,\nabla v:(w\otimes\nabla\chi)\,\dd x
+
\int_{B_{3/4}}
\chi^2F\cdot w\,\dd x .
\end{aligned}
\]

For the cutoff term,
\[
\begin{aligned}
2\eps\chi|\nabla v|\,|w|\,|\nabla\chi|
&\leq
\frac{\eps}{4}\chi^2|\nabla v|^2
+
C\eps|\nabla\chi|^2|w|^2 .
\end{aligned}
\]
For the forcing term, Young's inequality gives
\[
\chi^2|F|\,|w|
\leq
\frac{c_V}{2\eps}\chi^2|w|^2
+
\frac{\eps}{2c_V}\chi^2|F|^2 .
\]
Using
\[
DV(v)\cdot w\geq c_V|w|^2,
\]
we absorb the terms
\[
\frac{\eps}{4}\chi^2|\nabla v|^2
\qquad\text{and}\qquad
\frac{c_V}{2\eps}\chi^2|w|^2
\]
into the left-hand side.  Hence
\[
\int_{B_{3/4}}
\chi^2
\left(
\eps|\nabla v|^2+\frac1\eps|w|^2
\right)\dd x
\leq
C\eps
\int_{B_{3/4}}|\nabla\chi|^2|w|^2\,\dd x
+
C\eps\int_{B_{3/4}}|F|^2\,\dd x .
\]

Since \(|w|^2\leq c_V^{-1}V(v)\),
\[
\begin{aligned}
\eps\int_{B_{3/4}}|\nabla\chi|^2|w|^2\,\dd x
&\leq
C\eps\int_{B_1}V(v)\,\dd x\\
&=
C\eps^2
\int_{B_1}\frac{V(v)}{\eps}\,\dd x\\
&\leq
C\eps^2E_\eps(v;B_1).
\end{aligned}
\]
Furthermore,
\[
\eps\int_{B_{3/4}}|F|^2\,\dd x
=
\eps^2
\int_{B_{3/4}}\frac{|F|^2}{\eps}\,\dd x
\leq
\eps^2
\int_{B_1}\frac{|F|^2}{\eps}\,\dd x .
\]
Thus
\[
\int_{B_{3/4}}
\chi^2
\left(
\eps|\nabla v|^2+\frac1\eps|w|^2
\right)\dd x
\leq
C\eps^2
\left\{
E_\eps(v;B_1)
+
\int_{B_1}\frac{|F|^2}{\eps}\,\dd x
\right\}.
\]

Finally, because \(\chi\equiv1\) on \(B_{5/8}\) and
\(V(v)\leq C_V|w|^2\),
\[
\begin{aligned}
E_\eps(v;B_{5/8})
&=
\int_{B_{5/8}}
\left(
\frac{\eps}{2}|\nabla v|^2
+\frac1\eps V(v)
\right)\dd x\\
&\leq
C
\int_{B_{3/4}}
\chi^2
\left(
\eps|\nabla v|^2+\frac1\eps|w|^2
\right)\dd x\\
&\leq
C_{\rm f}\eps^2
\left\{
E_\eps(v;B_1)
+
\int_{B_1}\frac{|F|^2}{\eps}\,\dd x
\right\}.
\end{aligned}
\]
This proves \eqref{eq:forced-strong-inner}.
\end{proof}

\begin{rem}[Final comparison with Bethuel]

The proof follows the architecture of Sections~5 and 6 of
\cite{Bethuel2025}:

\begin{enumerate}
\item Proposition~\ref{prop:forced-decay} is the forced version of
\cite[Proposition~1.12]{Bethuel2025};

\item the dyadic iteration in
Lemma~\ref{lem:forced-dyadic-centre} corresponds to
\cite[Proposition~6.5]{Bethuel2025};

\item the passage from pointwise clearing to confinement in one well
corresponds to the confinement part of
\cite[Theorem~1.11]{Bethuel2025};

\item the final Caccioppoli estimate is the forced counterpart of the
near-well estimate in \cite[Proposition~6.8]{Bethuel2025}.
\end{enumerate}

The geometric and elliptic core is therefore the same as in Bethuel's
argument.  The only genuinely new terms are those produced by \(F\),
all of which are controlled by the scale-invariant quantity
\[
\int\frac{|F|^2}{\eps}.
\]
\end{rem}

\section{A forced Bethuel-type theorem}
\label{sec:forced-bethuel}

The purpose of this section is to pass from the quantitative clearing-out
theorem of Section~\ref{sec:forced-clearing} to a measure-theoretic
compactness result for forced critical points.  The conclusions are the
forced analogues of the rectifiability, absolute-continuity, and discrepancy
statements in \cite[Theorems~1.2, 1.8, and 1.9]{Bethuel2025}.  The proof
given here differs from Bethuel's proof at the measure-theoretic stage:
once clearing-out has supplied a positive one-dimensional lower density, we
use the structure theorem of De Philippis and Rindler for
divergence-constrained measures and Allard's rectifiability theorem.

Throughout this section,
\[
 e_\eps(v)
 =
 \frac{\eps}{2}|\nabla v|^2+\frac1\eps V(v),
 \qquad
 E_\eps(v;A)
 =
 \int_A e_\eps(v)\,\dd x,
\]
as in Section~\ref{sec:forced-clearing}.

\subsection{Energy--potential comparison}

We first record a local estimate comparing the full energy with the
potential energy.  Unlike the clearing theorem, this estimate does not
require any smallness assumption.

\begin{lem}[Energy--potential comparison]
\label{lem:forced-energy-potential}
Let \(0<\eps\leq1\), and let
\(v\in C^2(B_1;\R^k)\) and \(F\in L^2(B_1;\R^k)\) satisfy
\[
 -\eps\Delta v+\frac1\eps DV(v)=F
 \qquad\text{in }B_1
\]
and
\[
 \|v\|_{L^\infty(B_1)}
 +\eps\|\nabla v\|_{L^\infty(B_1)}
 \leq L.
\]
Then
\begin{equation}\label{eq:forced-energy-potential}
\begin{aligned}
 E_\eps\left(v;B_{\frac12}\right)
 \leq C(V,L)\bigg\{&
 \int_{B_{\frac34}}\frac{V(v)}{\eps}\,\dd x
 +\eps E_\eps\left(
 v;B_{\frac34}\setminus B_{\frac12}
 \right)\\
 &+\eps^2
 \int_{B_1}\frac{|F|^2}{\eps}\,\dd x
 \bigg\}.
\end{aligned}
\end{equation}
\end{lem}

\begin{proof}
By the coarea formula,
\[
 \int_{\frac{9}{16}}^{\frac58}
 \left(
 \int_{\partial B_s}e_\eps(v)\,\dd\cH^1
 \right)\dd s
 =
 E_\eps\left(
 v;B_{\frac58}\setminus B_{\frac{9}{16}}
 \right).
\]

Hence there exists
\(\rho\in[\frac{9}{16},\frac58]\) such that
\begin{equation}\label{eq:energy-potential-good-radius}
 \int_{\partial B_\rho}e_\eps(v)\,\dd\cH^1
 \leq
 16E_\eps\left(
 v;B_{\frac58}\setminus B_{\frac{9}{16}}
 \right).
\end{equation}
Because \(B_{1/2}\subset B_\rho\), estimate
\eqref{eq:forced-full-by-potential}, applied on \(B_\rho\), gives
\[
\begin{aligned}
 E_\eps\left(v;B_{\frac12}\right)
 \leq C(V,L)\bigg\{&
 \int_{B_\rho}\frac{V(v)}{\eps}\,\dd x
 +\eps\int_{\partial B_\rho}e_\eps(v)\,\dd\cH^1\\
 &+\eps^2
 \int_{B_\rho}\frac{|F|^2}{\eps}\,\dd x
 \bigg\}.
\end{aligned}
\]
Now
\[
 B_\rho\subset B_{\frac34},
 \qquad
 B_{\frac58}\setminus B_{\frac{9}{16}}
 \subset
 B_{\frac34}\setminus B_{\frac12},
\]
so \eqref{eq:energy-potential-good-radius} yields
\[
 \eps\int_{\partial B_\rho}e_\eps(v)\,\dd\cH^1
 \leq
 16\eps E_\eps\left(
 v;B_{\frac34}\setminus B_{\frac12}
 \right).
\]
The forcing integral over \(B_\rho\) is bounded by the corresponding
integral over \(B_1\).  Substitution proves
\eqref{eq:forced-energy-potential}.
\end{proof}

\begin{rem}[Correspondence with Bethuel]
When \(F=0\), Lemma~\ref{lem:forced-energy-potential} is the local
energy--potential comparison established in
\cite[Proposition~4.11]{Bethuel2025}; see also the bounded-energy variant
\cite[Proposition~4.13]{Bethuel2025}.  The additional term
\[
 \eps^2\int_{B_1}\frac{|F|^2}{\eps}\,\dd x
\]
comes from the forced near-well estimate
\eqref{eq:forced-full-by-potential}.
\end{rem}

We shall use the following rescaled form.  Let \(B_r(x_0)\) be compactly
contained in the domain of \(v\).  Applying
Lemma~\ref{lem:forced-energy-potential} to
\[
 v_r(y):=v(x_0+ry),
 \qquad
 \eps_r:=\frac{\eps}{r},
 \qquad
 F_r(y):=rF(x_0+ry),
\]
and using
\[
 E_{\eps_r}(v_r;A)
 =
 \frac1r E_\eps(v;x_0+rA),
\]
as well as
\[
 \int_{B_1}\frac{|F_r|^2}{\eps_r}\,\dd y
 =
 r\int_{B_r(x_0)}\frac{|F|^2}{\eps}\,\dd x,
\]
we obtain
\begin{equation}\label{eq:forced-energy-potential-scaled}
\begin{aligned}
 E_\eps\left(v;B_{\frac r2}(x_0)\right)
 \leq C(V,L)\bigg\{&
 \int_{B_{\frac{3r}{4}}(x_0)}
 \frac{V(v)}{\eps}\,\dd x\\
 &+\frac{\eps}{r}
 E_\eps\left(
 v;
 B_{\frac{3r}{4}}(x_0)\setminus B_{\frac r2}(x_0)
 \right)\\
 &+\eps^2
 \int_{B_r(x_0)}\frac{|F|^2}{\eps}\,\dd x
 \bigg\}.
\end{aligned}
\end{equation}

\subsection{A rank-one stress lemma}

We next isolate the measure-theoretic mechanism that converts the limiting
stress into a rectifiable varifold.  This is the only point in the argument
where the fact that the spatial dimension is two is essential.

\begin{thm}[De Philippis--Rindler structure theorem for $\mathcal A$-free measures]
\label{thm:dpr-structure}
Let $\Omega\subset\mathbb R^d$ be open and let
\[
\mathcal A=\sum_{|\alpha|\le k}A_\alpha\partial^\alpha
\]
be a linear constant-coefficient differential operator. Let
\[
\mu\in\mathcal M(\Omega;\mathbb R^N),
\qquad
\mathcal A\mu=0
\quad\text{in }\mathcal D'(\Omega).
\]
Write $\mu=\mu^a+\mu^s$ for the Lebesgue decomposition with respect to
$\mathcal L^d$. Define the principal symbol and wave cone by
\[
\mathbb A^k(\xi)
:=
\sum_{|\alpha|=k}A_\alpha\xi^\alpha,
\qquad
\Lambda_{\mathcal A}
:=
\bigcup_{\xi\in\mathbb R^d\setminus\{0\}}
\ker\mathbb A^k(\xi).
\]
Then
\[
\frac{\dd\mu^s}{\dd|\mu^s|}(x)
\in\Lambda_{\mathcal A}
\qquad
\text{for }|\mu^s|\text{-almost every }x\in\Omega.
\]
The statement is local: it applies on every relatively compact open subset
where $\mu$ has finite mass.
\end{thm}

The theorem is quoted without proof from
De Philippis--Rindler~\cite[Theorem~1.1]{DePhilippisRindler2016}.

For the row-wise divergence operator on $\mathbb R^{2\times2}$,
\[
(\operatorname{div}M)_i:=\partial_jM_{ij},
\]
the symbol is
\[
\mathbb A(\xi)M=M\xi.
\]
Consequently,
\[
\Lambda_{\operatorname{div}}
=
\left\{
M\in\mathbb R^{2\times2}:
M\xi=0
\text{ for some }\xi\in\mathbb S^1
\right\}.
\]
In dimension two, every non-zero element of
$\Lambda_{\operatorname{div}}$ has rank one.

\begin{lem}[Rank-one stress and rectifiability]
\label{lem:rank-one-stress}
Let $U\subset\mathbb R^2$ be open. Suppose that $T$ is a locally finite
symmetric $2\times2$ matrix-valued Radon measure on $U$, and that $q$ is a
locally finite $\mathbb R^2$-valued Radon measure on $U$. Let $\zeta$ and
$\nu$ be non-negative Radon measures on $U$ satisfying
\[
\operatorname{div}T=q,
\qquad
\operatorname{tr}T=2\zeta,
\qquad
T\perp\mathcal L^2,
\qquad
\zeta\le\nu,
\qquad
\nu\ll\zeta.
\]
Assume also that
\[
\Theta_*^1(\nu,x)
:=
\liminf_{r\downarrow0}
\frac{\nu(B_r(x))}{2r}
>0
\qquad
\text{for }\nu\text{-almost every }x.
\]
Then there exist a countably $1$-rectifiable set $S\subset U$, a density
$\Theta>0$ defined $\mathcal H^1$-almost everywhere on $S$, and a measurable
unit tangent field $\tau:S\to\mathbb S^1$ such that
\[
\zeta=\Theta\,\mathcal H^1\llcorner S,
\qquad
T=2\Theta\,\tau\otimes\tau\,
\mathcal H^1\llcorner S.
\]
Moreover,
\[
\operatorname{span}\tau(x)=T_xS
\qquad
\text{for }\mathcal H^1\text{-almost every }x\in S.
\]
\end{lem}

\begin{proof}
Fix $W\Subset U$. Choose an open set $U'$ such that
\[
W\Subset U'\Subset U,
\]
and choose $\chi\in C_c^\infty(U)$ satisfying
$\chi\equiv1$ on $U'$.

Set
\[
T^\chi:=\chi T.
\]
Then
\[
q^\chi:=\operatorname{div}T^\chi
=\chi q+T\nabla\chi
\]
is a compactly supported vector-valued Radon measure in $U$.

Extend $T^\chi$ and $q^\chi$ by zero to $\mathbb R^2$. For each row define
\[
R_{i\boldsymbol\cdot}
:=
\nabla\Delta^{-1}q_i^\chi.
\]
Then, in the sense of distributions,
\[
\partial_jR_{ij}=q_i^\chi.
\]
Since the gradient of the two-dimensional Newtonian potential is a
Riesz potential of order one,
\[
R\in L^p_{\mathrm{loc}}(\mathbb R^2)
\qquad\text{for every }1\le p<2.
\]

Define
\[
M:=T^\chi-R\mathcal L^2.
\]
Then $M$ is locally finite and
\[
\operatorname{div}M=0.
\]
Because $T\perp\mathcal L^2$ and
$R\mathcal L^2\ll\mathcal L^2$, the singular part of $M$ is
\[
M^s=T^\chi.
\]

Restrict $M$ to $U'$. It is a finite Radon measure on $U'$ and is
divergence-free there. Hence Theorem~\ref{thm:dpr-structure}, applied with
$\mathcal A=\operatorname{div}$ acting row by row, gives
\[
\frac{\dd T}{\dd|T|}(x)
\in\Lambda_{\operatorname{div}}
\qquad
\text{for }|T|\text{-almost every }x\in U',
\]
because $\chi\equiv1$ on $U'$.

Thus the polar of $T$ has rank one almost everywhere. Since $T$ is symmetric,
there exist a scalar $\lambda$ and a measurable unit vector field $\tau$ such
that
\[
\frac{\dd T}{\dd|T|}
=
\lambda\,\tau\otimes\tau.
\]
Moreover,
\[
\operatorname{tr}T=2\zeta\ge0.
\]
A non-zero symmetric rank-one matrix has non-zero trace, and therefore
$\lambda>0$ almost everywhere on the support of $T$. Consequently,
\[
T=2\tau\otimes\tau\,\zeta.
\]
Since $W\Subset U$ was arbitrary, this identity holds throughout $U$.

Define the $1$-varifold
\[
\mathbf V(\varphi)
:=
\int_U
\varphi\bigl(x,\operatorname{span}\tau(x)\bigr)\,\dd\zeta(x),
\qquad
\varphi\in C_c(U\times G(2,1)).
\]
Its weight measure is $\|\mathbf V\|=\zeta$. For every
$X\in C_c^1(U;\mathbb R^2)$,
\[
\begin{aligned}
\delta\mathbf V(X)
&=
\int_U\tau\otimes\tau:DX\,\dd\zeta\\
&=
\frac12\int_U T:DX\\
&=
-\frac12\int_U X\cdot\dd q.
\end{aligned}
\]
Hence $\mathbf V$ has locally bounded first variation.

It remains to verify the positive lower $1$-density condition for $\zeta$.
Since $\nu\ll\zeta$, let
\[
f:=\frac{\dd\nu}{\dd\zeta}.
\]
The inequality $\zeta\le\nu$ implies
\[
f\ge1
\qquad
\text{for }\zeta\text{-almost every }x.
\]
Thus $\nu$ and $\zeta$ are mutually absolutely continuous. At
$\zeta$-almost every differentiation point of $f$,
\[
\frac{\nu(B_r(x))}{\zeta(B_r(x))}
\longrightarrow f(x)<\infty.
\]
Therefore, for sufficiently small $r$,
\[
\frac{\zeta(B_r(x))}{2r}
\ge
\frac1{2f(x)}
\frac{\nu(B_r(x))}{2r}.
\]
Using the assumed positive lower density of $\nu$, we obtain
\[
\Theta_*^1(\zeta,x)
:=
\liminf_{r\downarrow0}
\frac{\zeta(B_r(x))}{2r}
>0
\qquad
\text{for }\zeta\text{-almost every }x.
\]

Allard's rectifiability theorem
\cite[Theorem~5.5]{Allard1972} now applies locally to $\mathbf V$: its first
variation is locally bounded and its weight measure has positive lower
$1$-density almost everywhere. Hence there exist a countably
$1$-rectifiable set $S\subset U$ and a density $\Theta>0$ such that
\[
\zeta=\Theta\,\mathcal H^1\llcorner S.
\]
Furthermore, the plane field of $\mathbf V$ agrees with the approximate
tangent line:
\[
\operatorname{span}\tau(x)=T_xS
\qquad
\text{for }\mathcal H^1\text{-almost every }x\in S.
\]
Substituting this representation of $\zeta$ into
\[
T=2\tau\otimes\tau\,\zeta
\]
gives
\[
T=2\Theta\,\tau\otimes\tau\,
\mathcal H^1\llcorner S.
\]
\end{proof}

\begin{rem}[Why the argument is two-dimensional]
For the row-divergence operator in \(\R^d\), the wave cone consists of
matrices having a non-trivial kernel.  In dimension \(d=2\), this forces a
non-zero matrix to have rank one.  In higher dimensions it gives only
\(\operatorname{rank}A\leq d-1\), which is insufficient to identify a
one-dimensional tangent projection.  This is the precise point where the
planar assumption enters the measure-theoretic part of the proof.
\end{rem}

\begin{rem}[Relation with Bethuel's elliptic argument]
\label{rem:relation-with-bethuel}
Lemma~\ref{lem:rank-one-stress} isolates the geometric compactness step that
is implicit in Bethuel's proof of
\cite[Theorem~1.2]{Bethuel2025}.  In the present formulation, the
modified Cauchy--Riemann identities appearing in
\cite[Lemma~1.19, (1.71)--(1.72)]{Bethuel2025} are encoded by the stress
relations
\[
 \operatorname{div}T=q,
 \qquad
 \operatorname{tr}T=2\zeta.
\]
The De Philippis--Rindler structure theorem first yields the rank-one
direction of the singular stress, while Allard's rectifiability theorem
identifies this direction with the approximate tangent of a countably
\(1\)-rectifiable set.  Thus, once the positive one-dimensional density and
the mutual absolute continuity of \(\nu\) and \(\zeta\) have been established,
the tangent-cone and rectifiability part of Bethuel's argument can be
replaced by this measure-theoretic route.
\end{rem}

\subsection{Forced planar compactness}

We now prove the main result of this section.

\begin{thm}[Forced planar Bethuel compactness]
\label{thm:forced-bethuel}
Let \(U\subset\R^2\) be open and let \(\eps_j\downarrow0\).  Suppose that
\[
 v_j\in C^2(U;\R^k),
 \qquad
 F_j\in L^2(U;\R^k),
\]
solve
\begin{equation}\label{eq:forced-sequence}
 -\eps_j\Delta v_j+\frac1{\eps_j}DV(v_j)=F_j
 \qquad\text{in }U,
\end{equation}
and satisfy
\begin{equation}\label{eq:forced-sequence-bounds}
\begin{aligned}
 E_{\eps_j}(v_j;U)&\leq M,\\
 \int_U\frac{|F_j|^2}{\eps_j}\,\dd x&\leq\Lambda,\\
 \|v_j\|_{L^\infty(U)}
 +\eps_j\|\nabla v_j\|_{L^\infty(U)}
 &\leq L.
\end{aligned}
\end{equation}
Define the Radon measures
\begin{equation}\label{eq:forced-slice-measures}
\begin{aligned}
 \nu_j
 &:=
 e_{\eps_j}(v_j)\,\dd x,\\
 \zeta_j
 &:=
 \frac1{\eps_j}V(v_j)\,\dd x,\\
 \mu_{j,ab}
 &:=
 \eps_j\partial_av_j\cdot\partial_bv_j\,\dd x,\\
 q_{j,a}
 &:=
 F_j\cdot\partial_av_j\,\dd x,
 \qquad a,b\in\{1,2\}.
\end{aligned}
\end{equation}
After passing to a subsequence,
\begin{equation}\label{eq:forced-measure-convergence}
 \nu_j\stackrel{*}{\rightharpoonup}\nu,
 \qquad
 \zeta_j\stackrel{*}{\rightharpoonup}\zeta,
 \qquad
 \mu_{j,ab}\stackrel{*}{\rightharpoonup}\mu_{ab},
 \qquad
 q_{j,a}\stackrel{*}{\rightharpoonup}q_a
\end{equation}
locally in \(U\).
Write \(q:=(q_1,q_2)\) for the limiting vector-valued measure.

There exist a relatively closed, countably \(1\)-rectifiable set
\(S\subset U\), densities \(e,\Theta>0\), and a symmetric matrix-valued
density \(m=(m_{ab})\) such that
\begin{equation}\label{eq:forced-bethuel-representations}
 \nu=e\,\cH^1\llcorner S,
 \qquad
 \zeta=\Theta\,\cH^1\llcorner S,
 \qquad
 \mu_{ab}=m_{ab}\,\cH^1\llcorner S.
\end{equation}
The measures \(\nu\) and \(\zeta\) are mutually absolutely continuous.
If \((\tau,n)\) is an approximate tangent--normal frame of \(S\), then
\begin{equation}\label{eq:forced-bethuel-identities}
 e=m_{nn},
 \qquad
 m_{n\tau}=0,
 \qquad
 2\Theta=m_{nn}-m_{\tau\tau}
\end{equation}
for \(\cH^1\)-almost every point of \(S\).

Moreover, if
\begin{equation}\label{eq:forced-limit-stress-definition}
 T:=\nu I-\mu,
\end{equation}
then
\begin{equation}\label{eq:forced-bethuel-stress}
 T
 =
 2\Theta\,\tau\otimes\tau\,
 \cH^1\llcorner S,
 \qquad
 \operatorname{div}T=q.
\end{equation}
Finally,
\[
 q\ll\zeta,
\]
and \(q\) has a locally square-integrable density with respect to \(\zeta\).
\end{thm}

\begin{proof}
We divide the proof into seven steps.

\medskip
\noindent
\textbf{Step 1: Weak-star compactness of the diffuse measures.}
By definition,
\[
 0\leq\zeta_j\leq\nu_j.
\]
Moreover,
\[
 |\mu_{j,ab}|
 \leq
 \frac{\eps_j}{2}
 \left(
 |\partial_av_j|^2+|\partial_bv_j|^2
 \right)\dd x
 \leq2\nu_j.
\]
Thus \((\nu_j)\), \((\zeta_j)\), and \((\mu_{j,ab})\) are locally
uniformly bounded families of Radon measures.  A diagonal extraction over
a countable exhaustion of \(U\) by relatively compact open sets gives
\[
 \nu_j\stackrel{*}{\rightharpoonup}\nu,
 \qquad
 \zeta_j\stackrel{*}{\rightharpoonup}\zeta,
 \qquad
 \mu_{j,ab}\stackrel{*}{\rightharpoonup}\mu_{ab}.
\]
The diffuse domination inequalities pass to the limit:
\begin{equation}\label{eq:forced-limit-dominations}
 0\leq\zeta\leq\nu,
 \qquad
 |\mu_{ab}|\leq2\nu.
\end{equation}

Set \(q_j:=(q_{j,1},q_{j,2})\).
For the momentum, direct Cauchy--Schwarz gives
\begin{align*}
 |q_j|(U)
 &\leq
 \int_U|F_j|\,|\nabla v_j|\,\dd x\\
 &\leq
 \left(
 \int_U\frac{|F_j|^2}{\eps_j}\,\dd x
 \right)^{\frac12}
 \left(
 \int_U\eps_j|\nabla v_j|^2\,\dd x
 \right)^{\frac12}\\
 &\leq
 \sqrt{2\Lambda M}.
\end{align*}
Hence, after a further extraction,
\[
 q_j\stackrel{*}{\rightharpoonup}q
 \qquad\text{in }\mathcal M_{\rm loc}(U;\R^2).
\]

More precisely, for every
\(X=(X_1,X_2)\in C_c(U;\R^2)\),
\begin{align}
 \left|
 \int_U X\cdot\dd q_j
 \right|^2
 &=
 \left|
 \int_U
 F_j\cdot
 \bigl(X_a\partial_av_j\bigr)\,\dd x
 \right|^2
 \nonumber\\
 &\leq
 \left(
 \int_U\frac{|F_j|^2}{\eps_j}\,\dd x
 \right)
 \left(
 \int_U
 \eps_j
 X_aX_b
 \partial_av_j\cdot\partial_bv_j\,\dd x
 \right)
 \nonumber\\
 &\leq
 \Lambda
 \int_U X_aX_b\,\dd\mu_{j,ab}.
 \label{eq:forced-momentum-action}
\end{align}
Passing to the limit yields
\begin{equation}\label{eq:forced-limit-momentum-action}
 \left|
 \int_U X\cdot\dd q
 \right|^2
 \leq
 \Lambda
 \int_U X_aX_b\,\dd\mu_{ab}.
\end{equation}

Taking \(X=\varphi e_a\) in
\eqref{eq:forced-limit-momentum-action} gives
\[
 \left|
 \int_U\varphi\,\dd q_a
 \right|^2
 \leq
 \Lambda\int_U\varphi^2\,\dd\mu_{aa}
 \leq
 2\Lambda\int_U\varphi^2\,\dd\nu.
\]
By the Riesz representation theorem for \(L^2(\nu)\), this implies
\begin{equation}\label{eq:forced-q-L2-nu}
 q\ll\nu,
 \qquad
 \frac{\dd q}{\dd\nu}\in L^2_{\rm loc}(\nu;\R^2).
\end{equation}

\medskip
\noindent
\textbf{Step 2: Positive lower one-dimensional density.}
Let
\[
 S:=\operatorname{spt}\nu\cap U.
\]
Fix compact sets
\[
 K\Subset K'\Subset U.
\]
We claim that there exist \(r_*>0\) and \(\eta_*>0\), depending only on
\(V,L,\Lambda,K,K'\), such that
\begin{equation}\label{eq:forced-lower-density}
 \nu(B_r(x))\geq\eta_*r
\end{equation}
for every \(x\in S\cap K\) and every \(0<r<r_*\).

Choose \(r_*>0\) so small that
\[
 B_{r_*}(x)\subset K'
 \qquad\text{for every }x\in K,
\]
and
\begin{equation}\label{eq:forced-small-rescaled-action}
 r_*\Lambda\leq\gamma_{\rm f},
\end{equation}
where \(\gamma_{\rm f}\) is the constant in
Theorem~\ref{thm:forced-clearing}.

Suppose that \eqref{eq:forced-lower-density} fails for some
\(x\in S\cap K\) and \(0<r<r_*\).  Then
\[
 \nu(B_r(x))<\eta_*r.
\]
Choose
\[
 s\in\left(\frac{3r}{4},r\right)
\]
such that
\[
 \nu(\partial B_s(x))=0.
\]
This is possible because a Radon measure charges at most countably many
pairwise disjoint circles centered at \(x\).  Since \(B_s(x)\subset B_r(x)\),
\[
 \frac{\nu(B_s(x))}{s}
 \leq
 \frac{4}{3}\eta_*.
\]
The boundary condition
\(\nu(\partial B_s(x))=0\) and weak convergence imply
\[
 E_{\eps_j}(v_j;B_s(x))
 =
 \nu_j(B_s(x))
 \longrightarrow
 \nu(B_s(x)).
\]
Define the rescaled maps and forcing terms
\[
 \widetilde v_j(y):=v_j(x+sy),
 \qquad
 \widetilde\eps_j:=\frac{\eps_j}{s},
 \qquad
 \widetilde F_j(y):=sF_j(x+sy).
\]
Then
\[
 -\widetilde\eps_j\Delta\widetilde v_j
 +\frac1{\widetilde\eps_j}DV(\widetilde v_j)
 =
 \widetilde F_j
 \qquad\text{in }B_1.
\]
Furthermore,
\[
 E_{\widetilde\eps_j}(\widetilde v_j;B_1)
 =
 \frac1sE_{\eps_j}(v_j;B_s(x))
 \longrightarrow
 \frac{\nu(B_s(x))}{s}
 \leq\frac43\eta_*,
\]
and
\begin{align}
 \int_{B_1}
 \frac{|\widetilde F_j|^2}{\widetilde\eps_j}\,\dd y
 &=
 s\int_{B_s(x)}\frac{|F_j|^2}{\eps_j}\,\dd x
 \nonumber\\
 &\leq
 s\Lambda
 \leq
 r_*\Lambda
 \leq\gamma_{\rm f}.
 \label{eq:forced-rescaled-action-density}
\end{align}
The scaled uniform estimate is unchanged:
\[
 \|\widetilde v_j\|_{L^\infty(B_1)}
 +\widetilde\eps_j
 \|\nabla\widetilde v_j\|_{L^\infty(B_1)}
 \leq L.
\]
Choose \(\eta_*>0\) so small that
\[
 \frac43\eta_*<2\eta_{\rm f}.
\]
For all sufficiently large \(j\),
Theorem~\ref{thm:forced-clearing} applies to \(\widetilde v_j\).
Its strong inner estimate gives
\[
 E_{\widetilde\eps_j}
 \left(
 \widetilde v_j;B_{\frac58}
 \right)
 \leq
 C\widetilde\eps_j^2
 \left\{
 E_{\widetilde\eps_j}(\widetilde v_j;B_1)
 +
 \int_{B_1}
 \frac{|\widetilde F_j|^2}{\widetilde\eps_j}\,\dd y
 \right\}.
\]
Because \(\widetilde\eps_j=\eps_j/s\to0\), the right-hand side tends to
zero.  Scaling back,
\[
 \nu_j\left(B_{\frac{5s}{8}}(x)\right)
 =
 sE_{\widetilde\eps_j}
 \left(
 \widetilde v_j;B_{\frac58}
 \right)
 \longrightarrow0.
\]
By the Portmanteau inequality for the open ball
\(B_{5s/8}(x)\),
\[
 \nu\left(B_{\frac{5s}{8}}(x)\right)
 \leq
 \liminf_{j\to\infty}
 \nu_j\left(B_{\frac{5s}{8}}(x)\right)
 =0.
\]
This contradicts \(x\in\operatorname{spt}\nu\).  Therefore
\eqref{eq:forced-lower-density} holds.

In particular,
\begin{equation}\label{eq:forced-positive-lower-density-ae}
 \Theta_*^1(\nu,x)>0
 \qquad\text{for every }x\in S.
\end{equation}

\medskip
\noindent
\textbf{Step 3: Locally finite length of the support.}
We next prove
\begin{equation}\label{eq:forced-support-length}
 \cH^1(S\cap K)
 \leq
 C\eta_*^{-1}\nu(K')
 <\infty.
\end{equation}
Fix \(0<\delta<r_*/10\).  The family
\[
 \left\{
 B_r(x):
 x\in S\cap K,\ 0<r<\delta
 \right\}
\]
is a fine cover of \(S\cap K\).  By the \(5r\)-covering theorem, there
exists a countable pairwise disjoint subfamily
\((B_{r_i}(x_i))_i\) such that
\[
 S\cap K\subset\bigcup_iB_{5r_i}(x_i).
\]
Using \eqref{eq:forced-lower-density},
\[
 \sum_i diam(B_{5r_i}(x_i))
 =
 10\sum_i r_i
 \leq
 \frac{10}{\eta_*}\sum_i\nu(B_{r_i}(x_i)).
\]
The balls \(B_{r_i}(x_i)\) are pairwise disjoint and contained in \(K'\),
so
\[
 \sum_i\nu(B_{r_i}(x_i))\leq\nu(K').
\]
Taking the infimum over such covers and then letting \(\delta\downarrow0\)
proves \eqref{eq:forced-support-length}.

Since a set of locally finite \(\cH^1\)-measure has zero
two-dimensional Lebesgue measure,
\[
 \mathcal L^2(S)=0.
\]
The measure \(\nu\) is supported on \(S\), and hence
\begin{equation}\label{eq:forced-nu-singular}
 \nu\perp\mathcal L^2.
\end{equation}
By \eqref{eq:forced-limit-dominations}, the measures \(\zeta\),
\(\mu_{ab}\), and \(T:=\nu I-\mu\) are also singular with respect to
\(\mathcal L^2\).

\medskip
\noindent
\textbf{Step 4: Comparison of the limiting full and potential energies.}
Fix \(x\in U\) and \(r>0\) such that
\[
 \overline{B_r(x)}\subset U
\]
and
\[
 \nu(\partial B_{r/2}(x))
 =
 \nu(\partial B_{3r/4}(x))
 =
 \zeta(\partial B_{3r/4}(x))
 =0.
\]
Apply the scaled comparison estimate
\eqref{eq:forced-energy-potential-scaled} to \(v_j\) on \(B_r(x)\):
\begin{align}
 \nu_j\left(B_{\frac r2}(x)\right)
 \leq C(V,L)\bigg\{&
 \zeta_j\left(B_{\frac{3r}{4}}(x)\right)
 \nonumber\\
 &+\frac{\eps_j}{r}
 \nu_j\left(
 B_{\frac{3r}{4}}(x)\setminus B_{\frac r2}(x)
 \right)
 \nonumber\\
 &+\eps_j^2
 \int_{B_r(x)}\frac{|F_j|^2}{\eps_j}\,\dd x
 \bigg\}.
 \label{eq:forced-prelimit-comparison}
\end{align}
The second error term satisfies
\[
 \frac{\eps_j}{r}
 \nu_j\left(
 B_{\frac{3r}{4}}(x)\setminus B_{\frac r2}(x)
 \right)
 \leq
 \frac{\eps_j}{r}M
 \longrightarrow0,
\]
and the forcing error satisfies
\[
 \eps_j^2
 \int_{B_r(x)}\frac{|F_j|^2}{\eps_j}\,\dd x
 \leq
 \eps_j^2\Lambda
 \longrightarrow0.
\]
Passing to the limit in
\eqref{eq:forced-prelimit-comparison} gives
\begin{equation}\label{eq:forced-limit-comparison}
 \nu\left(B_{\frac r2}(x)\right)
 \leq
 C(V,L)
 \zeta\left(B_{\frac{3r}{4}}(x)\right).
\end{equation}
By approximation of the radius from above and below, the same estimate
holds for every sufficiently small \(r\) for which the two sides are
well-defined; this is enough for the differentiation argument below.

We already know from \eqref{eq:forced-limit-dominations} that
\[
 \zeta\leq\nu.
\]
We now prove the reverse absolute continuity
\[
 \nu\ll\zeta.
\]
Let
\[
 f:=\frac{\dd\zeta}{\dd\nu},
 \qquad 0\leq f\leq1.
\]
Suppose, for contradiction, that
\[
 A:=\{x:f(x)=0\}
\]
has positive \(\nu\)-measure.  By the differentiation theorem, for
\(\nu\)-almost every \(x\in A\),
\begin{equation}\label{eq:forced-zero-density-ratio}
 \frac{\zeta(B_r(x))}{\nu(B_r(x))}
 \longrightarrow0
 \qquad\text{as }r\downarrow0.
\end{equation}
A standard pointwise doubling consequence of the Besicovitch covering
theorem states that, for \(\nu\)-almost every \(x\), there exist
\(C_x<\infty\) and a sequence \(r_\ell\downarrow0\) such that
\begin{equation}\label{eq:forced-doubling-sequence}
 \nu(B_{r_\ell}(x))
 \leq
 C_x\nu\left(B_{\frac{r_\ell}{2}}(x)\right).
\end{equation}
Choose \(x\in A\) for which
\eqref{eq:forced-zero-density-ratio} and
\eqref{eq:forced-doubling-sequence} both hold.  Combining
\eqref{eq:forced-limit-comparison} with
\eqref{eq:forced-doubling-sequence}, we obtain
\begin{align*}
 \nu(B_{r_\ell}(x))
 &\leq
 C_x\nu\left(B_{\frac{r_\ell}{2}}(x)\right)\\
 &\leq
 C_xC(V,L)
 \zeta\left(B_{\frac{3r_\ell}{4}}(x)\right)\\
 &\leq
 C_xC(V,L)\zeta(B_{r_\ell}(x)).
\end{align*}
After division by \(\nu(B_{r_\ell}(x))>0\), this gives
\[
 1
 \leq
 C_xC(V,L)
 \frac{\zeta(B_{r_\ell}(x))}
 {\nu(B_{r_\ell}(x))},
\]
contradicting \eqref{eq:forced-zero-density-ratio}.  Therefore
\begin{equation}\label{eq:forced-mutual-ac}
 \zeta\leq\nu,
 \qquad
 \nu\ll\zeta.
\end{equation}

At every common differentiation point of \(\nu\) and \(\zeta\) lying on a
rectifiable carrier, \eqref{eq:forced-limit-comparison} will also imply
the local density bound
\begin{equation}\label{eq:forced-density-comparison}
 \frac{\dd\nu}{\dd\zeta}\leq C(V,L).
\end{equation}
We shall justify this after rectifiability has been established.

\medskip
\noindent
\textbf{Step 5: The exact stress identities.}
Define the diffuse stresses
\[
 T_j
 :=
 \nu_jI-\mu_j,
\]
that is,
\[
 T_{j,ab}
 =
 e_{\eps_j}(v_j)\delta_{ab}\,\dd x
 -
 \eps_j\partial_av_j\cdot\partial_bv_j\,\dd x.
\]
A direct computation using \eqref{eq:forced-sequence} gives
\begin{align*}
 \partial_b
 \left(
 e_{\eps_j}(v_j)\delta_{ab}
 -
 \eps_j\partial_av_j\cdot\partial_bv_j
 \right)
 &=
 \left(
 -\eps_j\Delta v_j
 +\frac1{\eps_j}DV(v_j)
 \right)\cdot\partial_av_j\\
 &=
 F_j\cdot\partial_av_j.
\end{align*}
Hence
\begin{equation}\label{eq:forced-exact-stress-divergence}
 \operatorname{div}T_j=q_j.
\end{equation}
Since the spatial dimension is two,
\begin{align*}
 \operatorname{tr}T_j
 &=
 2e_{\eps_j}(v_j)\,\dd x
 -
 \eps_j|\nabla v_j|^2\,\dd x\\
 &=
 \frac2{\eps_j}V(v_j)\,\dd x
 =
 2\zeta_j.
\end{align*}
Thus
\begin{equation}\label{eq:forced-exact-stress}
 \operatorname{div}T_j=q_j,
 \qquad
 \operatorname{tr}T_j=2\zeta_j.
\end{equation}
Passing to the weak-star limits gives
\begin{equation}\label{eq:forced-limit-stress-identities}
 \operatorname{div}T=q,
 \qquad
 \operatorname{tr}T=2\zeta,
 \qquad
 T:=\nu I-\mu.
\end{equation}

By \eqref{eq:forced-nu-singular} and
\eqref{eq:forced-limit-dominations},
\[
 T\perp\mathcal L^2.
\]
Together with \eqref{eq:forced-positive-lower-density-ae},
\eqref{eq:forced-mutual-ac}, and
\eqref{eq:forced-limit-stress-identities}, all the hypotheses of
Lemma~\ref{lem:rank-one-stress} are satisfied.

\medskip
\noindent
\textbf{Step 6: Rectifiability and the Bethuel identities.}
Lemma~\ref{lem:rank-one-stress} gives a countably \(1\)-rectifiable set
\(S_0\subset U\), a positive density \(\Theta\), and an approximate unit
tangent \(\tau\) such that
\begin{equation}\label{eq:forced-stress-on-S0}
 \zeta
 =
 \Theta\,\cH^1\llcorner S_0,
 \qquad
 T
 =
 2\Theta\,\tau\otimes\tau\,
 \cH^1\llcorner S_0.
\end{equation}
Because \(\nu\ll\zeta\), there exists a positive density \(e\) such that
\[
 \nu=e\,\cH^1\llcorner S_0.
\]
Similarly, \(|\mu_{ab}|\leq2\nu\) implies that
\[
 \mu_{ab}
 =
 m_{ab}\,\cH^1\llcorner S_0
\]
for a symmetric matrix-valued density \(m=(m_{ab})\).

At \(\cH^1\)-almost every point of \(S_0\), choose the orthonormal frame
\((\tau,n)\), where \(\tau\) spans the approximate tangent line and \(n\)
is a unit normal.  The identity
\[
 T=\nu I-\mu
 =
 2\Theta\,\tau\otimes\tau\,
 \cH^1\llcorner S_0
\]
becomes, at the level of densities,
\begin{equation}\label{eq:forced-density-matrix-identity}
 eI-m
 =
 2\Theta\,\tau\otimes\tau.
\end{equation}
Taking the \(n\otimes n\) component gives
\[
 e-m_{nn}=0,
\]
and hence
\[
 e=m_{nn}.
\]
Taking the \(n\otimes\tau\) component gives
\[
 -m_{n\tau}=0,
\]
and hence
\[
 m_{n\tau}=0.
\]
Finally, the \(\tau\otimes\tau\) component gives
\[
 e-m_{\tau\tau}=2\Theta.
\]
Using \(e=m_{nn}\), we obtain
\[
 2\Theta=m_{nn}-m_{\tau\tau}.
\]
Because each \(\mu_j\) is positive semidefinite as a matrix-valued
measure, so is its weak-star limit \(\mu\); hence
\(m_{\tau\tau}\geq0\).  This proves
\eqref{eq:forced-bethuel-identities}.

We now verify \eqref{eq:forced-density-comparison}.  At a common
\(\cH^1\)-density point of \(S_0\), \(e\), and \(\Theta\),
\[
 \nu(B_r(x))=2e(x)r+o(r),
 \qquad
 \zeta(B_r(x))=2\Theta(x)r+o(r).
\]
Applying \eqref{eq:forced-limit-comparison} gives
\[
 e(x)
 \leq
 C(V,L)\Theta(x)
\]
after absorbing the fixed radius ratios \(1/2\) and \(3/4\) into the
constant.  Thus
\begin{equation}\label{eq:forced-e-theta-comparison}
 \Theta\leq e\leq C(V,L)\Theta
 \qquad\text{for }\cH^1\text{-almost every }x\in S_0.
\end{equation}

\medskip
\noindent
\textbf{Step 7: The closed carrier and the \(L^2\) momentum density.}
We finally replace \(S_0\) by the relatively closed set
\[
 S:=\operatorname{spt}\nu\cap U.
\]
Since \(\nu(U\setminus S_0)=0\), we have
\[
 \nu(S\setminus S_0)=0.
\]
We claim that
\[
 \cH^1(S\setminus S_0)=0.
\]
Indeed, let \(A:=S\setminus S_0\), and fix \(K\Subset U\).
For every open set \(O\supset A\cap K\), the lower density estimate
\eqref{eq:forced-lower-density} and the same \(5r\)-covering argument as
in Step~3 give
\[
 \cH^1(A\cap K)
 \leq
 C\eta_*^{-1}\nu(O).
\]
Since \(\nu(A)=0\), outer regularity allows us to choose \(O\supset A\)
with arbitrarily small \(\nu(O)\).  Hence
\[
 \cH^1(A\cap K)=0.
\]
Exhausting \(U\) by compact sets proves the claim.  Adjoining the
\(\cH^1\)-null set \(S\setminus S_0\) preserves countable rectifiability
and all the representations above.  Thus \(S\) is relatively closed and
\eqref{eq:forced-bethuel-representations} holds on \(S\).

It remains to improve \eqref{eq:forced-q-L2-nu} to an \(L^2(\zeta)\)
statement.  Write
\[
 q=b\,\nu,
 \qquad
 b:=\frac{\dd q}{\dd\nu}
 \in L^2_{\rm loc}(\nu;\R^2).
\]
By \eqref{eq:forced-e-theta-comparison},
\[
 \frac{\dd\nu}{\dd\zeta}
 =
 \frac{e}{\Theta}
 \leq C(V,L).
\]
Therefore
\[
 \frac{\dd q}{\dd\zeta}
 =
 b\,\frac{\dd\nu}{\dd\zeta},
\]
and, on every compact \(K\Subset U\),
\begin{align*}
 \int_K
 \left|
 \frac{\dd q}{\dd\zeta}
 \right|^2\dd\zeta
 &=
 \int_K
 |b|^2
 \left(
 \frac{\dd\nu}{\dd\zeta}
 \right)^2
 \dd\zeta\\
 &=
 \int_K
 |b|^2
 \frac{\dd\nu}{\dd\zeta}
 \,\dd\nu\\
 &\leq
 C(V,L)\int_K|b|^2\,\dd\nu
 <\infty.
\end{align*}
Thus
\[
 q\ll\zeta,
 \qquad
 \frac{\dd q}{\dd\zeta}
 \in L^2_{\rm loc}(\zeta;\R^2).
\]
All the assertions of the theorem are proved.
\end{proof}

\begin{rem}[Correspondence with Bethuel's conclusions]
For \(F_j=0\), the conclusions of
Theorem~\ref{thm:forced-bethuel} correspond to the following results in
\cite{Bethuel2025}:

\begin{enumerate}
\item the existence of a closed countably \(1\)-rectifiable concentration
set with locally finite length corresponds to
\cite[Theorem~1.2]{Bethuel2025};

\item the representations of \(\nu\) and \(\zeta\), their common carrier,
and the comparison of their densities correspond to
\cite[Theorem~1.8]{Bethuel2025};

\item the identities
\[
 m_{n\tau}=0,
 \qquad
 2\Theta=m_{nn}-m_{\tau\tau}
\]
are precisely the vectorial discrepancy relations in
\cite[Theorem~1.9, (1.48)]{Bethuel2025};

\item the limiting stress equation is the forced real-variable analogue of
the domain-stationarity identity
\cite[Lemma~1.19, (1.71)--(1.72)]{Bethuel2025}.
\end{enumerate}

The forcing affects the proof only through the rescaled action bound, the
additional errors in the energy--potential comparison, and the
measure-valued right-hand side
\(\operatorname{div}T=q\).  Once these terms are controlled, the limiting
rank-one stress and the Bethuel discrepancy relations have the same form as
in the unforced problem.
\end{rem}

\section{Parabolic compactness}
\label{sec:parabolic-compactness}

This section transfers the forced elliptic compactness theory of
Section~\ref{sec:forced-bethuel} to the parabolic equation
\eqref{eq:intro-flow}.  There are two points that require some care.
First, the elliptic forcing at a fixed time is
\[
 F_j(\cdot,t)=-\eps_j\partial_tu_j(\cdot,t),
\]
and the parabolic dissipation controls its action only for almost every
time.  Second, although one may choose a time-dependent subsequence at
almost every fixed time, the limiting potential and gradient measures used
later must arise from a single spacetime subsequence.  The latter issue is
resolved by first taking spacetime limits and then disintegrating them with
respect to time.

The proof is organized as follows.  We first establish the exact diffuse
balance laws and the positive-time bounds needed in the forced elliptic
theorem.  We then obtain compactness of both the phase maps and the
time-dependent energy measures.  Finally, we apply
Theorem~\ref{thm:forced-bethuel} at almost every time, disintegrate the
remaining spacetime measures, and identify the limiting stress and the
Bethuel relations on almost every time slice.

Throughout this section, we write
\[
 u_j:=u_{\eps_j},
\]
and use the diffuse measures introduced in
\eqref{eq:intro-measures}--\eqref{eq:intro-action}.

\subsection{Diffuse balance laws and positive-time estimates}

We begin with two exact identities.  They use only the equation and do not
require a scalar discrepancy estimate.

\begin{lem}[Diffuse stress and energy balances]
\label{lem:diffuse-balances}
For every smooth solution \(u_\eps\) of \eqref{eq:intro-flow}, one has
\begin{equation}\label{eq:diffuse-stress-balance}
 \operatorname{div}_xT_\eps^t=-q_\eps^t
\end{equation}
and
\begin{equation}\label{eq:diffuse-local-energy}
 \partial_te_\eps
 =
 \operatorname{div}_xq_\eps-\alpha_\eps
\end{equation}
in the sense of distributions on \(\T^2\times(0,T)\), where
\[
 e_\eps
 =
 \frac{\eps}{2}|\nabla u_\eps|^2
 +\frac1\eps V(u_\eps),
 \qquad
 q_{\eps,i}
 =
 \eps\partial_tu_\eps\cdot\partial_i u_\eps,
 \qquad
 \alpha_\eps
 =
 \eps|\partial_tu_\eps|^2.
\]
Consequently, for every \(0\leq t_1\leq t_2\leq T\),
\begin{equation}\label{eq:global-dissipation}
 E_\eps(u_\eps(\cdot,t_2);\T^2)
 +
 \int_{t_1}^{t_2}\int_{\T^2}
 \eps|\partial_tu_\eps|^2\,\dd x\dd t
 =
 E_\eps(u_\eps(\cdot,t_1);\T^2).
\end{equation}
\end{lem}

\begin{proof}
Differentiate the energy density:
\begin{align*}
 \partial_te_\eps
 &=
 \eps\nabla u_\eps:\nabla\partial_tu_\eps
 +\frac1\eps DV(u_\eps)\cdot\partial_tu_\eps\\
 &=
 \operatorname{div}_x
 \left(
 \eps\partial_tu_\eps\cdot\nabla u_\eps
 \right)
 +
 \left(
 -\eps\Delta u_\eps+\frac1\eps DV(u_\eps)
 \right)\cdot\partial_tu_\eps.
\end{align*}
The parabolic equation is equivalent to
\[
 -\eps\Delta u_\eps+\frac1\eps DV(u_\eps)
 =
 -\eps\partial_tu_\eps.
\]
Hence
\[
 \partial_te_\eps
 =
 \operatorname{div}_xq_\eps
 -\eps|\partial_tu_\eps|^2,
\]
which proves \eqref{eq:diffuse-local-energy}.

Next, recall that
\[
 T_{\eps,ij}
 =
 e_\eps\delta_{ij}
 -\eps\partial_i u_\eps\cdot\partial_j u_\eps.
\]
A direct computation gives
\begin{align*}
 \partial_jT_{\eps,ij}
 &=
 \partial_i e_\eps
 -\eps\partial_{ij}u_\eps\cdot\partial_j u_\eps
 -\eps\partial_i u_\eps\cdot\Delta u_\eps\\
 &=
 \left(
 -\eps\Delta u_\eps+\frac1\eps DV(u_\eps)
 \right)\cdot\partial_i u_\eps\\
 &=
 -\eps\partial_tu_\eps\cdot\partial_i u_\eps\\
 &=
 -q_{\eps,i}.
\end{align*}
This proves \eqref{eq:diffuse-stress-balance}.

Finally, integrate \eqref{eq:diffuse-local-energy} over the torus.
The divergence term vanishes because \(\T^2\) has no boundary, and hence
\[
 \frac{\dd}{\dd t}E_\eps(u_\eps(\cdot,t);\T^2)
 =
 -\int_{\T^2}\eps|\partial_tu_\eps|^2\,\dd x.
\]
Integration in time proves \eqref{eq:global-dissipation}.
\end{proof}

\begin{rem}[The two roles of dissipation]
The global identity \eqref{eq:global-dissipation} gives both
\[
 \sup_{0\leq t\leq T}
 E_{\eps_j}(u_j(\cdot,t);\T^2)
 \leq M_0
\]
and
\[
 \int_0^T\int_{\T^2}
 \eps_j|\partial_tu_j|^2\,\dd x\dd t
 \leq M_0.
\]
The first estimate provides spatial energy compactness.  The second is
exactly the spacetime integral of the forced action
\[
 \int_{\T^2}\frac{|F_j(\cdot,t)|^2}{\eps_j}\,\dd x,
 \qquad
 F_j(\cdot,t)=-\eps_j\partial_tu_j(\cdot,t),
\]
required by Theorem~\ref{thm:forced-bethuel}.
\end{rem}

The forced elliptic theorem also requires uniform control of the target
range and of the scaled gradient.  These estimates hold away from the
initial time.

\begin{lem}[Positive-time bounds]
\label{lem:positive-time-bounds}
For every \(0<\tau<T\), there exists
\[
 L_\tau=L_\tau(V,M_0,\tau,T)<\infty
\]
such that
\begin{equation}\label{eq:positive-time-bounds}
 \|u_j\|_{L^\infty(\T^2\times[\tau,T])}
 +
 \eps_j
 \|\nabla u_j\|_{L^\infty(\T^2\times[\tau,T])}
 \leq L_\tau
\end{equation}
for all sufficiently large \(j\).  This is the only uniformity needed
below; the finitely many discarded indices play no role in the subsequence
argument.
\end{lem}

\begin{proof}
Assumption~\ref{assum:potential} implies that there exist \(c,C>0\),
depending only on \(V\), such that
\begin{equation}\label{eq:target-coercive-growth}
 y\cdot DV(y)\geq c|y|^2-C
 \qquad\text{for every }y\in\R^k.
\end{equation}
Indeed, this is part of the assumption for large \(|y|\), and the remaining
bounded target region can be absorbed into the constant \(C\).

Set
\[
 w_j:=|u_j|^2.
\]
Using the equation and \eqref{eq:target-coercive-growth}, we obtain
\begin{align*}
 (\partial_t-\Delta)w_j
 &=
 -2|\nabla u_j|^2
 -\frac2{\eps_j^2}u_j\cdot DV(u_j)\\
 &\leq
 -\frac{2c}{\eps_j^2}w_j
 +\frac{2C}{\eps_j^2}.
\end{align*}
After changing \(c,C\), this becomes
\begin{equation}\label{eq:positive-time-scalar}
 (\partial_t-\Delta)w_j
 \leq
 -c\eps_j^{-2}w_j+C\eps_j^{-2}.
\end{equation}

The growth assumptions on \(V\) also imply
\[
 V(y)\geq c_0|y|^2-C_0.
\]
Since
\[
 \int_{\T^2}V(u_j(x,0))\,\dd x
 \leq
 \eps_jE_{\eps_j}(u_j(\cdot,0))
 \leq M_0,
\]
we have
\[
 \|u_j(\cdot,0)\|_{L^2(\T^2)}\leq C(V,M_0).
\]
Comparison in \eqref{eq:positive-time-scalar} gives
\[
 w_j(t)
 \leq
 e^{-ct/\eps_j^2}e^{t\Delta}w_j(0)+C.
\]
The \(L^1\)-to-\(L^\infty\) heat-kernel estimate on \(\T^2\) yields
\[
 \|e^{t\Delta}w_j(0)\|_{L^\infty}
 \leq
 C_\tau\|w_j(0)\|_{L^1}
 \qquad\text{for }t\geq\frac{\tau}{2}.
\]
It follows that
\begin{equation}\label{eq:positive-time-target-bound}
 \|u_j\|_{L^\infty(\T^2\times[\tau/2,T])}
 \leq C_\tau.
\end{equation}

We next estimate the gradient.  For all sufficiently large \(j\),
\[
 \eps_j^2<\frac{\tau}{2}.
\]
Fix \(t\in[\tau,T]\), and write Duhamel's formula on the interval
\([t-\eps_j^2,t]\):
\[
 u_j(t)
 =
 e^{\eps_j^2\Delta}u_j(t-\eps_j^2)
 -
 \frac1{\eps_j^2}
 \int_0^{\eps_j^2}
 e^{s\Delta}DV(u_j(t-s))\,\dd s.
\]
Taking a spatial gradient and using
\[
 \|\nabla e^{s\Delta}f\|_{L^\infty}
 \leq
 Cs^{-1/2}\|f\|_{L^\infty},
\]
we obtain
\begin{align*}
 \|\nabla u_j(t)\|_{L^\infty}
 &\leq
 C\eps_j^{-1}
 \|u_j(t-\eps_j^2)\|_{L^\infty}\\
 &\quad+
 \frac{C}{\eps_j^2}
 \int_0^{\eps_j^2}
 s^{-1/2}
 \|DV(u_j(t-s))\|_{L^\infty}\,\dd s.
\end{align*}
By \eqref{eq:positive-time-target-bound}, the arguments of \(DV\) remain
in a fixed compact subset of \(\R^k\) for \(t-s\geq\tau/2\), and hence
\[
 \|DV(u_j(t-s))\|_{L^\infty}\leq C_\tau.
\]
Therefore
\[
 \|\nabla u_j(t)\|_{L^\infty}
 \leq
 C_\tau\eps_j^{-1}
 +
 \frac{C_\tau}{\eps_j^2}
 \int_0^{\eps_j^2}s^{-1/2}\,\dd s
 \leq
 C_\tau\eps_j^{-1}.
\]
This proves \eqref{eq:positive-time-bounds}.
\end{proof}

\begin{rem}[Why the initial time is excluded]
No uniform \(L^\infty\) bound is assumed for the initial data.  The
parabolic equation supplies the required target and gradient estimates only
after a positive amount of time.  This is why all slice-wise structural
conclusions are first established on \((\tau,T)\) and then extended to
almost every \(t\in(0,T)\) by taking a countable union over
\(\tau\downarrow0\).
\end{rem}

\subsection{Compactness of the phases}

We next show that the maps converge to a partition taking values in the
finite well set.  The correct scalar quantities are weighted distances in
the target.

For each \(a\in\{1,\ldots,q\}\), define
\[
 \Phi_a(y)
 :=
 \inf_\gamma
 \int_0^1
 \sqrt{2V(\gamma(s))}\,|\gamma'(s)|\,\dd s,
\]
where the infimum is taken over all absolutely continuous curves
\(\gamma:[0,1]\to\R^k\) satisfying
\[
 \gamma(0)=\sigma_a,
 \qquad
 \gamma(1)=y.
\]
We replace \(\Phi_a\) by a fixed truncation at large target distance.  The
truncated function remains Lipschitz, separates the wells jointly with the
other \(\Phi_b\), and satisfies
\begin{equation}\label{eq:weighted-distance-gradient}
 |D\Phi_a(y)|\leq\sqrt{2V(y)}
\end{equation}
for almost every \(y\in\R^k\).

\begin{lem}[Phase compactness]
\label{lem:phase-compactness}
After passing to a subsequence, there exists a spacetime Caccioppoli
partition \((E_a)_{a=1}^q\) of \(\T^2\times(0,T)\) such that
\begin{equation}\label{eq:phase-convergence-repeated}
 u_j
 \longrightarrow
 u_0:=\sum_{a=1}^q\sigma_a\mathbf1_{E_a}
 \qquad\text{in }L^1(\T^2\times(0,T);\R^k).
\end{equation}
\end{lem}

\begin{proof}
By the chain rule and \eqref{eq:weighted-distance-gradient},
\[
 |\nabla_x\Phi_a(u_j)|
 \leq
 \sqrt{2V(u_j)}\,|\nabla u_j|.
\]
Young's inequality gives the pointwise estimate
\[
 \sqrt{2V(u_j)}\,|\nabla u_j|
 \leq
 \frac{\eps_j}{2}|\nabla u_j|^2
 +\frac1{\eps_j}V(u_j)
 =
 e_{\eps_j}(u_j).
\]
Therefore
\begin{align}
 \int_0^T\int_{\T^2}
 |\nabla_x\Phi_a(u_j)|\,\dd x\dd t
 &\leq
 \int_0^T
 E_{\eps_j}(u_j(\cdot,t);\T^2)\,\dd t
 \nonumber\\
 &\leq
 TM_0.
 \label{eq:phase-spatial-BV}
\end{align}

Similarly,
\[
 |\partial_t\Phi_a(u_j)|
 \leq
 \sqrt{2V(u_j)}\,|\partial_tu_j|.
\]
Cauchy--Schwarz in spacetime and
\eqref{eq:global-dissipation} yield
\begin{align}
 &\int_0^T\int_{\T^2}
 |\partial_t\Phi_a(u_j)|\,\dd x\dd t
 \nonumber\\
 &\qquad\leq
 \left(
 \int_0^T\int_{\T^2}
 \frac{2V(u_j)}{\eps_j}\,\dd x\dd t
 \right)^{\frac12}
 \left(
 \int_0^T\int_{\T^2}
 \eps_j|\partial_tu_j|^2\,\dd x\dd t
 \right)^{\frac12}
 \nonumber\\
 &\qquad\leq
 \sqrt{2T}\,M_0.
 \label{eq:phase-time-BV}
\end{align}
The truncation gives a uniform \(L^\infty\) bound for
\(\Phi_a(u_j)\).  Thus \eqref{eq:phase-spatial-BV} and
\eqref{eq:phase-time-BV} show that
\[
 \Phi_a(u_j)
\]
is uniformly bounded in
\[
 BV(\T^2\times(0,T))
\]
for every \(a\).  By \(BV\) compactness and a diagonal extraction,
\[
 \Phi_a(u_j)\to\varphi_a
 \qquad\text{in }L^1
\]
for every \(a\).

Moreover,
\begin{align}
 \int_0^T\int_{\T^2}V(u_j)\,\dd x\dd t
 &\leq
 \eps_j
 \int_0^T
 E_{\eps_j}(u_j(\cdot,t);\T^2)\,\dd t
 \nonumber\\
 &\leq
 \eps_jTM_0
 \longrightarrow0.
 \label{eq:phase-potential-vanishing}
\end{align}
Hence
\[
 \dist(u_j,\Sigma)\to0
\]
in measure.  The coercive lower bound
\[
 V(y)\geq c|y|^2-C
\]
also gives a uniform spacetime \(L^2\) bound for \(u_j\), so the sequence is
uniformly integrable in \(L^1\).

The vector-valued map
\[
 \Phi:=(\Phi_1,\ldots,\Phi_q)
\]
takes distinct values at distinct wells: indeed,
\(\Phi_a(\sigma_a)=0\), whereas
\(\Phi_a(\sigma_b)>0\) for \(b\neq a\).  Therefore the \(L^1\) limits
\((\varphi_a)_a\), together with
\eqref{eq:phase-potential-vanishing}, determine a measurable well-valued
map
\[
 u_0(x,t)\in\Sigma
 \qquad\text{for almost every }(x,t).
\]
Define
\[
 E_a:=\{(x,t):u_0(x,t)=\sigma_a\}.
\]
Then
\[
 u_0=\sum_{a=1}^q\sigma_a\mathbf1_{E_a}.
\]
Since the finite-valued map \(\Phi(u_0)\) belongs to \(BV\) and its values
are distinct, the standard structure theorem for finite-valued \(BV\)
maps implies that each \(E_a\) has locally finite perimeter and that
\((E_a)_{a=1}^q\) is a Caccioppoli partition.

Finally, convergence in measure to \(u_0\), combined with uniform
\(L^1\)-integrability, gives
\[
 u_j\to u_0
 \qquad\text{in }L^1,
\]
which proves \eqref{eq:phase-convergence-repeated}.
\end{proof}

\begin{rem}[Visible and hidden interfaces]
The partition \((E_a)\) records the interfaces visible in the limiting
phase map.  The support of the limiting diffuse energy may be larger than
the reduced boundaries of this partition: multiple transition layers can
collapse onto the same curve, and vectorial pseudo-profiles can carry
energy without changing the limiting well values.  For this reason, the
remainder of the paper studies the full-energy measures \(\nu_t\), rather
than identifying them a priori with the perimeter measure of the
partition.
\end{rem}

\subsection{Compactness of the time-dependent energy measures}

The next step is stronger than ordinary spacetime weak-star compactness:
we construct a limiting energy measure at almost every individual time.

\begin{lem}[Time-slice compactness of the energy]
\label{lem:energy-time-compactness}
After passing to a subsequence, there exist a weakly measurable family of
finite Radon measures
\[
 (\nu_t)_{t\in[0,T]}
\]
and a common full-measure set
\[
 \mathcal C\subset(0,T)
\]
such that
\begin{equation}\label{eq:energy-slice-convergence-repeated}
 \nu_j^t\stackrel{*}{\rightharpoonup}\nu_t
 \qquad\text{for every }t\in\mathcal C.
\end{equation}
After choosing right-continuous \(BV\) representatives,
\(\mathcal C\) may be taken to be the complement of a countable set.
Moreover,
\begin{equation}\label{eq:energy-spacetime-convergence}
 \nu_j^t\,\dd t
 \stackrel{*}{\rightharpoonup}
 \nu_t\,\dd t
 \qquad\text{on }\T^2\times(0,T).
\end{equation}
\end{lem}

\begin{proof}
Fix \(\phi\in C^1(\T^2)\).  Testing
\eqref{eq:diffuse-local-energy} with \(\phi\), we obtain
\begin{align}
 \frac{\dd}{\dd t}\nu_j^t(\phi)
 &=
 -\int_{\T^2}
 \phi\eps_j|\partial_tu_j|^2\,\dd x
 \nonumber\\
 &\quad-
 \int_{\T^2}
 \eps_j\partial_tu_j\cdot
 \bigl(\nabla u_j\nabla\phi\bigr)\,\dd x.
 \label{eq:energy-slice-derivative}
\end{align}
The first term satisfies
\[
 \int_0^T
 \left|
 \int_{\T^2}
 \phi\eps_j|\partial_tu_j|^2\,\dd x
 \right|\dd t
 \leq
 M_0\|\phi\|_{L^\infty}.
\]
For the second term, Cauchy--Schwarz gives
\begin{align*}
 &\int_0^T
 \left|
 \int_{\T^2}
 \eps_j\partial_tu_j\cdot
 \bigl(\nabla u_j\nabla\phi\bigr)\,\dd x
 \right|\dd t\\
 &\qquad\leq
 \|\nabla\phi\|_{L^\infty}
 \left(
 \int_0^T\int_{\T^2}
 \eps_j|\partial_tu_j|^2\,\dd x\dd t
 \right)^{\frac12}
 \left(
 \int_0^T\int_{\T^2}
 \eps_j|\nabla u_j|^2\,\dd x\dd t
 \right)^{\frac12}\\
 &\qquad\leq
 \sqrt{2T}\,M_0\|\nabla\phi\|_{L^\infty}.
\end{align*}
Therefore
\begin{equation}\label{eq:energy-time-variation}
 \operatorname{Var}_{[0,T]}
 \bigl(\nu_j^{\boldsymbol\cdot}(\phi)\bigr)
 \leq
 M_0\|\phi\|_{L^\infty}
 +
 \sqrt{2T}\,M_0\|\nabla\phi\|_{L^\infty}.
\end{equation}

Choose a countable set
\[
 \{\phi_\ell\}_{\ell\in\mathbb N}\subset C^1(\T^2)
\]
that is dense in \(C(\T^2)\) with respect to the uniform norm and contains
the constant function \(1\).  By Helly's selection theorem and a diagonal
argument, there exists a common subsequence such that
\[
 \nu_j^{\boldsymbol\cdot}(\phi_\ell)
 \longrightarrow
 f_\ell
\]
pointwise at every continuity point of the right-continuous \(BV\)
representative of \(f_\ell\), for every \(\ell\).

Each \(f_\ell\) has at most countably many discontinuities.  Let
\[
 \mathcal C
 :=
 (0,T)\setminus
 \bigcup_{\ell=1}^\infty\operatorname{Disc}(f_\ell).
\]
Then \((0,T)\setminus\mathcal C\) is countable.  For every
\(t\in\mathcal C\), the map
\[
 \phi_\ell\longmapsto f_\ell(t)
\]
is linear and positive on the dense family and satisfies
\[
 |f_\ell(t)|
 \leq
 M_0\|\phi_\ell\|_{L^\infty}.
\]
It therefore extends uniquely to a positive bounded linear functional on
\(C(\T^2)\).  By the Riesz representation theorem, there exists a finite
Radon measure \(\nu_t\) such that
\[
 \nu_t(\phi_\ell)=f_\ell(t)
 \qquad\text{for every }\ell.
\]
The uniform mass bound
\[
 \nu_j^t(\T^2)\leq M_0
\]
extends convergence from the dense family to every
\(\phi\in C(\T^2)\), proving
\eqref{eq:energy-slice-convergence-repeated}.

The scalar functions \(t\mapsto\nu_t(\phi_\ell)\) are measurable, and
density again gives weak measurability of \(t\mapsto\nu_t\).
For a product test function
\[
 \Psi(x,t)=\phi(x)\eta(t),
\]
dominated convergence in \(t\) yields
\[
 \int_0^T\eta(t)\nu_j^t(\phi)\,\dd t
 \longrightarrow
 \int_0^T\eta(t)\nu_t(\phi)\,\dd t.
\]
Finite sums of product functions are dense in
\(C(\T^2\times[0,T])\), while the total masses are uniformly bounded.
This proves \eqref{eq:energy-spacetime-convergence}.
\end{proof}

\begin{rem}[Why Helly compactness is needed]
A weak-star limit of the spacetime measures
\(\nu_j^t\,\dd t\) determines \(\nu_t\) only for almost every time and does
not by itself give convergence of the individual slices
\(\nu_j^t\).  The \(BV\)-in-time estimate
\eqref{eq:energy-time-variation} provides a common set
\(\mathcal C\) on which the time-slice convergence is valid.  This makes it
possible to apply the elliptic theorem at a fixed time without changing the
already determined full-energy limit \(\nu_t\).
\end{rem}

\subsection{Almost-every-time rectifiability}

We now use the parabolic dissipation to verify the forcing hypothesis in
Theorem~\ref{thm:forced-bethuel} at almost every time.

\begin{prop}[Almost-every-time rectifiability]
\label{prop:time-slice-rectifiability}
For almost every \(t\in(0,T)\), there exist a relatively closed countably
\(1\)-rectifiable set \(S_t\subset\T^2\) and a positive density
\(e_t\in L^1(\cH^1\llcorner S_t)\) such that
\begin{equation}\label{eq:time-slice-full-energy}
 \nu_t=e_t\,\cH^1\llcorner S_t.
\end{equation}
Moreover, \(\nu_t\) has positive lower one-dimensional density at
\(\nu_t\)-almost every point.
\end{prop}

\begin{proof}
Set
\[
 D_j(t)
 :=
 \int_{\T^2}
 \eps_j|\partial_tu_j(x,t)|^2\,\dd x.
\]
By \eqref{eq:global-dissipation},
\[
 \int_0^TD_j(t)\,\dd t\leq M_0.
\]
Fatou's lemma gives
\[
 \int_0^T\liminf_{j\to\infty}D_j(t)\,\dd t
 \leq
 \liminf_{j\to\infty}\int_0^TD_j(t)\,\dd t
 \leq M_0.
\]
Consequently,
\begin{equation}\label{eq:finite-slice-dissipation}
 \liminf_{j\to\infty}D_j(t)<\infty
\end{equation}
for almost every \(t\in(0,T)\).

Fix a time \(t\) satisfying
\eqref{eq:finite-slice-dissipation}, with \(t\in\mathcal C\) and
\(t\geq\tau>0\).  Choose a subsubsequence, depending on \(t\), such that
\[
 D_j(t)\leq\Lambda(t)<\infty.
\]
At this fixed time, \(v_j:=u_j(\cdot,t)\) solves
\begin{equation}\label{eq:time-slice-forced-equation}
 -\eps_j\Delta v_j+\frac1{\eps_j}DV(v_j)=F_j,
 \qquad
 F_j:=-\eps_j\partial_tu_j(\cdot,t).
\end{equation}
The forcing action is
\begin{equation}\label{eq:time-slice-forced-action}
 \int_{\T^2}\frac{|F_j|^2}{\eps_j}\,\dd x
 =
 \int_{\T^2}
 \eps_j|\partial_tu_j(\cdot,t)|^2\,\dd x
 =
 D_j(t)
 \leq\Lambda(t).
\end{equation}
The forced theorem uses the momentum density \(F_j\cdot\nabla v_j\),
which equals the negative of the parabolic flux
\(q_j^{\rm par}=\eps_j\partial_tu_j\cdot\nabla u_j\).
Accordingly, all stress identities below use the convention
\(\operatorname{div}T_j=-q_j^{\rm par}\).
The global energy bound gives
\[
 E_{\eps_j}(v_j;\T^2)\leq M_0,
\]
and Lemma~\ref{lem:positive-time-bounds} gives
\[
 \|v_j\|_{L^\infty(\T^2)}
 +\eps_j\|\nabla v_j\|_{L^\infty(\T^2)}
 \leq L_\tau.
\]
Thus the hypotheses of
Theorem~\ref{thm:forced-bethuel} hold locally in every coordinate disk of
\(\T^2\).  A finite covering by such disks and uniqueness of the measure
\(\nu_t\) yield a relatively closed countably \(1\)-rectifiable carrier
\(S_t\).

Although the subsubsequence depends on \(t\), the full-energy measures
already satisfy
\[
 \nu_j^t\stackrel{*}{\rightharpoonup}\nu_t
\]
along the original common subsequence because \(t\in\mathcal C\).  Hence
the full-energy limit produced by the forced elliptic theorem is necessarily
the previously fixed measure \(\nu_t\).  Theorem~\ref{thm:forced-bethuel}
therefore gives \eqref{eq:time-slice-full-energy} and the positive lower
density assertion.

Finally, choose a sequence \(\tau_m\downarrow0\).  The preceding conclusion
holds for almost every \(t\in(\tau_m,T)\) for every \(m\), and hence for
almost every \(t\in(0,T)\).
\end{proof}

\begin{rem}[Why only rectifiability is taken from the slice theorem]
At a fixed time, Theorem~\ref{thm:forced-bethuel} also produces subsequential
limits of the potential and gradient-tensor measures.  However, the
subsubsequence used above depends on \(t\).  Those fixed-time limits cannot
yet be identified with the time slices of a single spacetime limit.
Accordingly, Proposition~\ref{prop:time-slice-rectifiability} uses only the
full-energy conclusion, whose limit \(\nu_t\) was fixed in advance.  The
potential and gradient measures are identified below by spacetime
disintegration.
\end{rem}

\subsection{Spacetime disintegration and comparison of the energies}

\begin{prop}[Spacetime disintegration and energy comparison]
\label{prop:spacetime-disintegration}
After passing to a further subsequence, there exist weakly measurable
families of spatial Radon measures
\[
 (\zeta_t)_{t\in(0,T)},
 \qquad
 (\mu_{ab,t})_{t\in(0,T)}
\]
such that
\begin{equation}\label{eq:spacetime-zeta-mu-convergence}
 \zeta_j^t\,\dd t
 \stackrel{*}{\rightharpoonup}
 \zeta_t\,\dd t,
 \qquad
 \mu_{j,ab}^t\,\dd t
 \stackrel{*}{\rightharpoonup}
 \mu_{ab,t}\,\dd t.
\end{equation}
For every \(0<\tau<T\), there exists
\(C_\tau<\infty\) such that
\begin{equation}\label{eq:spacetime-comparison}
 \zeta_t\leq\nu_t\leq C_\tau\zeta_t
 \qquad\text{for almost every }t\in(\tau,T).
\end{equation}
Consequently, for almost every \(t\),
\[
 \zeta_t=\Theta_t\,\cH^1\llcorner S_t,
 \qquad
 \mu_{ab,t}=m_{ab,t}\,\cH^1\llcorner S_t
\]
for suitable densities \(\Theta_t>0\) and \(m_{ab,t}\).
\end{prop}

\begin{proof}
The diffuse inequalities
\[
 0\leq\zeta_j^t\leq\nu_j^t,
 \qquad
 |\mu_{j,ab}^t|\leq2\nu_j^t
\]
give uniform bounds for the corresponding spacetime measures.  After
extraction,
\[
 \zeta_j^t\,\dd t\stackrel{*}{\rightharpoonup}\bar\zeta,
 \qquad
 \mu_{j,ab}^t\,\dd t\stackrel{*}{\rightharpoonup}\bar\mu_{ab}.
\]
Passing the domination inequalities to the limit gives
\[
 0\leq\bar\zeta\leq\nu_t\,\dd t,
 \qquad
 |\bar\mu_{ab}|\leq2\nu_t\,\dd t.
\]
In particular, \(\bar\zeta\) and \(\bar\mu_{ab}\) are absolutely continuous
with respect to a measure that is itself absolutely continuous in the time
variable.  The disintegration theorem therefore gives weakly measurable
families \(\zeta_t\) and \(\mu_{ab,t}\) such that
\[
 \bar\zeta=\zeta_t\,\dd t,
 \qquad
 \bar\mu_{ab}=\mu_{ab,t}\,\dd t.
\]
Moreover,
\begin{equation}\label{eq:slice-dominations}
 0\leq\zeta_t\leq\nu_t,
 \qquad
 |\mu_{ab,t}|\leq2\nu_t
\end{equation}
for almost every \(t\).

It remains to prove the reverse comparison
\(\nu_t\leq C_\tau\zeta_t\).  Fix \(0<\tau<T\), and choose a spatial ball
\(B_r(x)\subset\T^2\) with \(r\) smaller than the injectivity radius.
At every \(t\in(\tau,T)\), the positive-time bound
\eqref{eq:positive-time-bounds} permits us to apply the scaled
energy--potential comparison
\eqref{eq:forced-energy-potential-scaled} to
\(v=u_j(\cdot,t)\) and
\[
 F=-\eps_j\partial_tu_j(\cdot,t).
\]
Thus
\begin{align}
 \nu_j^t\left(B_{\frac r2}(x)\right)
 \leq C_\tau\bigg\{&
 \zeta_j^t\left(B_{\frac{3r}{4}}(x)\right)
 \nonumber\\
 &+\frac{\eps_j}{r}
 \nu_j^t\left(
 B_{\frac{3r}{4}}(x)\setminus B_{\frac r2}(x)
 \right)
 \nonumber\\
 &+\eps_j^2
 \int_{B_r(x)}
 \eps_j|\partial_tu_j|^2\,\dd x
 \bigg\}.
 \label{eq:parabolic-slice-comparison}
\end{align}
Let \(\eta\in C_c((\tau,T))\) be non-negative.  Multiply
\eqref{eq:parabolic-slice-comparison} by \(\eta(t)\) and integrate in time.
The annular error is bounded by
\[
 \frac{C_\tau\eps_j}{r}
 \|\eta\|_{L^\infty}
 \int_\tau^T\nu_j^t(\T^2)\,\dd t
 \leq
 \frac{C_\tau\eps_j}{r}TM_0\|\eta\|_{L^\infty},
\]
which tends to zero.  The forcing error is bounded by
\begin{align*}
 &C_\tau\eps_j^2
 \|\eta\|_{L^\infty}
 \int_0^T\int_{\T^2}
 \eps_j|\partial_tu_j|^2\,\dd x\dd t\\
 &\qquad\leq
 C_\tau\eps_j^2M_0\|\eta\|_{L^\infty}
 \longrightarrow0.
\end{align*}

Choose a countable dense set of centers and radii for which all relevant
boundary circles have zero mass for the limiting spacetime measures.  After
passing to the spacetime limits, we obtain
\[
 \int_\tau^T
 \eta(t)\nu_t\left(B_{\frac r2}(x)\right)\dd t
 \leq
 C_\tau
 \int_\tau^T
 \eta(t)\zeta_t\left(B_{\frac{3r}{4}}(x)\right)\dd t.
\]
Since \(\eta\geq0\) is arbitrary,
\begin{equation}\label{eq:parabolic-expanded-ball-comparison}
 \nu_t\left(B_{\frac r2}(x)\right)
 \leq
 C_\tau
 \zeta_t\left(B_{\frac{3r}{4}}(x)\right)
\end{equation}
for almost every \(t\in(\tau,T)\), simultaneously for every ball in the
chosen countable family.  Let \(\mathcal B\) be the countable family of balls with rational
centres and radii used above.  Intersecting the corresponding full-measure
sets of times gives one set on which all inequalities in \(\mathcal B\)
hold simultaneously.  Approximate an arbitrary ball from inside and outside
by elements of \(\mathcal B\); monotonicity and continuity of Radon
measures give the same inequality whenever the relevant boundary circles
have zero mass.  The circular differentiation theorem then applies at
\(\nu_t\)-almost every rectifiable point.

Fix a time for which Proposition~\ref{prop:time-slice-rectifiability},
\eqref{eq:slice-dominations}, and
\eqref{eq:parabolic-expanded-ball-comparison} all hold.  Write
\[
 \nu_t=e_t\,\cH^1\llcorner S_t
\]
and
\[
 f_t:=\frac{\dd\zeta_t}{\dd\nu_t},
 \qquad 0\leq f_t\leq1.
\]
At \(\nu_t\)-almost every point \(x\), the set \(S_t\) has an approximate
tangent line, \(e_t\) is approximately continuous, and \(x\) is a
differentiation point for \(f_t\).  Therefore
\[
 \frac{\nu_t(B_{r/2}(x))}{\nu_t(B_r(x))}
 \longrightarrow\frac12,
\]
\[
 \frac{\nu_t(B_{3r/4}(x))}{\nu_t(B_r(x))}
 \longrightarrow\frac34,
\]
and
\[
 \frac{\zeta_t(B_{3r/4}(x))}
 {\nu_t(B_{3r/4}(x))}
 \longrightarrow f_t(x).
\]
Divide \eqref{eq:parabolic-expanded-ball-comparison} by
\(\nu_t(B_r(x))\) and let \(r\downarrow0\).  We obtain
\[
 \frac12
 \leq
 C_\tau\frac34 f_t(x).
\]
Hence
\[
 f_t(x)\geq\frac{2}{3C_\tau}
\]
for \(\nu_t\)-almost every \(x\).  Equivalently,
\[
 \nu_t\leq C'_\tau\zeta_t.
\]
After renaming the constant, this proves
\eqref{eq:spacetime-comparison}.

Since \(\nu_t\) and \(\zeta_t\) are mutually absolutely continuous and
\[
 \nu_t=e_t\,\cH^1\llcorner S_t,
\]
there exists \(\Theta_t>0\) such that
\[
 \zeta_t=\Theta_t\,\cH^1\llcorner S_t.
\]
Similarly, \(|\mu_{ab,t}|\leq2\nu_t\) gives
\[
 \mu_{ab,t}=m_{ab,t}\,\cH^1\llcorner S_t.
\]
\end{proof}

\begin{rem}[Why the comparison is proved in spacetime]
Applying Theorem~\ref{thm:forced-bethuel} separately at a fixed time would
give mutual absolute continuity only for limits along a subsequence that
may depend on that time.  Proposition~\ref{prop:spacetime-disintegration}
instead proves the comparison for the time slices of the single spacetime
limits selected in \eqref{eq:spacetime-zeta-mu-convergence}.  This
compatibility is essential for the evolution identities in the next
section.
\end{rem}

\subsection{The limiting stress and the Bethuel-type discrepancy relations}

We finally identify the direction of the limiting stress on almost every
time slice.

\begin{prop}[Limiting stress and slice-wise Bethuel-type discrepancy relations]
\label{prop:parabolic-bethuel-relations}
After passing to a further subsequence, there exists a spacetime
vector-valued Radon measure \(q\) such that
\[
 q_j\stackrel{*}{\rightharpoonup}q.
\]
Moreover,
\[
 q\ll\nu_t\,\dd t,
 \qquad
 \frac{\dd q}{\dd(\nu_t\,\dd t)}
 \in L^2(\nu_t\,\dd t;\R^2),
\]
and hence
\[
 q=q_t\,\dd t
\]
for a weakly measurable family of spatial vector measures \(q_t\).

For almost every \(t\in(0,T)\),
\begin{equation}\label{eq:parabolic-limit-stress-balance}
 \operatorname{div}_xT_t=-q_t,
 \qquad
 \operatorname{tr}T_t=2\zeta_t,
\end{equation}
where
\[
 T_t:=\nu_tI-\mu_t.
\]
Furthermore,
\begin{equation}\label{eq:parabolic-rank-one-stress}
 T_t
 =
 2\Theta_t\,\tau_t\otimes\tau_t\,
 \cH^1\llcorner S_t,
\end{equation}
where \(\tau_t\) is an approximate unit tangent to \(S_t\).

Consequently, in the approximate tangent--normal frame
\((\tau_t,n_t)\),
\begin{equation}\label{eq:parabolic-bethuel-relations}
 e_t=m_{n_tn_t,t},
 \qquad
 m_{n_t\tau_t,t}=0,
 \qquad
 2\Theta_t
 =
 m_{n_tn_t,t}-m_{\tau_t\tau_t,t}
\end{equation}
for \(\cH^1\)-almost every point of \(S_t\).  In particular,
\begin{equation}\label{eq:parabolic-rho-definition}
 \rho_t
 :=
 e_t-2\Theta_t
 =
 m_{\tau_t\tau_t,t}
 \geq0.
\end{equation}
\end{prop}

\begin{proof}
For every continuous compactly supported spacetime vector field
\(X=(X_1,X_2)\), Cauchy--Schwarz gives
\begin{align}
 \left|\int X\cdot\dd q_j\right|^2
 &=
 \left|
 \int_0^T\int_{\T^2}
 \eps_j\partial_tu_j\cdot
 \bigl(X_a\partial_au_j\bigr)\,\dd x\dd t
 \right|^2
 \nonumber\\
 &\leq
 \left(
 \int_0^T\int_{\T^2}
 \eps_j|\partial_tu_j|^2\,\dd x\dd t
 \right)
 \left(
 \int X_aX_b\,\dd\mu_{j,ab}^t\dd t
 \right)
 \nonumber\\
 &\leq
 M_0
 \int X_aX_b\,\dd\mu_{j,ab}^t\dd t.
 \label{eq:parabolic-action-duality}
\end{align}
In particular, the total variations of \(q_j\) are uniformly bounded.
After extraction,
\[
 q_j\stackrel{*}{\rightharpoonup}q.
\]
Passing to the limit in
\eqref{eq:parabolic-action-duality} gives
\begin{equation}\label{eq:parabolic-limit-action-duality}
 \left|\int X\cdot\dd q\right|^2
 \leq
 M_0
 \int X_aX_b\,\dd\mu_{ab,t}\dd t.
\end{equation}
Taking \(X=\varphi e_a\) and using
\[
 \mu_{aa,t}\leq2\nu_t
\]
gives
\[
 \left|\int\varphi\,\dd q_a\right|^2
 \leq
 2M_0\int\varphi^2\,\dd(\nu_t\,\dd t).
\]
The Riesz representation theorem on
\(L^2(\nu_t\,\dd t)\) therefore implies
\[
 q\ll\nu_t\,\dd t,
 \qquad
 \frac{\dd q}{\dd(\nu_t\,\dd t)}
 \in L^2(\nu_t\,\dd t;\R^2).
\]
Since \(\nu_t\,\dd t\) is absolutely continuous in the time variable,
\(q\) admits a time disintegration
\[
 q=q_t\,\dd t.
\]

Define the spacetime limiting stress by
\[
 T:=\nu_tI\,\dd t-\mu_t\,\dd t=T_t\,\dd t.
\]
Passing to the limit in
\[
 \operatorname{div}_xT_j=-q_j
\]
gives
\[
 \operatorname{div}_xT=-q
\]
in spacetime distributions.  Testing with fields of the form
\[
 X(x,t)=\eta(t)Y(x)
\]
and using the uniqueness of disintegration yields
\[
 \operatorname{div}_xT_t=-q_t
\]
for almost every \(t\).  Similarly, the exact identity
\[
 \operatorname{tr}T_j=2\zeta_j
\]
passes to the limit and gives
\[
 \operatorname{tr}T_t=2\zeta_t.
\]
This proves \eqref{eq:parabolic-limit-stress-balance}.

Fix a time \(t\) for which all the preceding conclusions hold.  Since
\[
 \nu_t=e_t\,\cH^1\llcorner S_t,
 \qquad
 \mu_{ab,t}=m_{ab,t}\,\cH^1\llcorner S_t,
\]
there exists a symmetric matrix field \(A_t\) such that
\begin{equation}\label{eq:parabolic-stress-density}
 T_t=A_t\,\cH^1\llcorner S_t.
\end{equation}
We identify \(A_t\) at a typical point \(x\in S_t\).  Choose \(x\) such
that:

\begin{enumerate}
\item \(S_t\) has an approximate tangent line \(L_x=T_xS_t\) at \(x\);

\item \(x\) is an approximate Lebesgue point of \(A_t\) and \(\Theta_t\)
with respect to \(\cH^1\llcorner S_t\);

\item \(q_t\) has no atom at \(x\).
\end{enumerate}

These properties hold for \(\cH^1\)-almost every \(x\in S_t\).  The last
one follows from
\[
 q_t\ll\nu_t
 \ll\cH^1\llcorner S_t.
\]

Let
\[
 \eta_{x,r}(y):=\frac{y-x}{r}
\]
and define the normalized blow-up
\[
 T_{t,x,r}
 :=
 \frac1r(\eta_{x,r})_\#T_t.
\]
Rectifiability and approximate continuity imply
\begin{equation}\label{eq:parabolic-stress-tangent-limit}
 T_{t,x,r}
 \stackrel{*}{\rightharpoonup}
 A_t(x)\,\cH^1\llcorner L_x
 \qquad\text{as }r\downarrow0.
\end{equation}

We next compute the divergence under this scaling.  For
\(Y\in C_c^1(\R^2;\R^2)\),
\begin{align*}
 \left\langle
 \operatorname{div}T_{t,x,r},Y
 \right\rangle
 &=
 -\frac1r
 \int
 DY\left(\frac{y-x}{r}\right):\dd T_t(y)\\
 &=
 \left\langle
 (\eta_{x,r})_\#(\operatorname{div}T_t),Y
 \right\rangle.
\end{align*}
Thus
\[
 \operatorname{div}T_{t,x,r}
 =
 -(\eta_{x,r})_\#q_t.
\]
For every fixed \(R>0\),
\[
 \left|
 (\eta_{x,r})_\#q_t
 \right|(B_R)
 =
 |q_t|(B_{Rr}(x))
 \longrightarrow
 |q_t|(\{x\})
 =0.
\]
Hence the limit in
\eqref{eq:parabolic-stress-tangent-limit} is divergence-free:
\begin{equation}\label{eq:parabolic-tangent-divergence-free}
 \operatorname{div}
 \left(
 A_t(x)\,\cH^1\llcorner L_x
 \right)
 =0.
\end{equation}

Let \(\tau_t(x)\) be a unit vector spanning \(L_x\), and let
\(n_t(x)\) be a unit normal.  We claim that
\begin{equation}\label{eq:parabolic-matrix-kernel}
 A_t(x)n_t(x)=0.
\end{equation}
Indeed, for any \(Y\in C_c^1(\R^2;\R^2)\),
\begin{align*}
 &\int_{L_x}A_t(x):DY\,\dd\cH^1\\
 &\qquad=
 \int_{L_x}
 \bigl(A_t(x)\tau_t\bigr)\cdot\partial_{\tau_t}Y
 \,\dd\cH^1
 +
 \int_{L_x}
 \bigl(A_t(x)n_t\bigr)\cdot\partial_{n_t}Y
 \,\dd\cH^1.
\end{align*}
The first integral vanishes by one-dimensional integration by parts because
\(A_t(x)\tau_t\) is constant.  Since the normal derivative
\(\partial_{n_t}Y|_{L_x}\) can be prescribed arbitrarily,
\eqref{eq:parabolic-tangent-divergence-free} forces
\eqref{eq:parabolic-matrix-kernel}.

The matrix \(A_t(x)\) is symmetric.  In the orthonormal basis
\((\tau_t,n_t)\), condition \eqref{eq:parabolic-matrix-kernel} therefore
implies
\[
 A_t(x)=c_t(x)\tau_t\otimes\tau_t
\]
for some scalar \(c_t(x)\).  Taking the trace and using
\[
 \operatorname{tr}T_t
 =
 2\Theta_t\,\cH^1\llcorner S_t
\]
gives
\[
 c_t(x)=2\Theta_t(x).
\]
Thus
\[
 A_t(x)
 =
 2\Theta_t(x)\tau_t\otimes\tau_t
\]
for \(\cH^1\)-almost every \(x\in S_t\), proving
\eqref{eq:parabolic-rank-one-stress}.

Finally,
\[
 T_t=\nu_tI-\mu_t
\]
and \eqref{eq:parabolic-rank-one-stress} imply, at the level of densities,
\[
 e_tI-m_t
 =
 2\Theta_t\tau_t\otimes\tau_t.
\]
Taking the \(n_t\otimes n_t\), \(n_t\otimes\tau_t\), and
\(\tau_t\otimes\tau_t\) components gives, respectively,
\[
 e_t=m_{n_tn_t,t},
\]
\[
 m_{n_t\tau_t,t}=0,
\]
and
\[
 e_t-m_{\tau_t\tau_t,t}=2\Theta_t.
\]
Using the first identity in the last equation yields
\[
 2\Theta_t
 =
 m_{n_tn_t,t}-m_{\tau_t\tau_t,t}.
\]
This proves \eqref{eq:parabolic-bethuel-relations}.

Since each diffuse spatial-gradient tensor is positive semidefinite as a
matrix-valued measure, the same is true of its weak-star limit \(\mu_t\).
Consequently, its tangential density satisfies
\[
 m_{\tau_t\tau_t,t}\geq0.
\]
Equivalently,
\[
 \rho_t
 =
 e_t-2\Theta_t
 =
 m_{\tau_t\tau_t,t}
 \geq0,
\]
which proves \eqref{eq:parabolic-rho-definition}.
\end{proof}

\begin{rem}[Output of the compactness argument]
The conclusions of this section are static statements on almost every time
slice:
\[
 \nu_t=e_t\,\cH^1\llcorner S_t,
 \qquad
 \zeta_t=\Theta_t\,\cH^1\llcorner S_t,
\]
\[
 \mu_{ab,t}=m_{ab,t}\,\cH^1\llcorner S_t,
 \qquad
 T_t
 =
 2\Theta_t\tau_t\otimes\tau_t\,
 \cH^1\llcorner S_t,
\]
together with the Bethuel relations
\[
 e_t=m_{n_tn_t,t},
 \qquad
 m_{n_t\tau_t,t}=0,
 \qquad
 2\Theta_t
 =
 m_{n_tn_t,t}-m_{\tau_t\tau_t,t}.
\]
No evolution law has yet been identified.  The next section combines these
slice-wise identities with the limiting momentum and dissipation to derive
the normal curvature balance, the tangential internal-mobility law, and the
localized energy inequality.
\end{rem}

\section{The two-mobility measure evolution}
\label{sec:two-mobility-evolution}

The purpose of this section is to convert the compactness and structural
information obtained in the preceding sections into an evolution law for
the limiting interfacial measures.  There are three main points.  First, the
positivity of the diffuse spacetime Gram matrix gives a sharp lower bound
for the limiting action.  Second, the stress identity identifies the
limiting momentum with the first variation of the weighted rectifiable
set.  Its normal and tangential components produce two distinct mobility
laws.  Third, the exact diffuse local energy identity passes to the limit
and yields the localized energy inequality, including a non-negative
dissipation-defect measure.

Throughout this section, all assertions concerning a fixed time are
understood to hold at a time for which the conclusions of
Propositions~\ref{prop:time-slice-rectifiability} and
\ref{prop:parabolic-bethuel-relations} hold.  We continue to write
\[
 \lambda_t:=\cH^1\llcorner S_t,
 \qquad
 \nu_t=e_t\lambda_t,
 \qquad
 \zeta_t=\Theta_t\lambda_t,
\]
and
\[
 \mu_t
 =\rho_t\,\tau_t\otimes\tau_t\,\lambda_t
  +e_t\,n_t\otimes n_t\,\lambda_t,
 \qquad
 \rho_t=e_t-2\Theta_t\geq0.
\]

\begin{lem}[Jointly measurable slice representatives]
\label{lem:joint-measurable-slices}
After modifying the slice representatives on a set of times of measure zero,
all objects used below may be chosen jointly measurable in \(x,t\):
\(S_t\), the densities \(e_t,\Theta_t,m_t\), and a unit tangent field
\(\tau(x,t)\) on the rectifiable carrier.  A measurable unit normal is
then fixed by \(n(x,t)=(-\tau_2,\tau_1)\).  The measures
\(\lambda_t=\mathcal H^1\llcorner S_t\) and the fields obtained by
orthogonal projection are weakly measurable.
\end{lem}

\begin{proof}
The spacetime measures \(\nu,\zeta,\mu\) are fixed before slicing.
Disintegration and measurable Radon--Nikodym differentiation give jointly
measurable densities.  The measurable tangent theorem for rectifiable
measures gives a jointly measurable approximate tangent line to the
rectifiable slices.  Choose its orientation by the Borel rule
\(\tau_1>0\), or \(\tau_1=0,\tau_2>0\), and set
\(n=(-\tau_2,\tau_1)\).  All scalar and matrix fields below are then
measurable combinations of these densities and projections.  We henceforth
work on the intersection of the countably many full-measure sets occurring
in the preceding propositions.
\end{proof}

\subsection{The sharp action lower bound}

The diffuse action, momentum, and spatial gradient measures are not
independent.  They are the entries of a single positive-semidefinite Gram
matrix.  Positivity is preserved under weak-star convergence and therefore
imposes a pointwise constraint on the densities of the limiting measures.

\begin{prop}[Sharp action lower bound]
\label{prop:sharp-action}
Let
\[
 \alpha=a\,\lambda_t\dd t+\alpha^{\rm s},
 \qquad
 \alpha^{\rm s}\perp\lambda_t\dd t,
\]
be the decomposition in \eqref{eq:action-decomposition}, and write
\[
 q_t=\mathbf q_t\,\lambda_t,
 \qquad
 \mathbf q_t=q_{\tau,t}\tau_t+q_{n,t}n_t.
\]
Then
\begin{equation}\label{eq:sharp-action-repeated}
 a\geq
 \frac{\abs{q_{n,t}}^2}{e_t}
 +\frac{\abs{q_{\tau,t}}^2}{\rho_t}
 \quad
 \text{for }\lambda_t\dd t\text{-almost every }(x,t).
\end{equation}
Here the second quotient is understood as:
\[
 \frac{\abs{q_{\tau,t}}^2}{\rho_t}
 :=
 \begin{cases}
  \abs{q_{\tau,t}}^2/\rho_t,
       &\rho_t>0,\\
  0,   &\rho_t=0\text{ and }q_{\tau,t}=0,\\
  +\infty,
       &\rho_t=0\text{ and }q_{\tau,t}\neq0.
 \end{cases}
\]
In particular,
\begin{equation}\label{eq:zero-rho-zero-momentum}
 q_{\tau,t}=0
 \quad\text{for }\lambda_t\dd t\text{-almost every point of }
 \{\rho_t=0\}.
\end{equation}
\end{prop}
\begin{rem}[Interpretation of the action bound]
The inequality
\[
 a\geq \frac{|q_{n,t}|^2}{e_t}
   +\frac{|q_{\tau,t}|^2}{\rho_t}
\]
is the positivity condition inherited from the Gram matrix of
\(\partial_tu_\eps\), \(\partial_{n_t}u_\eps\), and
\(\partial_{\tau_t}u_\eps\).  It states that the interfacial dissipation
controls both the normal energy flux and, when \(\rho_t>0\), the tangential
internal flux.  If \(\rho_t=0\), positivity forces \(q_{\tau,t}=0\), and the
bound reduces to \(a\geq |q_{n,t}|^2/e_t\).  Thus, in the scalar
equipartition regime, where \(e_t=2\Theta_t\), one recovers the usual
normal-motion dissipation estimate.
\end{rem}

\begin{proof}
Regard
\[
 Z_j:=
 \begin{pmatrix}
  \partial_tu_j\\
  \partial_1u_j\\
  \partial_2u_j
 \end{pmatrix}
\]
as a \(3\) by \(k\) matrix whose rows are vectors in \(\R^k\).  Its Gram
matrix is \(Z_jZ_j^{\mathsf T}\), where \(^{\mathsf T}\) denotes the transpose.

Define the symmetric matrix-valued measure
\begin{equation}\label{eq:diffuse-gram-matrix}
 \mathcal G_j
 :=
 \eps_j Z_jZ_j^{\mathsf T}\dd x\dd t
 =
 \begin{pmatrix}
  \alpha_j&q_j^{\mathsf T}\\
  q_j&\mu_j
 \end{pmatrix}.
\end{equation}
For every \(c=(c_0,c_1,c_2)\in\R^3\) and every non-negative
\(\psi\in C_c(\T^2\times(0,T))\),
\begin{align}
 \int\psi\,c^{\mathsf T}\dd\mathcal G_j\,c
 &=
 \int\psi\,\eps_j
 \left|
  c_0\partial_tu_j
  +c_1\partial_1u_j
  +c_2\partial_2u_j
 \right|^2
 \dd x\dd t
 \notag\\
 &\geq0.
 \label{eq:diffuse-gram-positive}
\end{align}
Thus \(\mathcal G_j\) is positive semidefinite as a matrix-valued
measure.

After passing to the subsequence fixed in
\eqref{eq:spacetime-measure-convergence}, we have
\[
 \mathcal G_j\stackrel{*}{\rightharpoonup}
 \mathcal G
 :=
 \begin{pmatrix}
  \alpha&q^{\mathsf T}\\
  q&\mu
 \end{pmatrix}.
\]
Passing to the limit in \eqref{eq:diffuse-gram-positive} shows that
\(\mathcal G\) is positive semidefinite in the following precise sense:
for every constant vector \(c\in\mathbb R^3\), the scalar measure
\(c^{\mathsf T}\mathcal Gc\) is nonnegative.  Since the absolutely
continuous density of a positive matrix\-valued measure is positive
semidefinite almost everywhere, the fixed\-coordinate density of \(\mathcal G\)
is positive semidefinite.

Set
\[
 \Lambda:=\lambda_t\dd t.
\]
The measures \(q\) and \(\mu\) are absolutely continuous with respect to
\(\Lambda\), whereas
\[
 \alpha=a\Lambda+\alpha^{\rm s}.
\]
Consequently, the Lebesgue decomposition of \(\mathcal G\) relative to
\(\Lambda\) is
\begin{equation}\label{eq:limit-gram-decomposition}
 \mathcal G
 =
 G\,\Lambda
 +
 \begin{pmatrix}
  \alpha^{\rm s}&0\\
  0&0
 \end{pmatrix},
\end{equation}
where, in the orthonormal spacetime frame
\((\partial_t,\tau_t,n_t)\),
\begin{equation}\label{eq:limit-gram-density}
 G=
 \begin{pmatrix}
  a&q_{\tau,t}&q_{n,t}\\
  q_{\tau,t}&\rho_t&0\\
  q_{n,t}&0&e_t
 \end{pmatrix}.
\end{equation}
Here the entries of \(G\) are determined by the limiting Gram structure.  By the preceding lemma, the orthogonal matrix \(R(x,t)\) with columns
\((\tau_t,n_t)\) is measurable.  The spatial block is changed by the
pointwise congruence \(R^{\mathsf T}\mu R\), which preserves positive
semidefiniteness; no additional weak-limit operation is involved.

Indeed, with respect to the frame \((\partial_t,\tau_t,n_t)\), the
\((\tau_t,\tau_t)\), \((n_t,n_t)\), and mixed spatial components are given by
Bethuel's relations
\[
m_{\tau\tau,t}=\rho_t,\qquad
m_{n n,t}=e_t,\qquad
m_{n\tau,t}=0.
\]
The time--tangential and time--normal entries are precisely the components
\(q_{\tau,t}\) and \(q_{n,t}\) of the limiting momentum, while the
time--time entry is the interfacial dissipation density \(a\).

The absolutely continuous density of a positive-semidefinite
matrix-valued measure is positive semidefinite almost everywhere.
Therefore \(G\geq0\) for \(\Lambda\)-almost every \((x,t)\).

At points where \(\rho_t>0\), positivity gives, for every
\((s,\xi_\tau,\xi_n)\in\R^3\),
\[
 0\leq
 as^2+2s q_{\tau,t}\xi_\tau
 +2s q_{n,t}\xi_n
 +\rho_t\xi_\tau^2+e_t\xi_n^2.
\]
Since \(e_t>0\) for \(\lambda_t\)-almost every point, we may take
\(s=1\) and minimize the right-hand side with respect to
\(\xi_\tau\) and \(\xi_n\).  The minimizing values are
\[
 \xi_\tau=-\frac{q_{\tau,t}}{\rho_t},
 \qquad
 \xi_n=-\frac{q_{n,t}}{e_t}.
\]
It follows that
\[
 0\leq
 a-\frac{\abs{q_{\tau,t}}^2}{\rho_t}
  -\frac{\abs{q_{n,t}}^2}{e_t},
\]
which is \eqref{eq:sharp-action-repeated} on \(\{\rho_t>0\}\).

It remains to consider \(\rho_t=0\).  The principal minor of \(G\)
corresponding to the time and tangential directions is
\[
 \begin{pmatrix}
  a&q_{\tau,t}\\
  q_{\tau,t}&0
 \end{pmatrix}.
\]
Such a matrix can be positive semidefinite only if
\(q_{\tau,t}=0\).  Indeed, otherwise the quadratic polynomial
\[
 as^2+2q_{\tau,t}s\xi_\tau
\]
takes negative values for a suitable choice of \(s\) and \(\xi_\tau\).
This proves \eqref{eq:zero-rho-zero-momentum}.  Minimization in the normal
variable then gives \(a\geq\abs{q_{n,t}}^2/e_t\), completing the proof.
\end{proof}

\subsection{Weighted curvature and the two mobility fields}

We next interpret the stress identity as a first-variation identity.  The
relevant geometric object is the weighted varifold
\[
 \mathbf V_t:=\mathbf V(S_t,\Theta_t),
\]
whose weight measure is \(\zeta_t=\Theta_t\lambda_t\).  Its first variation
is
\begin{equation}\label{eq:weighted-varifold-definition}
 \delta\mathbf V_t(X)
 :=
 \int_{S_t}\Theta_t\,\tau_t\otimes\tau_t:DX
 \dd\lambda_t.
\end{equation}

\begin{defn}[regular branch]\label{defn:regular-branch}
Fix a time at which the conclusions of the preceding sections hold.  A
\emph{regular branch} is an open subset of \(S_t\) represented by a single
embedded \(C^1\) arclength curve
\[
 \gamma:I\longrightarrow\T^2,
 \qquad
 \abs{\gamma'(s)}=1,
\]
which contains no junction point and has the following local density
bounds: for every \(I'\Subset I\), there are constants
\(0<c_{I'}\leq C_{I'}<\infty\) such that
\[
 c_{I'}\leq e_t,\Theta_t\leq C_{I'}
 \quad\text{for almost every point of }\gamma(I').
\]
\end{defn}

\begin{prop}[Weighted curvature and mobility identities]
\label{prop:mobility-identities}
For almost every \(t\), the weighted varifold \(\mathbf V_t\) has
generalized curvature
\[
 H_{\Theta,t}\in
 L^2_{\rm loc}
 \bigl(\T^2\times(0,T),\zeta_t\dd t;\R^2\bigr),
\]
and
\begin{equation}\label{eq:weighted-first-variation-repeated}
 \delta\mathbf V_t(X)
 =-\int_{S_t}H_{\Theta,t}\cdot X\,\dd\zeta_t,
 \qquad
 q_t=-2\Theta_tH_{\Theta,t}\lambda_t.
\end{equation}

On every regular branch,
\[
 \Theta_t\in W^{1,2}_{\rm loc},
 \qquad
 \tau_t\in W^{1,2}_{\rm loc},
\]
and the branch is \(W^{2,2}_{\rm loc}\).  Moreover, $H_{\Theta,t}$ has the following normal-tangential decomposition:
\begin{equation}\label{eq:weighted-curvature-splitting-repeated}
 H_{\Theta,t}
 =H_{S_{t}}+\frac{1}{\Theta_{t}}\nabla_{S_{t}}\Theta_{t}=
 H_{S_t}+\nabla_{S_t}\log\Theta_t,
\end{equation}
where
\[
 H_{S_t}=\partial_s\tau_t
 \quad\text{and}\quad
 \nabla_{S_t}\Theta_t
 =(\partial_s\Theta_t)\tau_t
\]
on the branch.

With the global projection convention
\[
 H_{S_t}:=(n_t\otimes n_t)H_{\Theta,t},
 \qquad
 \nabla_{S_t}\Theta_t
 :=\Theta_t(\tau_t\otimes\tau_t)H_{\Theta,t},
\]
the fields in \eqref{eq:velocity-definitions} satisfy
\begin{equation}\label{eq:two-mobility-law-repeated}
 e_t\mathcal V_t=2\Theta_tH_{S_t},
 \qquad
 \rho_t\mathcal W_t=2\nabla_{S_t}\Theta_t.
\end{equation}
In addition,
\begin{equation}\label{eq:flux-decomposition-repeated}
 -\mathbf q_t=e_t\mathcal V_t+\rho_t\mathcal W_t.
\end{equation}
\end{prop}

\begin{proof}
We divide the proof into four steps.

\smallskip
\noindent
\emph{Step 1: the stress identity as a first variation.}
Let
\(X\in C_c^1(\T^2\times(0,T);\R^2)\).
Using
\[
 T_t
 =2\Theta_t\tau_t\otimes\tau_t\,\lambda_t
\]
and the definition \eqref{eq:weighted-varifold-definition}, we obtain
\begin{align}
 2\int_0^T\delta\mathbf V_t(X(\cdot,t))\dd t
 &=
 \int_0^T\int_{S_t}
 2\Theta_t\tau_t\otimes\tau_t:DX
 \dd\lambda_t\dd t
 \notag\\
 &=
 \int_0^T\int_{\T^2}T_t:DX\dd t.
 \label{eq:stress-to-first-variation}
\end{align}
Since \(\diver_xT_t=-q_t\), the distributional definition of divergence
gives
\[
 \int_0^T\int_{\T^2}T_t:DX\dd t
 =
 \int_{\T^2\times(0,T)}X\cdot\dd q.
\]
Consequently,
\begin{equation}\label{eq:spacetime-first-variation}
 2\int_0^T\delta\mathbf V_t(X(\cdot,t))\dd t
 =
 \int_{\T^2\times(0,T)}X\cdot\dd q.
\end{equation}
Disintegration in time shows that, for almost every \(t\),
\begin{equation}\label{eq:slice-first-variation}
 2\delta\mathbf V_t(X)
 =
 \int_{S_t}\mathbf q_t\cdot X\dd\lambda_t.
\end{equation}

\smallskip
\noindent
\emph{Step 2: absolute continuity with respect to the weighted measure.}
Let
\[
 \Lambda:=\lambda_t\,\dd t,
 \qquad
 \nu=\nu_t\,\dd t=e_t\,\Lambda,
 \qquad
 \zeta_t\,\dd t=\Theta_t\,\Lambda.
\]
By the previous step,
\[
 q\ll \nu,
 \qquad
 g:=\frac{\dd q}{\dd\nu}
 \in L^2(\nu;\mathbb R^2).
\]
Since \(q_t=\mathbf q_t\,\lambda_t\), the Radon--Nikodym relation
\(q=g\,\nu\) gives
\[
 \mathbf q_t=e_tg
 \qquad\text{for }\Lambda\text{-almost every }(x,t).
\]

Fix \(\tau>0\).  The energy--potential comparison
\[
 \nu_t\leq C_\tau\zeta_t
 \qquad\text{for a.e. }t\in(\tau,T)
\]
is equivalent to
\[
 e_t\leq C_\tau\Theta_t
 \qquad \Lambda\text{-almost everywhere on }(\tau,T).
\]
Consequently,
\[
 \frac{|\mathbf q_t|^2}{\Theta_t}
 =
 \frac{e_t^2}{\Theta_t}|g|^2
 \leq C_\tau e_t|g|^2.
\]
Therefore,
\[
 \int_\tau^T\int_{S_t}
 \frac{|\mathbf q_t|^2}{\Theta_t}
 \,\dd\lambda_t\dd t
 \leq
 C_\tau\int_\tau^T\int_{\T^2}
 |g|^2\,\dd\nu_t\dd t
 <\infty.
\]

We may thus define, for \(\zeta_t\dd t\)-almost every \((x,t)\),
\begin{equation}\label{eq:weighted-curvature-definition}
 H_{\Theta,t}:=
 -\frac{\mathbf q_t}{2\Theta_t}.
\end{equation}
The preceding estimate implies
\[
 \int_\tau^T\int_{S_t}
 |H_{\Theta,t}|^2\,\dd\zeta_t\dd t
 =
 \frac14
 \int_\tau^T\int_{S_t}
 \frac{|\mathbf q_t|^2}{\Theta_t}
 \,\dd\lambda_t\dd t
 <\infty.
\]
Since \(\tau>0\) is arbitrary, this yields
\[
 H_{\Theta,t}\in
 L^2_{\mathrm{loc}}
 \bigl(\T^2\times(0,T),\zeta_t\dd t;\mathbb R^2\bigr).
\]

Finally, combining
\eqref{eq:slice-first-variation} with
\eqref{eq:weighted-curvature-definition}, we obtain
\[
 \delta\mathbf V_t(X)
 =
 -\int_{S_t}H_{\Theta,t}\cdot X\,\dd\zeta_t,
 \qquad
 q_t
 =
 -2\Theta_tH_{\Theta,t}\,\lambda_t.
\]
This proves \eqref{eq:weighted-first-variation-repeated}.

\smallskip
\noindent
\emph{Step 3: regularity and curvature splitting on a regular branch.}
Fix a time \(t\) for which the preceding conclusions hold, and let
\(\gamma:I\to S_t\) be an arclength parametrization of a regular branch.
After choosing the orientation, we have
\[
 \tau_t(\gamma(s))=\gamma'(s)
 \qquad\text{for a.e. }s\in I.
\]

Let \(X\in C_c^1(\mathbb T^2;\mathbb R^2)\) be supported in the branch.
Using
\[
 T_t=2\Theta_t\,\tau_t\otimes\tau_t\,\lambda_t
\]
and \(\operatorname{div}T_t=-q_t\), we obtain
\[
 \int_I
 2\Theta_t(\gamma(s))\tau_t(\gamma(s))
 \cdot\partial_s(X\circ\gamma)(s)\,\dd s
 =
 \int_I
 \mathbf q_t(\gamma(s))\cdot X(\gamma(s))\,\dd s.
\]
Since \(X\circ\gamma\) can be chosen as an arbitrary compactly supported
test function on \(I\), it follows that
\begin{equation}\label{eq:weighted-tangent-derivative}
 \partial_s(\Theta_t\tau_t)
 =
 -\frac12\mathbf q_t
 \qquad\text{in }\mathcal D'(I).
\end{equation}

By the local bounds in
Definition~\ref{defn:regular-branch}, \(\Theta_t\) and \(e_t\) are bounded
above and bounded away from zero on compact subintervals of \(I\).
Moreover, the sharp action estimate implies locally
\[
 \int_I\frac{|\mathbf q_t|^2}{\Theta_t}\,\dd s<\infty.
\]
Hence
\[
 \mathbf q_t\in L^2_{\rm loc}(I;\mathbb R^2).
\]
Equation \eqref{eq:weighted-tangent-derivative} therefore yields
\[
 \Theta_t\tau_t\in W^{1,2}_{\rm loc}(I;\mathbb R^2).
\]

Since \(|\tau_t|=1\), we have
\[
 \Theta_t=|\Theta_t\tau_t|.
\]
On every compact subinterval, \(\Theta_t\) is bounded below by a positive
constant.  The Sobolev chain rule applied to the Lipschitz map
\(z\mapsto |z|\) consequently gives
\[
 \Theta_t\in W^{1,2}_{\rm loc}(I).
\]
Applying the chain rule once more to
\[
 \tau_t=\frac{\Theta_t\tau_t}{|\Theta_t\tau_t|}
\]
gives
\[
 \tau_t\in W^{1,2}_{\rm loc}(I;\mathbb R^2).
\]
Since \(\gamma'=\tau_t\), it follows that
\[
 \gamma\in W^{2,2}_{\rm loc}(I;\mathbb R^2).
\]

We may now use the weak product rule in
\eqref{eq:weighted-tangent-derivative}:
\[
 \partial_s(\Theta_t\tau_t)
 =
 (\partial_s\Theta_t)\tau_t
 +\Theta_t\partial_s\tau_t.
\]
Differentiating \(|\tau_t|^2=1\) gives
\[
 \tau_t\cdot\partial_s\tau_t=0,
\]
so \(\partial_s\tau_t\) is normal to the branch.  Define
\[
 H_{S_t}:=\partial_s\tau_t,
 \qquad
 \nabla_{S_t}\Theta_t:=(\partial_s\Theta_t)\tau_t.
\]
Then
\[
 \frac1{\Theta_t}\partial_s(\Theta_t\tau_t)
 =
 H_{S_t}
 +\frac{\partial_s\Theta_t}{\Theta_t}\tau_t
 =
 H_{S_t}+\nabla_{S_t}\log\Theta_t.
\]
Finally, by
\eqref{eq:weighted-curvature-definition} and
\eqref{eq:weighted-tangent-derivative},
\[
 H_{\Theta,t}
 =
 -\frac{\mathbf q_t}{2\Theta_t}
 =
 \frac1{\Theta_t}\partial_s(\Theta_t\tau_t).
\]
Therefore
\[
 H_{\Theta,t}
 =
 H_{S_t}+\nabla_{S_t}\log\Theta_t,
\]
which proves \eqref{eq:weighted-curvature-splitting-repeated}.

\smallskip
\noindent
\emph{Step 4: normal and tangential mobility laws.}
The relation
\[
 \mathbf q_t=-2\Theta_tH_{\Theta,t}
\]
and the orthogonal decomposition of \(H_{\Theta,t}\) give
\begin{equation}\label{eq:momentum-curvature-components}
 \mathbf q_t
 =
 -2\Theta_tH_{S_t}
 -2\nabla_{S_t}\Theta_t.
\end{equation}
The first term is normal and the second is tangential.  Therefore
\[
 q_{n,t}n_t=-2\Theta_tH_{S_t},
 \qquad
 q_{\tau,t}\tau_t=-2\nabla_{S_t}\Theta_t.
\]
Using
\[
 \mathcal V_t=-\frac{q_{n,t}}{e_t}n_t,
 \qquad
 \mathcal W_t=-\frac{q_{\tau,t}}{\rho_t}\tau_t
\]
on \(\{\rho_t>0\}\), we obtain
\[
 e_t\mathcal V_t=2\Theta_tH_{S_t},
 \qquad
 \rho_t\mathcal W_t=2\nabla_{S_t}\Theta_t.
\]

On \(\{\rho_t=0\}\), Proposition~\ref{prop:sharp-action} gives
\(q_{\tau,t}=0\).  The tangential part of
\eqref{eq:momentum-curvature-components} then implies
\[
 \nabla_{S_t}\Theta_t=0
 \quad\text{on }\{\rho_t=0\}.
\]
Thus the second mobility identity remains valid there after setting
\(\mathcal W_t=0\).  Finally,
\[
 -\mathbf q_t
 =-q_{n,t}n_t-q_{\tau,t}\tau_t
 =e_t\mathcal V_t+\rho_t\mathcal W_t,
\]
which proves \eqref{eq:flux-decomposition-repeated}.
\end{proof}

\begin{rem}[Geometric meaning of the weighted curvature and mobility fields]
\label{rem:weighted-curvature-mobility}
The vector \(H_{\Theta,t}\) is the generalized curvature of the weighted
varifold \(\mathbf V(S_t,\Theta_t)\).  It decomposes into the normal
geometric curvature of the branch and a tangential contribution generated by
the variation of the weight:
\[
 H_{\Theta,t}
 =
 H_{S_t}+\nabla_{S_t}\log\Theta_t,
 \qquad
 H_{S_t}=(n_t\otimes n_t)H_{\Theta,t},
 \qquad
 \nabla_{S_t}\log\Theta_t
 =(\tau_t\otimes\tau_t)H_{\Theta,t}.
\]
Thus \(H_{S_t}\) bends the interface in the normal direction, whereas
\(\nabla_{S_t}\log\Theta_t\) records the tangential variation of the
interfacial energy density and contributes to the weighted first variation
even when the underlying curve is geometrically fixed.

Correspondingly, the limiting energy flux has two channels:
\[
 -\frac{\dd q}{\dd\nu}
 =
 \mathcal V_t+\frac{\rho_t}{e_t}\mathcal W_t,
 \qquad
 e_t\mathcal V_t=2\Theta_tH_{S_t},
 \qquad
 \rho_t\mathcal W_t=2\nabla_{S_t}\Theta_t.
\]
The field \(\mathcal V_t\) is the normal mobility driven by geometric
curvature, while \(\mathcal W_t\) is a tangential internal mobility describing
the redistribution of residual vectorial microstructure.  These fields need
not coincide with the velocity of the reduced phase boundary when hidden
interfaces or collapsed layers are present.

In the scalar equipartition regime, \(\rho_t=0\) and the tangential channel
disappears: \(q_{\tau,t}=0\), \(H_{\Theta,t}=H_{S_t}\) when
\(\Theta_t\) is constant along the branch, and the usual normal
mean-curvature-flow law is recovered.
\end{rem}

\subsection{The localized energy inequality}

The diffuse local energy balance is an equality.  In the limit, the part of
the action detected by the momentum is determined by the two mobility
fields.  Any remaining action is represented by a non-negative measure
\(\gamma\).

\begin{prop}[Localized energy identity and inequality]
\label{prop:localized-energy}
Define
\begin{equation}\label{eq:dissipation-density}
 D_t
 :=
 e_t\abs{\mathcal V_t}^2
 +\rho_t\abs{\mathcal W_t}^2.
\end{equation}
Then
\begin{equation}\label{eq:dissipation-defect-repeated}
 \gamma
 :=
 \alpha-D_t\lambda_t\dd t
\end{equation}
is a non-negative Radon measure.

For every non-negative
\(\phi\in C^1(\T^2\times[0,T])\) and almost every
\(0<t_1<t_2<T\), the sharper identity
\begin{align}
 &\int_{S_{t_2}}\phi(\cdot,t_2)e_{t_2}\dd\lambda_{t_2}
 -\int_{S_{t_1}}\phi(\cdot,t_1)e_{t_1}\dd\lambda_{t_1}
 \notag\\
 &=
 \int_{t_1}^{t_2}\int_{S_t}
 \Bigg[
 e_t\partial_t\phi
 +2\Theta_tH_{S_t}\cdot\nabla\phi
 +2\nabla_{S_t}\Theta_t\cdot\nabla_{S_t}\phi
 \notag\\
 &\hspace{8em}
 -\phi\left(
  \frac{4\Theta_t^2}{e_t}\abs{H_{S_t}}^2
  +\frac{4}{\rho_t}
       \abs{\nabla_{S_t}\Theta_t}^2
 \right)
 \Bigg]\dd\lambda_t\dd t
 -\int_{\T^2\times(t_1,t_2)}\phi\dd\gamma
 \label{eq:localized-energy-identity}
\end{align}
holds.  In particular, after dropping the last non-positive term, one
obtains \eqref{eq:localized-flow-inequality}.
\end{prop}

\begin{proof}
We divide the proof into four steps.

\smallskip
\noindent
\emph{Step 1: passage to the limit in the local energy identity.}
By Lemma~\ref{lem:diffuse-balances},
\[
 \partial_t e_j
 =
 \operatorname{div}_x q_j-\alpha_j,
\]
where
\[
 e_j
 =
 \frac{\eps_j}{2}|\nabla u_j|^2
 +\frac1{\eps_j}V(u_j),
 \qquad
 q_{j,i}
 =
 \eps_j\,\partial_tu_j\cdot\partial_i u_j,
 \qquad
 \alpha_j
 =
 \eps_j|\partial_tu_j|^2.
\]
Since \(\nu_j^t=e_j(\cdot,t)\,\dd x\), let
\(\phi\in C^1(\T^2\times[0,T])\) be non-negative.  Multiplying the
local energy identity by \(\phi\), integrating over \(\T^2\), and using
that \(\T^2\) has no boundary, we obtain
\[
\begin{aligned}
 \frac{\dd}{\dd t}
 \int_{\T^2}\phi(\cdot,t)\,\dd\nu_j^t
 &=
 \int_{\T^2}\partial_t\phi\,\dd\nu_j^t
 -\int_{\T^2}\nabla\phi\cdot\dd q_j
 -\int_{\T^2}\phi\,\dd\alpha_j^t .
\end{aligned}
\]
After integration in time,
\begin{align}
 &\int_{\T^2}\phi(\cdot,t_2)\,\dd\nu_j^{t_2}
 -\int_{\T^2}\phi(\cdot,t_1)\,\dd\nu_j^{t_1}
 \notag\\
 &=
 \int_{\T^2\times(t_1,t_2)}
 \partial_t\phi\,\dd(\nu_j^t\dd t)
 -\int_{\T^2\times(t_1,t_2)}
 \nabla\phi\cdot\dd q_j
 -\int_{\T^2\times(t_1,t_2)}
 \phi\,\dd\alpha_j .
 \label{eq:diffuse-tested-energy}
\end{align}

Choose \(t_1,t_2\) in the common full-measure set of slice-convergence
times, outside the countable discontinuity sets of all scalar distribution
functions generated by a countable dense family of spatial test functions.
Then
\[
 \nu_j^{t_i}\stackrel{*}{\rightharpoonup}\nu_{t_i},
 \qquad i=1,2.
\]
The passage below is first performed for product test functions and then
extended to every \(\phi\in C^1(\T^2\times[0,T])\) by uniform
approximation and the uniform mass bounds.
Since \(x\mapsto\phi(x,t_i)\) is continuous,
\[
 \int_{\T^2}\phi(\cdot,t_i)\,\dd\nu_j^{t_i}
 \longrightarrow
 \int_{\T^2}\phi(\cdot,t_i)\,\dd\nu_{t_i}.
\]
The spacetime weak-star convergences
\[
 \nu_j^t\dd t\stackrel{*}{\rightharpoonup}\nu_t\dd t,
 \qquad
 q_j\stackrel{*}{\rightharpoonup}q,
 \qquad
 \alpha_j\stackrel{*}{\rightharpoonup}\alpha
\]
allow us to pass to the limit in the three spacetime terms.  The
characteristic function of \((t_1,t_2)\) may be approximated by continuous
time cutoffs; the choice of \(t_1,t_2\) as continuity points ensures that
the approximation error tends to zero.  Hence
\begin{align}
 &\int_{\T^2}\phi(\cdot,t_2)\,\dd\nu_{t_2}
 -\int_{\T^2}\phi(\cdot,t_1)\,\dd\nu_{t_1}
 \notag\\
 &=
 \int_{\T^2\times(t_1,t_2)}
 \partial_t\phi\,\dd(\nu_t\dd t)
 -\int_{\T^2\times(t_1,t_2)}
 \nabla\phi\cdot\dd q
 -\int_{\T^2\times(t_1,t_2)}
 \phi\,\dd\alpha .
 \label{eq:limit-tested-energy}
\end{align}

\smallskip
\noindent
\emph{Step 2: decomposition of the dissipation and positivity of the
defect.}
Recall that
\[
 q_t=\mathbf q_t\,\lambda_t,
 \qquad
 \lambda_t=\mathcal H^1\llcorner S_t,
\]
and decompose
\[
 \mathbf q_t=q_{n,t}n_t+q_{\tau,t}\tau_t.
\]
By the definitions of the mobility fields,
\[
 \mathcal V_t=-\frac{q_{n,t}}{e_t}n_t,
 \qquad
 \mathcal W_t=-\frac{q_{\tau,t}}{\rho_t}\tau_t.
\]
Therefore
\[
 e_t|\mathcal V_t|^2
 =
 \frac{|q_{n,t}|^2}{e_t},
 \qquad
 \rho_t|\mathcal W_t|^2
 =
 \frac{|q_{\tau,t}|^2}{\rho_t},
\]
with the convention on \(\{\rho_t=0\}\) from
Proposition~\ref{prop:sharp-action}.  Set
\[
 D_t
 :=
 e_t|\mathcal V_t|^2+\rho_t|\mathcal W_t|^2
 =
 \frac{|q_{n,t}|^2}{e_t}
 +\frac{|q_{\tau,t}|^2}{\rho_t}.
\]
The sharp action bound gives
\[
 a\geq D_t
 \qquad\text{for }\lambda_t\dd t\text{-almost every }(x,t).
\]
Using the Lebesgue decomposition
\[
 \alpha=a\,\lambda_t\dd t+\alpha^{\rm s},
 \qquad
 \alpha^{\rm s}\geq0,
\]
define
\[
 \gamma
 :=
 \alpha-D_t\,\lambda_t\dd t.
\]
Then
\begin{equation}\label{eq:gamma-explicit}
 \gamma
 =
 (a-D_t)\lambda_t\dd t+\alpha^{\rm s}\geq0.
\end{equation}
Thus \(\gamma\) is the non-negative dissipation defect, consisting of both
the excess absolutely continuous dissipation and the singular part
\(\alpha^{\rm s}\).

\smallskip
\noindent
\emph{Step 3: identification of the limiting flux.}
The flux decomposition and the two mobility laws give
\[
 -\mathbf q_t
 =
 e_t\mathcal V_t+\rho_t\mathcal W_t
 =
 2\Theta_tH_{S_t}
 +2\nabla_{S_t}\Theta_t.
 \label{eq:flux-density-identification}
\]
Let
\[
 \nabla_{S_t}\phi
 :=
 (\tau_t\otimes\tau_t)\nabla\phi
\]
be the tangential gradient of \(\phi\) along \(S_t\).  Taking the scalar
product of \eqref{eq:flux-density-identification} with \(\nabla\phi\), we
obtain
\[
 -\nabla\phi\cdot\mathbf q_t
 =
 2\Theta_tH_{S_t}\cdot\nabla\phi
 +2\nabla_{S_t}\Theta_t\cdot\nabla\phi.
\]
Since \(\nabla_{S_t}\Theta_t\) is tangential, while
\[
 \nabla\phi
 =
 \nabla_{S_t}\phi
 +(\nabla\phi\cdot n_t)n_t,
\]
the normal component of \(\nabla\phi\) is orthogonal to
\(\nabla_{S_t}\Theta_t\).  Consequently,
\[
 \nabla_{S_t}\Theta_t\cdot\nabla\phi
 =
 \nabla_{S_t}\Theta_t\cdot\nabla_{S_t}\phi,
\]
and hence
\begin{equation}\label{eq:identified-flux-term}
 -\nabla\phi\cdot\mathbf q_t
 =
 2\Theta_tH_{S_t}\cdot\nabla\phi
 +2\nabla_{S_t}\Theta_t\cdot\nabla_{S_t}\phi.
\end{equation}

\smallskip
\noindent
\emph{Step 4: identification of the dissipation terms and conclusion.}
From the mobility identities,
\[
 e_t\mathcal V_t=2\Theta_tH_{S_t},
 \qquad
 \rho_t\mathcal W_t=2\nabla_{S_t}\Theta_t,
\]
we obtain
\begin{align}
 e_t|\mathcal V_t|^2
 &=
 \frac{4\Theta_t^2}{e_t}|H_{S_t}|^2,
 \label{eq:normal-dissipation}\\
 \rho_t|\mathcal W_t|^2
 &=
 \frac{4}{\rho_t}
 |\nabla_{S_t}\Theta_t|^2.
 \label{eq:tangential-dissipation}
\end{align}
On \(\{\rho_t=0\}\), finite action implies
\(q_{\tau,t}=0\), and the mobility law implies
\(\nabla_{S_t}\Theta_t=0\); the second term is therefore understood as
zero.

Substituting
\[
 \nu_t=e_t\lambda_t,
 \qquad
 \alpha=D_t\lambda_t\dd t+\gamma,
\]
together with \eqref{eq:identified-flux-term},
\eqref{eq:normal-dissipation}, and
\eqref{eq:tangential-dissipation}, into
\eqref{eq:limit-tested-energy}, we obtain
\eqref{eq:localized-energy-identity}.  Since
\(\phi\geq0\) and \(\gamma\geq0\), the defect term satisfies
\[
 -\int_{\T^2\times(t_1,t_2)}\phi\,\dd\gamma\leq0.
\]
Dropping this non-positive term gives
\eqref{eq:localized-flow-inequality}.
\end{proof}

\begin{rem}[The scalar case and the absence of tangential microstructure]
\label{rem:no-tangential-microstructure}
When \(\rho_t=0\), the sharp action bound implies
\[
 q_{\tau,t}=0.
\]
The tangential mobility law then gives
\[
 \nabla_{S_t}\Theta_t=0,
\]
so that the interfacial weight is constant along each regular branch.  Since
\[
 e_t=2\Theta_t,
\]
the normal mobility law becomes
\[
 \mathcal V_t=H_{S_t},
\]
and there is no tangential internal-mobility channel.

In particular, in the scalar equipartition regime, with
\[
 \nu_t=e_t\,\mathcal H^1\llcorner S_t,
 \qquad
 H_t:=H_{S_t},
\]
the localized energy inequality takes the classical Brakke form
\[
 \int_{\T^2}\phi(\cdot,t_2)\,\dd\nu_{t_2}
 -
 \int_{\T^2}\phi(\cdot,t_1)\,\dd\nu_{t_1}
 \leq
 \int_{t_1}^{t_2}\int_{\T^2}
 \left(
 \partial_t\phi
 +\nabla\phi\cdot H_t
 -\phi|H_t|^2
 \right)
 \,\dd\nu_t\dd t,
\]
for every non-negative
\(\phi\in C^1(\T^2\times[0,T])\).  This is the Brakke inequality
arising in the scalar Allen--Cahn limit, as in
Ilmanen~\cite{Ilmanen1993}.
\end{rem}

\subsection{Balance at finite junctions}

The stress identity also determines the balance law at a junction.  The
argument is local and uses the fact that the momentum measure has no atom
at the junction point.

\begin{prop}[Finite junction balance]
\label{prop:junction-balance}
Fix a good time \(t\).  Suppose that, in a neighborhood \(U\) of
\(p\in S_t\), the set \(S_t\) is the union of finitely many embedded
\(C^1\) half-arcs
\[
 \Gamma_{\ell,t}=\gamma_{\ell,t}([0,r_\ell)),
 \qquad
 \gamma_{\ell,t}(0)=p,
\]
which meet only at \(p\).  Orient each branch away from \(p\), and let
\[
 \tau_{\ell,t}
 :=\lim_{s\downarrow0}\gamma_{\ell,t}'(s)
\]
be its outward unit tangent.  Assume that \(\Theta_t\) has a positive
one-sided trace
\[
 \Theta_{\ell,t}
 :=\lim_{s\downarrow0}
 \Theta_t(\gamma_{\ell,t}(s))>0
\]
on each branch.  Then
\begin{equation}\label{eq:junction-balance}
 \sum_\ell\Theta_{\ell,t}\tau_{\ell,t}=0.
\end{equation}
\end{prop}

\begin{proof}
Because
\[
 q_t=\mathbf q_t\lambda_t
\]
and \(\lambda_t=\cH^1\llcorner S_t\) has no atoms, we have
\[
 q_t(\{p\})=0.
\]
On the interior of the \(\ell\)-th branch, set
\[
 g_{\ell,t}(s)
 :=
 \Theta_t(\gamma_{\ell,t}(s))
 \gamma_{\ell,t}'(s).
\]
The one-dimensional form of the stress identity, already obtained in
\eqref{eq:weighted-tangent-derivative}, is
\begin{equation}\label{eq:branch-momentum-identity}
 \partial_sg_{\ell,t}
 =
 -\frac12
 \mathbf q_t(\gamma_{\ell,t}(s))
\end{equation}
in distributions.  In particular, \(g_{\ell,t}\) has a
\(W^{1,1}_{\rm loc}\) representative on each half-arc and its one-sided
trace at \(s=0\) is
\[
 g_{\ell,t}(0+)
 =\Theta_{\ell,t}\tau_{\ell,t}.
\]

Let \(X\in C_c^1(U;\R^2)\).  Using the representation of the stress on the
branches and integrating by parts from the junction outward, we obtain
\begin{align}
 \int_U T_t:DX
 &=
 2\sum_\ell
 \int_0^{r_\ell}
 g_{\ell,t}(s)\cdot
 \partial_s\bigl(X\circ\gamma_{\ell,t}\bigr)(s)\dd s
 \notag\\
 &=
 -2\sum_\ell
 g_{\ell,t}(0+)\cdot X(p)
 -2\sum_\ell
 \int_0^{r_\ell}
 \partial_sg_{\ell,t}(s)\cdot
 X(\gamma_{\ell,t}(s))\dd s.
 \label{eq:junction-integration-by-parts}
\end{align}
There is no contribution from the outer endpoints because \(X\) is
compactly supported in \(U\).  By
\eqref{eq:branch-momentum-identity}, the second term on the right-hand side
of \eqref{eq:junction-integration-by-parts} equals
\[
 \int_U X\cdot\dd q_t.
\]
On the other hand, the global stress identity
\(\diver T_t=-q_t\) gives directly
\[
 \int_U T_t:DX=\int_U X\cdot\dd q_t.
\]
Comparing the two formulas yields
\[
 \sum_\ell
 g_{\ell,t}(0+)\cdot X(p)=0.
\]
Since \(X(p)\in\R^2\) is arbitrary,
\[
 \sum_\ell g_{\ell,t}(0+)=0.
\]
Substituting
\(g_{\ell,t}(0+)=\Theta_{\ell,t}\tau_{\ell,t}\)
proves \eqref{eq:junction-balance}.
\end{proof}

\subsection{Completion of the proof of the main theorem}

\begin{proof}[Proof of Theorem~\ref{thm:main}]
The proof is assembled as follows.

Lemma~\ref{lem:phase-compactness} gives the spacetime convergence to a
Caccioppoli partition.  Lemma~\ref{lem:energy-time-compactness} and
Proposition~\ref{prop:spacetime-disintegration} give the time-slice and
spacetime measure compactness, together with the disintegrations of the
potential and spatial-gradient measures.

Proposition~\ref{prop:time-slice-rectifiability} gives the rectifiability of
the limiting energy measure for almost every time.  Proposition~
\ref{prop:parabolic-bethuel-relations} identifies the limiting stress,
proves the Bethuel relations, and yields
\[
 \mu_t
 =
 \rho_t\tau_t\otimes\tau_t\,\lambda_t
 +e_tn_t\otimes n_t\,\lambda_t.
\]

Proposition~\ref{prop:sharp-action} applies positivity of the spacetime
Gram matrix to this anisotropic spatial-gradient measure and proves the
sharp action lower bound \eqref{eq:sharp-action-bound}.
Proposition~\ref{prop:mobility-identities} identifies the limiting momentum
with the weighted first variation, establishes the curvature splitting,
and proves the two mobility laws \eqref{eq:two-mobility-law}.
Finally, Proposition~\ref{prop:localized-energy} passes the exact diffuse
local energy identity to the limit and proves both the sharper identity
with the defect measure \(\gamma\) and the localized inequality
\eqref{eq:localized-flow-inequality}.

These conclusions establish all assertions of
Theorem~\ref{thm:main}.
\end{proof}
\section*{Acknowledgments}

This work is partially supported by the National Key R\&D Program of China
under Grant 2023YFA1008801.

The authors acknowledge the use of AI tools. All mathematical arguments and
proofs in the final manuscript were checked and written by the authors.

\nocite{*}
\bibliographystyle{amsplain}
\bibliography{bib}

\end{document}